\documentclass[11pt]{article}
\usepackage{graphicx}
\usepackage{amssymb}
\usepackage{amsfonts}
\usepackage{amsmath}
\usepackage{hyperref}

\newtheorem{theorem}{Theorem}[section]
\newtheorem{proposition}[theorem]{Proposition}
\newtheorem{lemma}[theorem]{Lemma}
\newtheorem{corollary}[theorem]{Corollary}

\newtheorem{definition}[theorem]{Definition}
\newtheorem{example}[theorem]{Example}
\newtheorem{remark}[theorem]{Remark}

\newenvironment{proof}[1][Proof]{\textbf{#1.} }{\ \rule{0.5em}{0.5em}}

\renewcommand{\theequation}{\thesection.\arabic{equation}}

\def\RM{\rm}

\def\limfunc#1{\mathop{\mathrm{#1}}}
\def\func#1{\mathop{\mathrm{#1}}\nolimits}

\def\dsum{\displaystyle\sum}

\def\enddoc{\end{document}}

\def\Qlb#1{#1}

\def\Qcb#1{#1} 

\def\FRAME#1#2#3#4#5#6#7#8
{
 \begin{figure}[ptbh]
 \begin{center}
 \includegraphics[height=#3]{#7}
 \caption{#5}
 \label{#6}
 \end{center}
 \end{figure}
}

\begin{document}

\author{Alexander Grigor'yan\thanks{%
Partially supported by SFB 701 of German Research Council} \\
Department of Mathematics\\
University of Bielefeld\\
33501 Bielefeld, Germany \and Yong Lin\thanks{%
Partially supported by the Fundamental Research Funds for the Central Universities and
the Research Funds of (11XNI004), and National Natural Science Foundation of China, Grant
No. 11271011}\\
Department of Mathematics \\
Renmin University of China \\
Beijing, China \and Yuri Muranov\thanks{%
Partially supported by the CONACyT Grants 98697 and 151338, SFB 701 of
German Research Council, and the travel grant of the Commission for
Developing Countries of the International Mathematical Union} \\
Department of Mathematics \\
University of Warmia and Mazury\\
Olsztyn, Poland \and Shing-Tung Yau \thanks{%
Partially supported by the grant "Geometry and Topology of Complex
Networks", no. FA-9550-13-1-0097} \\
Department of Mathematics\\
Harvard University\\
Cambridge MA 02138, USA}
\title{Homotopy theory for digraphs}
\date{June 2014}
\maketitle

\begin{abstract}
We introduce a homotopy theory of digraphs (directed graphs) and prove its
basic properties, including the relations to the homology theory of digraphs
constructed by the authors in previous papers. In particular, we prove the
homotopy invariance of homologies of digraphs and the relation between the
fundamental group of the digraph and its first homology group.

The category of (undirected) graphs can be identified by a natural way with
a full subcategory of digraphs. Thus we obtain also consistent homology and
homotopy theories for graphs. Note that the homotopy theory for graphs
coincides with the one constructed in \cite{Babson} and \cite{Barcelo}.
\end{abstract}

\tableofcontents

\section{Introduction}

\label{S1}

The homology theory of digraphs has been constructed in a series of the
previous papers of the authors (see, for example, \cite{Mi2}, \cite{Mi3},
\cite{Mi2012}). In the present paper we introduce a homotopy theory of
digraphs and prove that there are natural relations to aforementioned
homology theory. In particular, we prove the invariance of the homology
theory under homotopy and the relation between the fundamental group and the
first homology group, which is similar to the one in the classical algebraic
topology. Let emphasize, that the theories of homology and homotopy of
digraphs are introduced entirely independent each other, but nevertheless
they exhibit a very tight connection similarly to the classical algebraic
topology.

The homotopy theory of undirected graphs was constructed by Babson, Barcelo,
Kramer, Laubenbacher, Longueville and Weaver in \cite{Babson} and \cite%
{Barcelo}. We identify in a natural way the category of graphs with a full
subcategory of digraphs, which allows us to transfer the homology and
homotopy theories to undirected graphs. The homotopy theory of graphs,
obtained in this way, coincides with the homotopy theory constructed in \cite%
{Babson} and \cite{Barcelo}. However, our notion of homology of graphs is
new, and the result about homotopy invariance of homologies of graphs is
also new. Hence, our results give an answer to the question raised in \cite%
{Babson} asking \textquotedblleft for a homology theory associated to the
A-theory of a graph\textquotedblright .

There are other homology theories on graphs that try to mimic the classical
singular homology theory. In those theories one uses predefined
\textquotedblleft small\textquotedblright\ graphs as basic cells and defines
singular chains as formal sums of the maps of the basic cell into the graph
(see, for example, \cite{Ivashchenko}, \cite{Talbi}\textbf{).} However,
simple examples show that the homology groups obtained in this way, depend
essentially on the choice of the basic cells.

Our homology theory of digraphs (and graphs) is very different from the
"singular" homology theories. We do not use predefined cells but formulate
only the desired properties of the cells in terms of the digraph (graph)
structure. Namely, each cell is determined by a sequence of vertices that
goes along the edges (allowed paths), and the boundary of the cell must also
be of this type. This homology theory has very clear algebraic \cite{Mi2}
and geometric \cite{Mi3}, \cite{Mi2012} motivation. It provides effective
computational methods for digraph (graph) homology, agrees with the homotopy
theory, and provides good connections with homology theories of simplicial
and cubical complexes \cite{Mi3} and, in particular, with homology of
triangulated manifolds.

Let us briefly describe the structure of the paper and the main results. In
Section \ref{S2} we give a short survey on homology theory for digraphs
following \cite{Mi2012}, \cite{Mi3}.

In Section \ref{S3}, we introduce the notion of homotopy of digraphs. We
prove the homotopy invariance of homology groups (Theorem \ref{t3.4}) and
give a number of examples based on the notion of deformation retraction.

In Section \ref{S4}, we define a fundamental group $\pi _{1}$ of digraph.
Elements of $\pi _{1}$ are equivalence classes of loops on digraphs, where
the equivalence of the loops is defined using a new notion of $C$-homotopy,
which is more general than a homotopy. A description of $C$-homotopy in
terms of local transformations of loops is given in Theorem \ref{p4.12}.

We prove the homotopy invariance of $\pi _{1}$ (Theorem \ref{TGHpi}) and the
relation $H_{1}=\pi _{1}/\left[ \pi _{1},\pi _{1}\right] $ between the first
homology group over $\mathbb{Z}$ and the fundamental group (Theorem \ref%
{t4.13}). We define higher homotopy groups by induction using the notion of
a loop digraph.

In Section \ref{SecSperner} we give a new proof of the classical Sperner
lemma, using fundamental groups of digraphs. We hope that our notions of
homotopy and homology theories on digraph can find further applications in
graph theory, in particular, in graph coloring.

In Section \ref{S5} we construct isomorphism between the category of
(undirected) graphs and a full subcategory of digraphs, thus transferring
the aforementioned results from the category of digraphs to the category of
graphs.

\section{Homology theory of digraphs}

\setcounter{equation}{0}\label{S2}In this Section we state the basic notions
of homology theory for digraphs in the form that we need in subsequent
sections. This is a slight adaptation of a more general theory from \cite%
{Mi2012}, \cite{Mi3}.

\subsection{The notion of a digraph}

\label{S21}We start with some definitions.

\begin{definition}
\label{d2.1}\RM A\emph{\ directed graph (digraph) } $G=(V,E)$ is a couple of
a set $V$, whose elements are called the \emph{vertices}, and a subset $%
E\subset \{V\times V\setminus {\func{diag}}\}$ of ordered pairs of vertices
that are called (directed) \emph{edges} or \emph{arrows}. The fact that $%
(v,w)\in E$ is also denoted by $v\rightarrow w$.
\end{definition}

In particular, a digraph has no edges $v\rightarrow v$ and, hence, it is a
combinatorial digraph in the sense of \cite{Ribenboim}. We write
\begin{equation*}
v\,\overrightarrow{=}w
\end{equation*}%
if either $v=w$ or $v\rightarrow w$. In this paper we consider only finite
digraphs, that is, digraphs with a finite set of vertices.

\begin{definition}
\label{d2.2} \RM A\emph{\ morphism }from a digraph\emph{\ }$G=\left(
V_{G},E_{G}\right) $\emph{\ }to a digraph\emph{\ }$H=\left(
V_{H},E_{H}\right) $ is a map $f\colon V_{G}\rightarrow V_{H}$ such that for
any edge $v\rightarrow w$ on $G$ we have $f\left( v\right) $\thinspace $%
\overrightarrow{=}f\left( w\right) $ on $H$ (that is, either $%
f(v)\rightarrow f(w)$ or $f(v)=f(w)$). We will refer to such morphisms also
as \emph{digraphs maps }(sometimes simply \emph{maps}) and denote them
shortly by $f:G\rightarrow H.$
\end{definition}

The set of all digraphs with digraphs maps form a \emph{category of digraphs}
that will be denoted by $\mathcal{D}$.

\begin{definition}
\label{d2.3}\RM For two digraphs $G=(V_{G},E_{G})$ and $H=(V_{H},E_{H})$
define the \emph{Cartesian product} $G\boxdot H$ as a digraph with the set
of vertices $V_{G}\times V_{H}$ and with the set of edges as follows: for $%
x,x^{\prime }\in V_{G}\ $and$\ \ y,y^{\prime }\in V_{H}$, we have $%
(x,y)\rightarrow (x^{\prime },y^{\prime })\ $in $G\boxdot H$ if and only if
\begin{equation*}
\text{either}\ \ x^{\prime }=x\text{ and}\ y\rightarrow y^{\prime }\text{,}\
\ \text{or}\ \ x\rightarrow x^{\prime }\ \text{and}\ \ y=y^{\prime },
\end{equation*}%
as is shown on the following diagram:%
\begin{equation*}
\begin{array}{cccccc}
y^{\prime }\bullet & \dots & \overset{\left( x,y^{\prime }\right) }{\bullet }
& \longrightarrow & \overset{\left( x^{\prime },y^{\prime }\right) }{\bullet
} & \dots \\
\ \ \ \ \uparrow \  &  & \uparrow &  & \uparrow &  \\
y\bullet & \dots & \overset{\left( x,y\right) }{\bullet } & \longrightarrow
& \overset{\left( x^{\prime },y\right) }{\bullet } & \dots \\
\ \ \ \ \ \ \ \ \ \  &  &  &  &  &  \\
^{H}\ \diagup \ _{G} & \dots & \underset{x}{\bullet } & \longrightarrow &
\underset{x^{\prime }}{\bullet }\  & \dots%
\end{array}%
\end{equation*}
\end{definition}

\subsection{Paths and their boundaries}

Let $V$ be a finite set. For any $p\geq 0$, an \emph{elementary} $p$-\emph{%
path} is any (ordered) sequence $i_{0},...,i_{p}$ of $p+1$ vertices of $V$
that will be denoted simply by $i_{0}...i_{p}$ or by $e_{i_{0}...i_{p}}$.
Fix a commutative ring $\mathbb{K}$ with unity and denote by $\Lambda
_{p}=\Lambda _{p}\left( V\right) =\Lambda _{p}\left( V,\mathbb{K}\right) $
the free $\mathbb{K}$-module that consist of all formal $\mathbb{K}$-linear
combinations of all elementary $p$-paths. Hence, each $p$-path has a form%
\begin{equation*}
v=\sum_{i_{0}i_{1}...i_{p}}v^{_{i_{0}i_{1}...i_{p}}}e_{i_{0}i_{1}...i_{p}},\
\ \text{where}\ \ v^{_{i_{0}i_{1}...i_{p}}}\in \mathbb{K}.
\end{equation*}

\begin{definition}
\label{d2.4} \RM Define for any $p\geq 0$ the \emph{boundary operator} $%
\partial :\Lambda _{p+1}\rightarrow \Lambda _{p}$ by%
\begin{equation}
\left( \partial v\right)
^{i_{0}...i_{p}}=\dsum\limits_{k}\dsum\limits_{q=0}^{p+1}\left( -1\right)
^{q}v^{i_{0}...i_{q-1}ki_{q}...i_{p}}  \label{dv}
\end{equation}%
where $1$ is the unity of $\mathbb{K}$ and the index $k$ is inserted so that
it is preceded by $q$ indices.
\end{definition}

Sometimes we need also the operator $\partial :\Lambda _{0}\rightarrow
\Lambda _{-1}$ where we set $\Lambda _{-1}=\left\{ 0\right\} \ $and $%
\partial v=0$ for all $v\in \Lambda _{0}$. It follows from (\ref{dv}) that%
\begin{equation}
\partial e_{j_{0}...j_{p+1}}=\dsum\limits_{q=0}^{p+1}\left( -1\right)
^{q}e_{j_{0}...\widehat{j_{q}}...j_{p+1}}.  \label{3.4}
\end{equation}%
%
%
%
%
%
%
%
%
%
%
%
%
%
%
%
%
%
%
%
%
%
%
%
%
%
%
%
%
%
%
%
%
%
%
%
%
%
%
%
%
%
%
%
%
%
%
%
%
%
%
%
%
%
%
%
%
%
%
%
%

It is easy to show that $\partial ^{2}v=0$ for any $v\in \Lambda _{p}$ (\cite%
{Mi2012}). Hence, the family of $\mathbb{K}$-modules $\left\{ \Lambda
_{p}\right\} _{p\geq -1}$ with the boundary operator $\partial $ determine a
chain complex that will be denoted by $\Lambda _{\ast }\left( V\right)
=\Lambda _{\ast }\left( V,\mathbb{K}\right) $.

\subsection{Regular paths}

\label{SecReg}

\begin{definition}
\label{d2.5} \RM An elementary $p$-path $e_{i_{0}...i_{p}}$ on a set $V$ is
called \emph{regular} if $i_{k}\neq i_{k+1}$ for all $k=0,...,p-1$, and
irregular otherwise.
\end{definition}

Let $I_{p}$ be the submodule of $\Lambda _{p}$ that is $\mathbb{K}$-spanned
by irregular $e_{i_{0}...i_{p}}.$ It is easy to verify that $\partial
I_{p}\subset I_{p-1}$ (cf. \cite{Mi2012}). Consider the quotient $\mathcal{R}%
_{p}:=\Lambda _{p}/I_{p}. $ Since $\partial I_{p}\subset I_{p-1}$, the
induced boundary operator
\begin{equation*}
\partial \colon \mathcal{R}_{p}\rightarrow \mathcal{R}_{p-1} \ \ (p\geq 0)
\end{equation*}
is well-defined. We denote by $\mathcal{R}_{\ast }\left( V\right) $ the
obtained chain complex. Clearly, $\mathcal{R}_{p}$ is linearly isomorphic to
the space of regular $p$-paths:
\begin{equation}
\mathcal{R}_{p}\cong \func{span}{}_{\mathbb{K}}\left\{
e_{i_{0}...i_{p}}:i_{0}...i_{p}\text{ is regular}\right\}  \label{Rp=}
\end{equation}%
For simplicity of notation, we will identify $\mathcal{R}_{p}$ with this
space, by setting all irregular $p$-paths to be equal to $0.$

Given a map $f\colon V\rightarrow V^{\prime } $ between two finite sets $V$
and $V^{\prime }$, define for any $p\geq 0$ the \emph{induced map}
\begin{equation*}
f_{\ast }\colon \Lambda _{p}(V)\rightarrow \Lambda _{p}(V^{\prime })
\end{equation*}%
by the rule $f_{\ast }\left( e_{i_{0}...i_{p}}\right)
=e_{f(i_{0})...f(i_{p})}, $ extended by $\mathbb{K}$-linearity to all
elements of $\Lambda _{p}\left( V\right) $. The map $f_{\ast }$ is a
morphism of chain complexes, because it trivially follows from (\ref{3.4})
that $\partial f_{\ast }=f_{\ast }\partial $. Clearly, if $e_{i_{0}...i_{p}}$
is irregular then $f_{\ast }\left( e_{i_{0}...i_{p}}\right) $ is also
irregular, so that
\begin{equation*}
f_{\ast }\left( I_{p}\left( V\right) \right) \subset I_{p}\left( V^{\prime
}\right) .
\end{equation*}%
Therefore, $f_{\ast }$ is well-defined on the quotient $\Lambda _{p}/I_{p}$
so that we obtain the induced map%
\begin{equation}
f_{\ast }:\mathcal{R}_{p}\left( V\right) \rightarrow \mathcal{R}_{p}\left(
V^{\prime }\right) .  \label{RpV}
\end{equation}%
Since $f_{\ast }$ still commutes with $\partial $, we see that the induced
map (\ref{RpV}) induces a morphism $\mathcal{R}_{\ast }(V)\rightarrow
\mathcal{R}_{\ast }(V^{\prime })$ of chain complexes. With identification (%
\ref{Rp=}) of $\mathcal{R}_{p}$ we have the following rule for the map (\ref%
{RpV}):%
\begin{equation}
f_{\ast }\left( e_{i_{0}...i_{p}}\right) =%
\begin{cases}
e_{f(i_{0})...f(i_{p})}, & \text{if }e_{f(i_{0})...f(i_{p})}\text{ is
regular,} \\
0, & \text{if }e_{f(i_{0})...f(i_{p})}\ \text{is irregular.}%
\end{cases}
\label{fe}
\end{equation}

\subsection{Allowed and $\partial $-invariant paths on digraphs}

\label{SecAllowed}

\begin{definition}
\label{d2.6} \RM Let $G=(V,E)$ be a digraph. An elementary $p$-path $%
i_{0}...i_{p}$ on $V$ is called \emph{allowed} if $i_{k}\rightarrow i_{k+1}$
for any $k=0,...,p-1$, and \emph{non-allowed }otherwise. The set of all
allowed elementary $p$-paths will be denoted by $E_{p}$.
\end{definition}

For example, $E_{0}=V$ and $E_{1}=E$. Clearly, all allowed paths are
regular. Denote by $\mathcal{A}_{p}=\mathcal{A}_{p}\left( G\right) $ the
submodule of $\mathcal{R}_{p}\left( G\right) :=\mathcal{R}_{p}\left(
V\right) $ spanned by the allowed elementary $p$-paths, that is,
\begin{equation}
\mathcal{A}_{p}=\func{span}{}_{\mathbb{K}}\left\{
e_{i_{0}...i_{p}}:i_{0}...i_{p}\in E_{p}\right\} .  \label{Padef}
\end{equation}%
The elements of $\mathcal{A}_{p}$ are called \emph{allowed} $p$-paths.

Note that the modules $\mathcal{A}_{p}$ of allowed paths are in general
\emph{not} invariant for $\partial $. Consider the following submodules of $%
\mathcal{A}_{p}$
\begin{equation}
\Omega _{p}\equiv \Omega _{p}\left( G\right) :=\left\{ v\in \mathcal{A}%
_{p}:\partial v\in \mathcal{A}_{p-1}\right\}  \label{PE=}
\end{equation}%
that are $\partial $-invariant. Indeed, $v\in \Omega _{p}$ implies $\partial
v\in \mathcal{A}_{p-1}$ and $\partial \left( \partial v\right) =0\in
\mathcal{A}_{p-2}$, whence $\partial v\in \Omega _{p-1}$. The elements of $%
\Omega _{p}$ are called $\partial $-\emph{invariant} $p$-paths.

Hence, we obtain a chain complex $\Omega _{\ast }=\Omega _{\ast }\left(
G\right) =\Omega _{\ast }\left( G,\mathbb{K}\right) $:
\begin{equation*}
\begin{array}{cccccccccccc}
0 & \leftarrow & \Omega _{0} & \overset{\partial }{\leftarrow } & \Omega _{1}
& \overset{\partial }{\leftarrow } & \dots & \overset{\partial }{\leftarrow }
& \Omega _{p-1} & \overset{\partial }{\leftarrow } & \Omega _{p} & \overset{%
\partial }{\leftarrow }\dots%
\end{array}%
\end{equation*}%
By construction we have $\Omega _{0}=\mathcal{A}_{0}$ and $\Omega _{1}=%
\mathcal{A}_{1}$, while in general $\Omega _{p}\subset \mathcal{A}_{p}$.

Let us define for any $p\geq 0$ the homologies of the digraph $G$ with
coefficients from $\mathbb{K}$ by
\begin{equation*}
H_{p}(G,\mathbb{K})=H_{p}\left( G\right) :=H_{p}\left( \Omega _{\ast }\left(
G\right) \right) =\left. \ker \partial |_{\Omega _{p}}\right/ \func{Im}%
\partial |_{\Omega _{p+1}}.
\end{equation*}%
Let us note that homology groups $H_{p}\left( G\right) $ (as well as the
modules $\Omega _{p}\left( G\right) $) can be computed directly by
definition using simple tools of linear algebra, in particular, those
implemented in modern computational software. On the other hand, some
theoretical tools for computation of homology groups like K\"{u}nneth
formulas were developed in \cite{Mi2012}.

\begin{example}
\label{e2.7} \RM Consider a digraph $G$ as on Fig. \ref{pic15a}.\FRAME{%
ftbphFU}{6.365in}{1.8395in}{0pt}{\Qcb{Planar digraph with a nontrivial
homology group $H_{2}$}}{\Qlb{pic15a}}{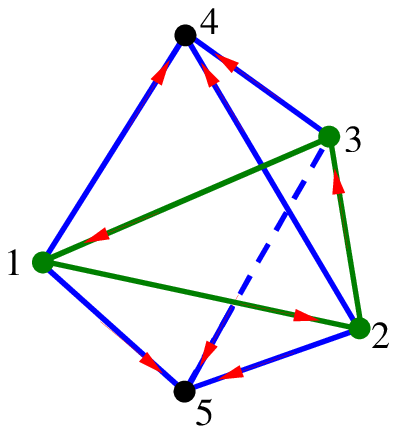}{\special{language
"Scientific Word";type "GRAPHIC";maintain-aspect-ratio TRUE;display
"USEDEF";valid_file "F";width 6.365in;height 1.8395in;depth
0pt;original-width 6.3027in;original-height 1.8023in;cropleft "0";croptop
"1";cropright "1";cropbottom "0";filename 'pic15a.eps';file-properties
"XNPEU";}}

A direct computation shows that $H_{1}\left( G,\mathbb{K}\right) =\left\{
0\right\} $ and $H_{2}\left( G,\mathbb{K}\right) \cong \mathbb{K}$, where $%
H_{2}\left( G\right) $ is generated by
\begin{equation*}
e_{124}+e_{234}+e_{314}-\left( e_{125}+e_{235}+e_{315}\right) .
\end{equation*}%
It is easy to see that $G$ is a planar graph but nevertheless its second
homology group is non-zero. This shows that the digraph homologies
\textquotedblleft see\textquotedblright\ some non-trivial intrinsic
dimensions of digraphs that are not necessarily related to embedding
properties.
\end{example}

\begin{example}
\label{e2.8}\RM Fix $n\geq 3$. Denote by $S_{n}$ a digraph with the vertex
set $V_{S_{n}}=\left\{ 0,...,n-1\right\} $ such that for any $i\in V_{S_{n}}$
either $i\rightarrow i+1$ or $i+1\rightarrow i$ (where $n$ and $0$ are
identified). We refer to $S_{n}$ as a \emph{cycle digraph}. Obviously, any
cycle digraph can be obtained from one of the line digraphs $I_{n}$ by
identifying the vertices $n$ and $0$.

The following $1$-path on $S_{n}$
\begin{equation}
\varpi =\sum_{\left\{ i\in S_{n}:i\rightarrow i+1\right\}
}e_{i(i+1)}-\sum_{\left\{ i\in S_{n}:i+1\rightarrow i\right\} }e_{(i+1)i}
\label{om}
\end{equation}%
lies in $\Omega _{1}(S_{n})$ and is closed. We will refer to $\varpi $ as a
\emph{standard} $1$-path on $S_{n}$. It is possible to show that $\varpi $
generates the space of all closed $1$-paths in $\Omega _{1}\left(
S_{n}\right) $, which is therefore one-dimensional. The homology group $%
H_{1}\left( S_{n},\mathbb{K}\right) $ is, hence, generated by the homology
class $\left[ \varpi \right] $, provided this class is non-trivial. One can
show that $\left[ \varpi \right] =0$ if and only if $S_{n}$ is isomorphic to
one of the following two digraphs:%
\begin{equation}
\text{a triangle }%
\begin{array}{c}
_{_{\nearrow }}\overset{1}{\bullet }_{_{\searrow }} \\
^{0}\bullet \ \rightarrow \ \bullet ^{2}%
\end{array}%
\ \ \text{or a square\ }%
\begin{array}{ccc}
^{1}\bullet & \longrightarrow & \bullet ^{2} \\
\ \uparrow &  & \uparrow \  \\
^{0}\bullet & \longrightarrow & \bullet ^{3}%
\end{array}%
,  \label{ts}
\end{equation}%
so that in this case $H_{1}\left( S_{n},\mathbb{K}\right) =\left\{ 0\right\}
$. In the case of triangle, $\varpi $ is the boundary of the $2$-path $%
e_{012}\in \Omega _{2}$, and, in the case of square, $\varpi $ is the
boundary of $e_{012}-e_{032}\in \Omega _{2}$.

If $S_{n}$ contains neither triangle nor square, then $\left[ \varpi \right]
$ is a generator of $H_{1}\left( S_{n},\mathbb{K}\right) \cong \mathbb{K}$.
\end{example}

\begin{proposition}
\label{pijk}Let $G$ be any finite digraph. Then any $\omega \in \Omega
_{2}\left( G,\mathbb{Z}\right) $ can be represented as a linear combination
of the $\partial $-invariant $2$-paths of following three types:

\begin{enumerate}
\item $e_{iji}$ with $i\rightarrow j\rightarrow i$ (a double edge in $G$);

\item $e_{ijk}$ with $i\rightarrow j\rightarrow k$ and $i\rightarrow k$ (a
triangle as a subgraph of $G$);

\item $e_{ijk}-e_{imk}$ with $i\rightarrow j\rightarrow k,\ i\rightarrow
m\rightarrow k,\ i\not\rightarrow k,$ $i\neq k$ (a square as a subgraph of $%
G $).
\end{enumerate}
\end{proposition}

\begin{proof}
Since the $2$-path $\omega $ is allowed, it can be represented as a sum of
elementary $2$-path $e_{ijk}$ with $i\rightarrow j\rightarrow k$ multiplied
with $+1$ or $-1$. If $k=i$ then $e_{ijk}$ is a double edge. If $i\neq k$
and $i\rightarrow k$ then $e_{ijk}$ is a triangle. Subtracting from $\omega $
all double edges and triangles, we can assume that $\omega $ has no such
terms any more. Then, for any term $e_{ijk}$ in $\omega $ we have $i\neq k$
and $i\not\rightarrow k$. Fix such a pair $i,k$ and consider any vertex $j$
with $i\rightarrow j\rightarrow k$. The $1$-path $\partial \omega $ is the
sum of $1$-paths of the form%
\begin{equation*}
\partial e_{ijk}=e_{ij}-e_{ik}+e_{jk}.
\end{equation*}%
Since $\partial \omega $ is allowed but $e_{ik}$ is not allowed, the term $%
e_{ik}$ should cancel out after we sum up all such terms over all possible $%
j $. Therefore, the number of $j$ such that $e_{ijk}$ enters $\omega $ with
coefficient $+1$ is equal to the number of $j$ such that $e_{ijk}$ enters in
$\omega $ with the coefficient $-1$. Combining the pair with $+1$ and $-1$
together, we obtain that $\omega $ is the sum of the terms of the third type
(squares).
\end{proof}

\begin{theorem}
\label{t2.9} Let $G$ and $G^{\prime }$ be two digraphs, and $f\colon
G\rightarrow G^{\prime }$ be a digraph map. Then the map $f_{\ast }|_{\Omega
_{p}(G)}$ (where $f_{\ast }$ is the induced map \emph{(\ref{RpV})}) provides
a morphism of chain complexes
\begin{equation*}
\Omega _{\ast }(G,\mathbb{K})\rightarrow \Omega _{\ast }(G^{\prime },\mathbb{%
K})
\end{equation*}%
and, consequently, a homomorphism of homology groups
\begin{equation*}
H_{\ast }(G,\mathbb{K})\rightarrow H_{\ast }(G^{\prime },\mathbb{K})
\end{equation*}%
that will also be denoted by $f_{\ast }$.
\end{theorem}

\begin{proof}
By construction $\Omega _{p}\left( G\right) $ is a submodule of $\mathcal{R}%
_{p}\left( G\right) $, and all we need to prove is that%
\begin{equation}
f_{\ast }\left( \Omega _{p}\left( G\right) \right) \subset \Omega _{p}\left(
G^{\prime }\right) .  \label{f*}
\end{equation}%
Let us first show that%
\begin{equation*}
f_{\ast }\left( \mathcal{A}_{p}\left( G\right) \right) \subset \mathcal{A}%
_{p}\left( G^{\prime }\right) .
\end{equation*}%
It suffices to prove that if $e_{i_{0}...i_{p}}$ is allowed on $G$ then $%
f_{\ast }\left( e_{i_{0}...i_{p}}\right) $ is allowed on $G^{\prime }$.
Indeed, if $e_{f\left( i_{0}\right) ...\left( i_{p}\right) }$ is irregular
then we have by (\ref{fe}) that $f_{\ast }\left( e_{i_{0}...i_{p}}\right)
=0\in \mathcal{A}_{p}\left( G^{\prime }\right) .$ If $e_{f\left(
i_{0}\right) ...\left( i_{p}\right) }$ is regular then $f\left( i_{k}\right)
\neq f\left( i_{k+1}\right) $ for all $k=0,...,p-1$. Since $i_{k}\rightarrow
i_{k+1}$ on $G$, by the definition of a digraph map we have either $f\left(
i_{k}\right) \rightarrow f\left( i_{k+1}\right) $ on $G^{\prime }$ or $%
f\left( i_{k}\right) =f\left( i_{k+1}\right) $. Since the second possibility
is excluded, we obtain $f\left( i_{k}\right) \rightarrow f\left(
i_{k+1}\right) $ for all $k$, whence it follows that $f_{\ast }\left(
e_{i_{0}...i_{p}}\right) =e_{f\left( i_{0}\right) ...\left( i_{p}\right) }$
is allowed on $G^{\prime }$.

Now we can prove (\ref{f*}). For any $v\in \Omega _{p}\left( G\right) $ we
have by (\ref{PE=}) $v\in \mathcal{A}_{p}\left( G\right) $ and $\partial
v\in \mathcal{A}_{p-1}\left( G\right) $, whence%
\begin{equation*}
f_{\ast }\left( v\right) \in \mathcal{A}_{p}\left( G^{\prime }\right) \
\text{and\ \ }\partial \left( f_{\ast }\left( v\right) \right) =f_{\ast
}\left( \partial v\right) \in \mathcal{A}_{p-1}\left( G^{\prime }\right) ,
\end{equation*}%
which implies $f_{\ast }\left( v\right) \in \Omega _{p}\left( G^{\prime
}\right) .$
\end{proof}

\label{SecPropertiesHomology}

\subsection{Cylinders}

\label{SecCyl}For any digraph $G$ consider its product $G\boxdot I$ with the
digraph $I=\left( ^{0}\bullet \rightarrow \bullet ^{1}\right) $ (see
Definition \ref{d2.3}).

\begin{definition}
\RM The digraph $G\boxdot I$ is called the \emph{cylinder} over $G$ and will
be denoted by $\limfunc{Cyl}G$ or by $\widehat{G}$.
\end{definition}

By the definition of Cartesian product, the set of vertices of $\widehat{G}$
is $\widehat{V}=V\times \left\{ 0,1\right\} $, and the set $\widehat{E}$ of
its edges is defined by the rule: $\left( x,a\right) \rightarrow \left(
y,b\right) $ if and only if either $x\rightarrow y$ in $G$ and $a=b$ or $x=y$
and $a\rightarrow b$ in $I$. We shall put the hat $\widehat{}$ over all
notation related to $\widehat{G}$, for example, $\widehat{\mathcal{R}}_{p}:=%
\mathcal{R}_{p}(\widehat{G})$ and $\widehat{\Omega }_{p}:=\Omega _{p}(%
\widehat{G})$. One can identify $\widehat{V}=V\times \left\{ 0,1\right\} $
with $V\sqcup V^{\prime }$ where $V^{\prime }$ is a copy of $V$, and use the
notation $\left( x,0\right) \equiv x$ and $\left( x,1\right) \equiv
x^{\prime }.$

Define the operation of \emph{lifting }paths from $G$ to $\widehat{G}$ as
follows. If $v=e_{i_{0}...i_{p}}$ then $\widehat{v}$ is a $\left( p+1\right)
$-path in $\widehat{G}$ defined by
\begin{equation}
\widehat{v}=\sum_{k=0}^{p}\left( -1\right) ^{k}e_{i_{0}...i_{k}i_{k}^{\prime
}...i_{p}^{\prime }}.  \label{elift}
\end{equation}%
By $\mathbb{K}$-linearity this definition extends to all $v\in \mathcal{R}%
_{p}$,$\,$ thus giving $\widehat{v}\in \widehat{\mathcal{R}}_{p+1}$. It
follows that, for any $v\in \mathcal{R}_{p}$ and any path $i_{0}...i_{p}$ on
$G$,%
\begin{equation}
\widehat{v}^{i_{0}...i_{k}i_{k}^{\prime }...i_{p}^{\prime }}=\left(
-1\right) ^{k}v^{i_{0}...i_{p}}.  \label{vlift}
\end{equation}%
Clearly, $i_{0}...i_{p}$ is allowed in $G$ if and only if $%
i_{0}...i_{k}i_{k}^{\prime }...i_{p}^{\prime }$ is allowed in $\widehat{G}$:%
\begin{equation*}
\begin{array}{ccccccccc}
&  & \cdots & \overset{i_{k}^{\prime }}{\bullet } & \longrightarrow &
\overset{i_{k+1}^{\prime }}{\bullet } & \longrightarrow & \cdots &
\longrightarrow \overset{i_{p}^{\prime }}{\bullet } \\
&  &  & \uparrow &  & \uparrow &  &  &  \\
\overset{i_{0}}{\bullet }\longrightarrow & \cdots & \longrightarrow &
\overset{i_{k}}{\bullet } & \longrightarrow & \overset{i_{k+1}}{\bullet } &
\cdots &  &
\end{array}%
,\
\end{equation*}%
for some/all $k$. Hence, we see that $v\in \mathcal{A}_{p}$ if and only if $%
\widehat{v}\in \widehat{\mathcal{A}}_{p+1}.$

\begin{proposition}
\label{p2.10}If $v\in \Omega _{p}$ then $\widehat{v}\in \widehat{\Omega }%
_{p+1}.$
\end{proposition}

\begin{proof}
We need to prove that if $v\in \mathcal{A}_{p}$ and $\partial v\in \mathcal{A%
}_{p-1}$ then $\partial \widehat{v}\in \widehat{\mathcal{A}}_{p}.$ Let us
prove first some properties of the lifting. For any path $v$ in $G$ define
its image $v^{\prime }$ in $G^{\prime }=\left( V^{\prime },E^{\prime
}\right) $ by%
\begin{equation*}
\left( e_{i_{0}...i_{p}}\right) ^{\prime }=e_{i_{0}^{\prime
}...i_{p}^{\prime }}.
\end{equation*}%
Let us show first that, for any $p$-path \thinspace $u$ and $q$-path $v$ on $%
G$, the following identity holds:%
\begin{equation}
\widehat{uv}=\widehat{u}v^{\prime }+\left( -1\right) ^{p+1}u\widehat{v}.
\label{hatp}
\end{equation}%
It suffices to prove it for $u=e_{i_{0}...i_{p}}$ and $v=e_{j_{0}...j_{q}}.$
Then $uv=e_{i_{0}...i_{p}j_{0}...j_{q}}$ and%
\begin{eqnarray*}
\widehat{uv} &=&\sum_{k=0}^{p}\left( -1\right)
^{k}e_{i_{0}...i_{k}i_{k}^{\prime }...i_{p}^{\prime }j_{0}^{\prime
}...j_{q}^{\prime }}+\sum_{k=0}^{q}\left( -1\right)
^{k+p+1}e_{i_{0}...i_{p}j_{0}...j_{k}j_{k}^{\prime }...j_{q}^{\prime }} \\
&=&\widehat{u}v^{\prime }+\left( -1\right) ^{p+1}u\widehat{v}.
\end{eqnarray*}%
Now let us show that, for any $p$-path $v$ with $p\geq 0$
\begin{equation}
\partial \widehat{v}=-\widehat{\partial v}+v^{\prime }-v.  \label{dhat}
\end{equation}%
It suffices to prove it for $v=e_{i_{0}...i_{p}}$, which will be done by
induction in $p$. For $p=0$ write $v=e_{a}$ so that $\partial v=0$ and $%
\widehat{v}=e_{aa^{\prime }}$ whence%
\begin{equation*}
\partial \widehat{v}=e_{a^{\prime }}-e_{a}=-\widehat{\partial v}+v^{\prime
}-v.
\end{equation*}%
For $p>1$ write $v=ue_{i_{p}}$ where $u=e_{i_{0}...i_{p-1}}$. Using (\ref%
{hatp}) and the inductive hypothesis with the $\left( p-1\right) $-path $u$
we obtain%
\begin{eqnarray*}
\partial \widehat{v} &=&\partial \left( \widehat{u}e_{i_{p}^{\prime
}}+\left( -1\right) ^{p}ue_{i_{p}i_{p}^{\prime }}\right) \\
&=&\left( \partial \widehat{u}\right) e_{i_{p}^{\prime }}+\left( -1\right)
^{p+1}\widehat{u}+\left( -1\right) ^{p}\left( \partial u\right)
e_{i_{p}i_{p}^{\prime }}+u\left( e_{i_{p}^{\prime }}-e_{i_{p}}\right) \\
&=&(-\widehat{\partial u}+u^{\prime }-u)e_{i_{p}^{\prime }}+\left( -1\right)
^{p+1}\widehat{u}+\left( -1\right) ^{p}\left( \partial u\right)
e_{i_{p}i_{p}^{\prime }}+ue_{i_{p}^{\prime }}-v \\
&=&-(\widehat{\partial u})e_{i_{p}^{\prime }}+v^{\prime }+\left( -1\right)
^{p+1}\widehat{u}+\left( -1\right) ^{p}\left( \partial u\right)
e_{i_{p}i_{p}^{\prime }}-v~.
\end{eqnarray*}%
On the other hand,%
\begin{equation*}
\widehat{\partial v}=\left( \left( \partial u\right) e_{i_{p}}+\left(
-1\right) ^{p}u\right) ^{\symbol{94}}=(\widehat{\partial u})e_{i_{p}^{\prime
}}+\left( -1\right) ^{p-1}\left( \partial u\right) e_{i_{p}i_{p}^{\prime
}}+\left( -1\right) ^{p}\widehat{u},
\end{equation*}%
whence it follows that $\partial \widehat{v}+\widehat{\partial v}=v^{\prime
}-v$, which finishes the proof of (\ref{dhat}).

Finally, if $v\in \mathcal{A}_{p}$ and $\partial v\in \mathcal{A}_{p-1}$
then $v^{\prime }$ and $\widehat{\partial v}$ belong to $\widehat{\mathcal{A}%
}_{p}$ whence it follows from (\ref{dhat}) also $\partial \widehat{v}\in
\widehat{\mathcal{A}}_{p}.$ This proves that $\widehat{v}\in \widehat{%
\mathcal{A}}_{p+1}.$
\end{proof}

\begin{example}
\label{e2.12} \RM The cylinder over the digraph $^{0}\bullet \rightarrow
\bullet ^{1}$ is a square
\begin{equation*}
\begin{array}{ccc}
^{2}\bullet & \longrightarrow & \bullet ^{3} \\
\ \uparrow &  & \uparrow \  \\
^{0}\bullet & \longrightarrow & \bullet ^{1}%
\end{array}%
\
\end{equation*}%
Lifting a $\partial $-invariant $1$-path $e_{01}\in \Omega _{1}$ we obtain a
$\partial $-invariant $2$-path on the square: $e_{00^{\prime }1^{\prime
}}-e_{011^{\prime }}$, that can be rewritten in the form $e_{023}-e_{013}.$

The cylinder over a square is a $3$-cube that is shown in Fig. \ref{pic9}.%
\FRAME{ftbphFU}{4.5152in}{1.7668in}{0pt}{\Qcb{$3$-cube}}{\Qlb{pic9}}{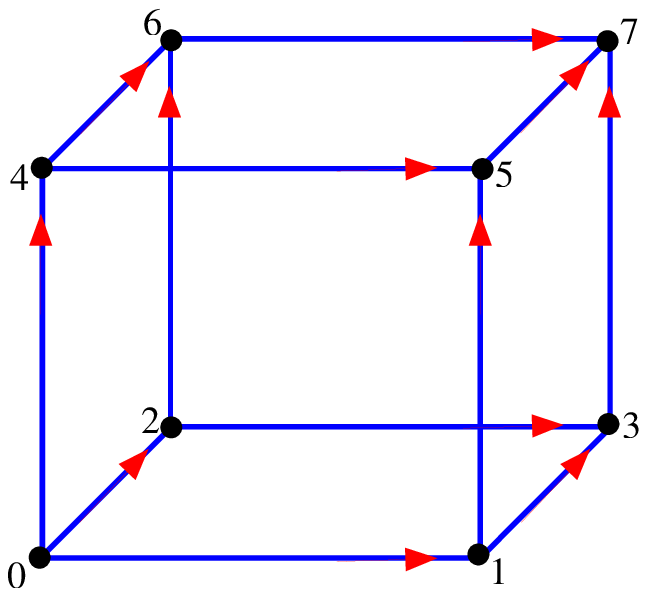%
}{\special{language "Scientific Word";type "GRAPHIC";maintain-aspect-ratio
TRUE;display "USEDEF";valid_file "F";width 4.5152in;height 1.7668in;depth
0pt;original-width 6.3027in;original-height 2.444in;cropleft "0";croptop
"1";cropright "1";cropbottom "0";filename 'pic9.eps';file-properties
"XNPEU";}}

Lifting the $2$-path $e_{023}-e_{013}$ we obtain a $\partial $-invariant $3$%
-path on the $3$-cube:%
\begin{equation*}
e_{0467}-e_{0267}+e_{0237}-e_{0457}+e_{0157}-e_{0137}.
\end{equation*}%
Defining further $n$-cube as the cylinder over $\left( n-1\right) $-cube, we
see that $n$-cube determines a $\partial $-invariant $n$-path that is a
lifting of a $\partial $-invariant $\left( n-1\right) $-path from $\left(
n-1\right) $-cube and that is an alternating sum of $n!$ elementary terms.
One can show that this $n$-path generates $\Omega _{n}$ on $n$-cube (see
\cite{Mi3}).
\end{example}

\section{Homotopy theory of digraphs}

\label{S3}\setcounter{equation}{0}In this Section we introduce a homotopy
theory of digraphs and establish the relations between this theory and the
homology theory of digraphs of \cite{Mi2012} and \cite{Mi3}.

\subsection{The notion of homotopy}

Fix $n\geq 0$. Denote by $I_{n}$ any digraph whose the set of vertices is $%
\{0,1,\dots ,n\}$ and the set of edges contains exactly one of the edges $%
i\rightarrow (i+1)$, $(i+1)\rightarrow i$ for any $i=0,1,\dots ,n-1$, and no
other edges. A digraph $I_{n}$ is called \textit{a }\emph{line digraph}.
Denote by $\mathcal{I}_{n}$ the set of all line digraphs $\mathcal{I}_{n}$
and by $\mathcal{I}$ the union of all $\mathcal{I}_{n}$.

Clearly, there is only one digraph in $\mathcal{I}_{0}$ -- the one-point
digraph. There are two digraphs in $\mathcal{I}_{1}$: the digraph $I$ with
the edge $(0\rightarrow 1)$ and the digraph $I^{-}$ with the edge $\left(
1\rightarrow 0\right) $.

\begin{definition}
\label{d3.1} \RM Let $G,H$ be two digraphs. Two digraph maps $f,g\colon
G\rightarrow H$ are\emph{\ }called \emph{homotopic} if there exists a line
digraph $I_{n}\in \mathcal{I}_{n}$ with $n\geq 1$ and a digraph map%
\begin{equation*}
F\colon G\boxdot I_{n}\rightarrow H
\end{equation*}%
such that
\begin{equation}
F|_{G\boxdot \{0\}}=f\ \ \text{and}\ \ F|_{G\boxdot \{n\}}=g.  \label{Ffg}
\end{equation}%
In this case we shall write $f\simeq g$. The map $F$ is called a \emph{%
homotopy} between $f$ and $g$.
\end{definition}

In the case $n=1$ we refer to the map $F$ as an \emph{one-step homotopy}
between $f$ and $g$. In this case the identities (\ref{Ffg}) determine $F$
uniquely, and the requirement is that the so defined $F$ is a digraph map of
$G\boxdot I_{1}$ to $H$. Since for $I_{1}$ there are only two choices $%
0\rightarrow 1$ and $0\leftarrow 1$, we obtain that $f$ and $g$ are one-step
homotopic, if
\begin{equation}
\text{either }f\left( x\right) \overrightarrow{=}g\left( x\right) \ \text{%
for\ all\ }x\in V_{G}\ \text{or\ \ }g\left( x\right) \overrightarrow{=}%
f\left( x\right) \ \text{for\ all\ }x\in V_{G}.  \label{f=g}
\end{equation}%
It follows that $f$ and $g$ are homotopic if there is a finite sequence of
digraph maps $f=f_{0},f_{1},...,f_{n}=g$ from $G$ to $H$ such that $f_{k}$
and $f_{k+1}$ are one-step homotopic. It is obvious that the relation "$%
\simeq $\textquotedblright\ is an equivalence relation on the set of all
digraph maps from $G$ to $H$.

\begin{definition}
\label{t5.3}\RM Two digraphs $G$ and $H$ are called \emph{homotopy equivalent%
} if there exist digraph maps
\begin{equation}
f:G\rightarrow H,\ \ \ g:H\rightarrow G  \label{fGH}
\end{equation}%
such that
\begin{equation}
f\circ g\simeq \func{id}_{H},\ \ \ \ \ g\circ f\simeq \func{id}_{G}.
\label{fg}
\end{equation}%
In this case we shall write $H\simeq G$. The maps $f$ and $g$ as in (\ref{fg}%
) are called \emph{\ homotopy inverses }of each other.
\end{definition}

A digraph $G$ is called \emph{contractible} if $G\simeq \left\{ \ast
\right\} $ where $\left\{ \ast \right\} $ is a single vertex digraph. It
follows from Definition \ref{t5.3} that a digraph $G$ is contractible if and
only if there is a digraph map $h:G\rightarrow G$ such that the image of $h$
consists of a single vertex and $h\simeq \func{id}_{G}.$ Examples of
contractible digraphs will be given in Section \ref{SecRet}.

\subsection{Homotopy preserves homologies}

Now we can prove the first result about connections between homotopy and
homology theories for digraphs.

\begin{theorem}
\label{t3.4} Let $G,H$ be two digraphs.

\begin{itemize}
\item[$\left( i\right) $] Let $f\simeq g\colon G\rightarrow H$ be two
homotopic digraph maps. Then these maps induce the identical homomorphisms
of homology groups of $G$ and $H$, that is, the maps%
\begin{equation*}
f_{\ast }:H_{p}\left( G\right) \rightarrow H_{p}\left( H\right) \ \ \text{%
and\ \ \ }g_{\ast }:H_{p}\left( G\right) \rightarrow H_{p}\left( H\right)
\end{equation*}%
are identical.

\item[$\left( ii\right) $] If the digraphs $G$ and $H$ are homotopy
equivalent, then they have isomorphic homology groups. Furthermore, if the
homotopical equivalence of $G$ and $H$ is provided by the digraph maps \emph{%
(\ref{fGH})} then their induced maps $f_{\ast }$ and $g_{\ast }$ provide
mutually inverse isomorphisms of the homology groups of $G$ and $H$.
\end{itemize}
\end{theorem}

In particular, if a digraph $G$ is contractible, then all the homology
groups of $G$ are trivial, except for $H_{0}$.

\begin{proof}
$\left( i\right) $ Let $F$ be a homotopy between $f$ and $g$ as in
Definition \ref{d3.1}. Consider first the case $n=1$ and let $I_{n}$ be the
digraph $I=\left( 0\rightarrow 1\right) $ (the case $I_{n}=I^{-}$ is
similar). The maps $f$ and $g$ induce morphisms of chain complexes
\begin{equation*}
f_{\ast },g_{\ast }\colon \Omega _{\ast }(G)\rightarrow \Omega _{\ast }(H),
\end{equation*}%
and $F$ induces a morphism
\begin{equation*}
F_{\ast }\colon \Omega _{\ast }(G\boxdot I)\rightarrow \Omega _{\ast }(H).
\end{equation*}%
Note that, for any path $v\in \Omega _{\ast }(G\boxdot I)$ that lies in $%
G\boxdot \left\{ 0\right\} $, we have $F_{\ast }\left( v\right) =f_{\ast
}\left( v\right) $, and for any path $v^{\prime }\in \Omega _{\ast
}(G\boxdot I)$ that lies in $G\boxdot \left\{ 1\right\} $, we have $F_{\ast
}\left( v^{\prime }\right) =g\left( v^{\prime }\right) .$

In order to prove that $f_{\ast }$ and $g_{\ast }$ induce the identical
homomorphisms $H_{\ast }\left( G\right) \rightarrow H_{\ast }\left( H\right)
$, it suffices by \cite[Theorem 2.1, p.40]{Maclane} to construct a chain
homotopy between the chain complexes $\Omega _{\ast }\left( G\right) $ and $%
\Omega _{\ast }\left( H\right) $, that is, the $\mathbb{K}$-linear mappings%
\begin{equation*}
L_{p}\colon \Omega _{p}(G)\rightarrow \Omega _{p+1}(H)
\end{equation*}%
such that%
\begin{equation*}
\partial L_{p}+L_{p-1}\partial =g_{\ast }-f_{\ast }
\end{equation*}%
(note that all the terms here are mapping from $\Omega _{p}\left( G\right) $
to $\Omega _{p}\left( H\right) $). Let us define the mapping $L_{p}$ as
follows%
\begin{equation*}
L_{p}(v)=F_{\ast }\left( \widehat{v}\right) ,
\end{equation*}%
for any $v\in \Omega _{p}\left( G\right) $, where $\widehat{v}\in \Omega
_{p+1}\left( G\boxdot I\right) $ is lifting of $v$ to the graph $\widehat{G}%
=G\boxdot I$ defined in Section \ref{SecCyl}. Using $\partial F_{\ast
}=F_{\ast }\partial $ (see Theorem \ref{t2.9}) and the product rule (\ref%
{dhat}), we obtain
\begin{eqnarray*}
(\partial L_{p}+L_{p-1}\partial )(v) &=&\partial (F_{\ast }(\widehat{v}%
))+F_{\ast }(\widehat{\partial v}) \\
&=&F_{\ast }\left( \partial \widehat{v}\right) +F_{\ast }(\widehat{\partial v%
}) \\
&=&F_{\ast }(\partial \widehat{v}+\widehat{\partial v}) \\
&=&F_{\ast }\left( v^{\prime }-v\right) \\
&=&g_{\ast }\left( v\right) -f_{\ast }\left( v\right) .
\end{eqnarray*}%
The case of an arbitrary $n$ follows then by induction.

$\left( ii\right) $ Let $f,g$ be the maps from Definition \ref{t5.3}. Then
they induce the following mappings%
\begin{equation*}
H_{p}\left( G\right) \overset{f_{\ast }}{\rightarrow }H_{p}\left( H\right)
\overset{g_{\ast }}{\rightarrow }H_{p}\left( G\right) \overset{f_{\ast }}{%
\rightarrow }H_{p}\left( H\right) .
\end{equation*}%
By $\left( i\right) $ and (\ref{fg}) we have $f_{\ast }\circ g_{\ast }=\func{%
id}$ and $g_{\ast }\circ f_{\ast }=\func{id}$, which implies that $f_{\ast }$
and $g_{\ast }$ are mutually inverse isomorphisms of $H_{p}\left( G\right) $
and $H_{p}\left( H\right) $.
\end{proof}

\subsection{Retraction}

\label{SecRet}A sub-digraph $H$ of a digraph $G$ is a digraph whose set of
vertices is a subset of that of $G$ and the edges of $H$ are all those edges
of $G$ whose adjacent vertices belong to $H$.

\begin{definition}
\label{d3.5} \RM Let $G$ be a digraph and $H$ be its sub-digraph.

$\left( i\right) $ A \emph{retraction} of $G$ onto $H$ is a digraph map $%
r\colon G\rightarrow H$ such that $r|_{H}=\func{id}_{H}$.

$\left( ii\right) $ A retraction $r:G\rightarrow H$ is called a \emph{%
deformation retraction} if $i\circ r\simeq \func{id}_{G},$ where $i\colon
H\rightarrow G$ is the natural inclusion map.
\end{definition}

\begin{proposition}
\label{p3.6} Let $r:G\rightarrow H$ be a deformation retraction. Then $%
G\simeq H$ and the maps $r,i$ are homotopy inverses.
\end{proposition}

\begin{proof}
By definition of retraction we have $r\circ i=\func{Id}_{H}$ and, in
particular $r\circ i\simeq \func{id}_{H}$. Since $i\circ r\simeq \func{id}%
_{G}$, we obtain by Definition \ref{t5.3} that $G\simeq H$.
\end{proof}

In general the existence of a deformation retraction $r:G\rightarrow H$ is a
stronger condition that the homotopy equivalence $G\simeq H$. However, in
the case when $H=\left\{ \ast \right\} $, the existence of a deformation
retraction $r:G\rightarrow \left\{ \ast \right\} $ is equivalent to the
contractibility of $G,$ which follows from the remark after Definition \ref%
{t5.3}.

The next two statements provide a convenient way of constructing a
deformation retraction.

\begin{proposition}
\label{c3.7}\label{Pr}Let $r:G\rightarrow H$ be a retraction of a digraph $G$
onto a sub-digraph \ $H$. Assume that there exists a finite sequence $%
\left\{ f_{k}\right\} _{k=0}^{n}$ of digraph maps $f_{k}:G\rightarrow G$
with the following properties:

\begin{enumerate}
\item $f_{0}=\func{id}_{G}$;

\item $f_{n}=i\circ r$ (where $i$ is the inclusion map $i:H\rightarrow G$),
that is, $f_{n}\left( v\right) =r\left( v\right) $ for all vertices $v$ of $%
G $;

\item for any $k=1,...,n$ either $f_{k-1}\left( x\right) $\thinspace $%
\overrightarrow{=}f_{k}\left( x\right) $ for all $x\in V_{G}$ or $%
f_{k}\left( x\right) $\thinspace $\overrightarrow{=}f_{k-1}\left( x\right) $
for all $x\in V_{G}.$
\end{enumerate}

Then $r$ is a deformation retraction, the digraphs $G$ and $H$ are homotopy
equivalent, and $i$, $r$ are their homotopy inverses.
\end{proposition}

\begin{proof}
Since $f_{k-1}$ and $f_{k}$ satisfy (\ref{f=g}), we see that $f_{k-1}\simeq
f_{k}$ whence by induction we obtain that $f_{n}\simeq f_{0}$ and, hence, $%
i\circ r\simeq \func{id}_{G}.$ Therefore, $r$ is a deformation retraction,
and the rest follow from Proposition \ref{d3.5}.
\end{proof}

\begin{corollary}
\label{c3.8}Let $r:G\rightarrow H$ be a retraction of a digraph $G$ onto a
sub-digraph \ $H$ and
\begin{equation}
x\,\overrightarrow{=}r\left( x\right) \ \text{for all }x\in V_{G}\ \ \text{%
or\ \ \ }r\left( x\right) \,\overrightarrow{=}x\ \ \text{for all }x\in V_{G}.
\label{kk}
\end{equation}%
Then $r$ is a deformation retraction, the digraphs $G$ and $H$ are homotopy
equivalent, and $i$, $r$ are their homotopy inverses.
\end{corollary}

Clearly, Corollary \ref{c3.8} is an important particular case $n=1$ of
Proposition \ref{c3.7}. Note also that the condition (\ref{kk}) is
automatically satisfies for all $x\in V_{H}$, so in applications it remains
to verify it for $v\in V_{G}\setminus V_{H}$.

\begin{corollary}
\label{c3.9}For any digraph $G$ and for any line digraph $I_{n}\in \mathcal{I%
}_{n}\ (n\geq 0)$ we have $G\boxdot I_{n}\simeq G.$
\end{corollary}

\begin{proof}
It suffices to show that $G\boxdot I_{n}\simeq G\boxdot I_{n-1}$ where $%
I_{n-1}$ is obtained from $I_{n}$ by removing the vertex $n$ and the
adjacent edge, and then to argue by induction since $G\boxdot I_{0}=G$.
Define a retraction $r:G\boxdot I_{n}\rightarrow G\boxdot I_{n-1}$ by%
\begin{equation*}
r\left( x,k\right) =\left\{
\begin{array}{ll}
\left( x,k\right) , & k\leq n-1, \\
\left( x,n-1\right) , & k=n.%
\end{array}%
\right.
\end{equation*}%
Let us show that $r$ is an $1$-step deformation retraction, that is, $r$
satisfied (\ref{kk}):%
\begin{equation*}
\left( x,k\right) \,\overrightarrow{=}r\left( x,k\right) \ \text{for all }%
\left( x,k\right) \in G\boxdot I_{n}\ \ \text{or\ \ }r\left( x,k\right) \,%
\overrightarrow{=}\left( x,k\right) \ \ \text{for all }\left( x,k\right) \in
G\boxdot I_{n}
\end{equation*}%
Indeed, for $k\leq n-1$ this is obvious. If $k=n$ then consider two cases.

\begin{enumerate}
\item If $\left( n-1\right) \rightarrow n$ in $I_{n}$ then $\left(
x,n-1\right) \rightarrow \left( x,n\right) $ in $G\boxdot I_{n}$ whence%
\begin{equation*}
r\left( x,k\right) =r\left( x,n\right) =\left( x,n-1\right) \rightarrow
\left( x,n\right) =\left( x,k\right) .
\end{equation*}

\item If $n\rightarrow \left( n-1\right) \ $in $I_{n}$ then $\left(
x,n\right) \rightarrow \left( x,n-1\right) $ in $G\boxdot I_{n}$ whence%
\begin{equation*}
\left( x,k\right) \rightarrow r\left( x,k\right) .
\end{equation*}
\end{enumerate}
\end{proof}

\begin{corollary}
\label{c3.10}Let $G$ be a digraph. Fix some $n\in \mathbb{N}$ and consider
for any $k=0,...,n$ the natural inclusion%
\begin{equation*}
i_{k}\colon G\rightarrow G\boxdot I_{n},\ \ \ \ i_{k}(v)=\left( v,k\right) \
\end{equation*}%
and a natural projection
\begin{equation*}
p\colon G\boxdot I_{n}\rightarrow G,\ \ \ \ p(v,k)=v.
\end{equation*}%
Then the maps $i_{,}p$ induce isomorphism of homology groups.
\end{corollary}

\begin{proof}
The projection $p$ can be decomposed into composition of retractions $%
G\boxdot I_{m}\rightarrow G\boxdot I_{m-1}$ which are homotopy equivalences
by the proof of Corollary \ref{c3.9}. Therefore, $p$ is also a homotopy
equivalence and hence induces isomorphism of homology groups. The inclusion $%
i_{k}$ can be decomposed into composition of natural inclusions $G\boxdot
I_{m-1}\rightarrow G\boxdot I_{m}$, each of them being homotopy inverse of
the retraction $G\boxdot I_{m}\rightarrow G\boxdot I_{m-1}$, which implies
the claim.
\end{proof}

\begin{example}
\label{e8.4} \RM A digraph $G$ is called a \emph{tree} if the underlying
undirected graph is a tree. We claim that if a digraph $G$ is a connected
tree then $G$ is contractible. Indeed, let $a$ be a pendant vertex of $G$
and let $b$ be another vertex such that $a\rightarrow b$ or $b\leftarrow a$.
Let $G^{\prime }$ be the subgraph of $G$ that is obtained from $G$ by
removing the vertex $a$ with the adjacent edge. Then the map $r:G\rightarrow
G^{\prime }$ defined by $r\left( a\right) =b$ and $r|_{H}=\func{id}$ is by
Corollary \ref{c3.8} a deformation retraction, whence $G\simeq G^{\prime }.$
Since $G^{\prime }$ is also a connected tree, continuing the procedure of
removing of a pendant vertices, we obtain in the end that $G$ is
contractible.
\end{example}

\begin{example}
\label{e3.11} \RM A digraph $G$ is called \emph{\ star-like} (resp. inverse
star like) if there is a vertex $a\in V_{G}$ such that $a\rightarrow x$
(resp. $x\rightarrow a$) for all $x\in V_{G}\setminus \left\{ a\right\} .$
if $G$ is a (inverse) star-like digraph, then the map $r\colon G\rightarrow
\left\{ a\right\} $ is by Corollary \ref{c3.8} a deformation retraction,
whence we obtain $G\simeq \left\{ a\right\} $, that is, $G$ is contractible.
Consequently, all homology groups of $G$ are trivial except for $H_{0}$.

For example, consider a \emph{digraph-simplex} of dimension $n$, which is a
digraph $G$ with the set of vertices $\{0,1,\dots ,n\}$ and the set of edges
given by the condition
\begin{equation*}
i\rightarrow j\ \Longleftrightarrow \ i<j
\end{equation*}%
(a digraph-simplex with $n=3$ is shown on the left panel on Fig. \ref{pic11}%
). Then $G$ is star-like and, hence, $G$ is contractible. In particular, the
triangular digraph from Example \ref{e2.8} is contractible. Another
star-like digraph is shown on the right panel of Fig. \ref{pic11}. \FRAME{%
ftbpFU}{6.3304in}{1.3863in}{0pt}{\Qcb{Star-like digraphs}}{\Qlb{pic11}}{%
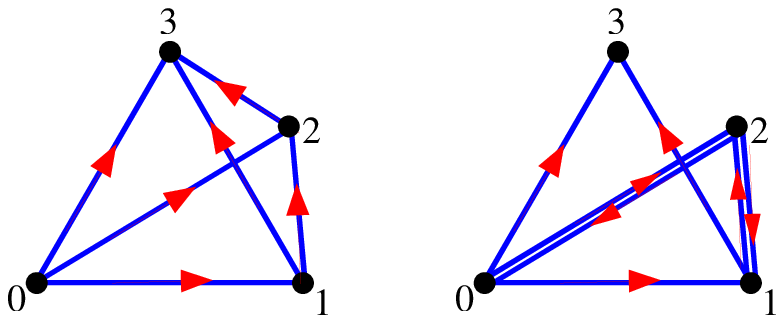}{\special{language "Scientific Word";type
"GRAPHIC";maintain-aspect-ratio TRUE;display "USEDEF";valid_file "F";width
6.3304in;height 1.3863in;depth 0pt;original-width 6.3027in;original-height
1.3586in;cropleft "0";croptop "1";cropright "1";cropbottom "0";filename
'pic11.eps';file-properties "XNPEU";}}
\end{example}

\begin{example}
\label{e3.13} \RM For any $n\geq 1$, consider the $n$-dimensional cube
\begin{equation*}
I^{n}=\underset{n\text{ times}}{\underbrace{I\boxdot I\boxdot \dots \boxdot I%
}}
\end{equation*}%
For example, $I_{2}$ is the square from Example \ref{e2.8} and $I_{3}$ is a $%
3$-cube shown on Fig. \ref{pic9}. By Corollary \ref{c3.9} we have $%
I^{k}\simeq I^{k-1}$, whence we obtain that, for all $n$, $I^{n}\simeq I$,
which implies that $I^{n}$ is contractible. In particular, this applies to a
square digraph from Example \ref{e2.8}. Consequently, the all homology
groups of $I^{n}$ are trivial except for $H_{0}$.
\end{example}

\begin{example}
\RM\label{ExSnSm}Let $S_{n}$ be a cycle digraph from Example \ref{e2.8}. If $%
S_{n}$ is the triangle or square as in (\ref{ts}) then $S_{n}$ is
contractible as was shown in Examples \ref{e3.11} and \ref{e3.13},
respectively. If $S_{n}$ is neither triangle nor square then by Example \ref%
{e2.8} $H_{1}(S_{n},\mathbb{K})\cong \mathbb{K}$ and, hence, $S_{n}$ is not
contractible. In particular, this is always the case when $n\geq 5$. Here
are other examples of non-contractible cycles with $n=3,4$:%
\begin{equation*}
\text{ }%
\begin{array}{c}
_{_{\nearrow }}\overset{1}{\bullet }_{_{\searrow }} \\
^{0}\bullet \ \longleftarrow \ \bullet ^{2}%
\end{array}%
\ \ \text{and \ }%
\begin{array}{ccc}
^{1}\bullet & \longrightarrow & \bullet ^{2} \\
\ \uparrow &  & \downarrow \  \\
^{0}\bullet & \longleftarrow & \bullet ^{3}%
\end{array}%
,
\end{equation*}

Let us show that two cycles $S_{n}$ and $S_{m}$ with $n\neq m$ are not
homotopy equivalent, except for the case when one of them is a triangle and
the other is a square. Assume that $S_{n}$ and $S_{m}$ with $n<m$ are
homotopy equivalent. Then by Theorem \ref{t3.4} there is a digraph map $%
f:S_{n}\rightarrow S_{m}$ such that $f_{\ast }:H_{1}\left( S_{n}\right)
\rightarrow H_{1}\left( S_{m}\right) $ is an isomorphism. If homology groups
$H_{1}\left( S_{n}\right) $ and $H_{1}\left( S_{m}\right) $ are not
isomorphic then we are done. If they are isomorphic, then they are
isomorphic to $\mathbb{K}$. Let $\varpi _{n}\in \Omega _{1}\left(
S_{n}\right) $ be the generator of closed $1$-paths on $S_{n}$ and $\varpi
_{m}\in \Omega _{1}\left( S_{m}\right) $ be the generator of closed $1$%
-paths on $S_{n}$, as in (\ref{om}). Then $\left[ \varpi _{n}\right] $
generates $H_{1}\left( S_{n}\right) $, $\left[ \varpi _{m}\right] $
generates $H_{1}\left( S_{m}\right) $, and we should have
\begin{equation*}
f_{\ast }\left( \left[ \varpi _{n}\right] \right) =k\left[ \varpi _{m}\right]
\end{equation*}%
for some non-zero constant $k\in \mathbb{K}$. Consequently, we obtain
\begin{equation*}
f_{\ast }\left( \varpi _{n}\right) =k\varpi _{m},
\end{equation*}%
which is impossible because $f$ cannot be surjective by $n<m$, whereas $%
\varpi _{m}$ uses all the vertices of $S_{m}$.
\end{example}

\begin{example}
\RM Consider the digraph $G$ as on Fig. \ref{pic20}.\FRAME{ftbpF}{5.6611in}{%
1.7487in}{0pt}{}{\Qlb{pic20}}{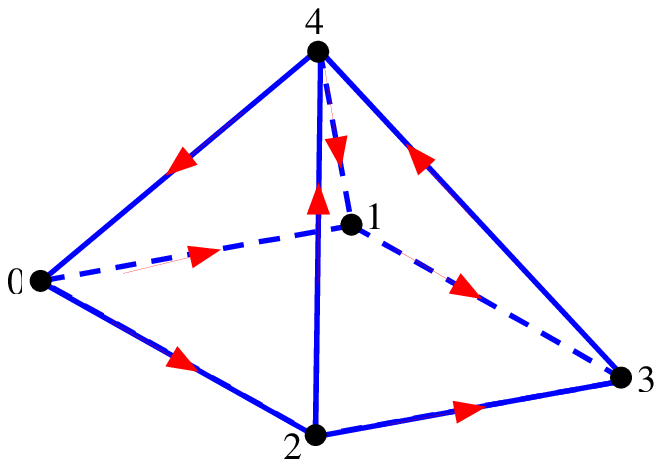}{\special{language "Scientific
Word";type "GRAPHIC";maintain-aspect-ratio TRUE;display "USEDEF";valid_file
"F";width 5.6611in;height 1.7487in;depth 0pt;original-width
6.3027in;original-height 1.9268in;cropleft "0";croptop "1";cropright
"1";cropbottom "0";filename 'pic20.eps';file-properties "XNPEU";}}

Consider also its sub-digraph $H$ with the vertex set $V_{H}=\left\{
1,3,4\right\} $ and a retraction $r:G\rightarrow H$ given by $r\left(
0\right) =1,\ r\left( 2\right) =3$ and $r|_{H}=\func{id}$. By Corollary \ref%
{c3.8}, $r$ is a deformation retraction, whence $G\simeq H$. Consequently,
we obtain $H_{1}\left( G,\mathbb{K}\right) \cong H_{1}\left( H,\mathbb{K}%
\right) \cong \mathbb{K}$ and $H_{p}\left( G,\mathbb{K}\right) =\left\{
0\right\} $ for $p\geq 2$.
\end{example}

\begin{example}
\RM Let $a$ be a vertex in a digraph $G$ and let $b_{0},b_{1},...,b_{n}$ be
all the neighboring vertices of $a$ in $G$. Assume that the following
condition is satisfied:%
\begin{equation}
\forall i=1,...,n\ \ \ \ a\rightarrow b_{i}\Rightarrow b_{0}\rightarrow
b_{i}\ \ \text{and\ \ \ }a\leftarrow b_{i}\Rightarrow b_{0}\leftarrow b_{i}.
\label{abi}
\end{equation}%
Denote by $H$ the digraph that is obtained from $G$ by removing a vertex $a$
with all adjacent edges. The map $r:G\rightarrow H$ given by $r\left(
a\right) =b_{0}$ and $r|_{H}=\func{id}$ is by Corollary \ref{c3.8} a
deformation retraction, whence we obtain that $G\simeq H.$ Consequently, all
homology groups of $G$ and $H$ are the same. This is very similar to the
results about transformations of simplicial complexes in the simple homotopy
theory (see, for example, \cite{Cohen}).

In particular, (\ref{abi}) is satisfied if $a\rightarrow b_{i}\ $and $%
b_{0}\rightarrow b_{i}$ for all $i\geq 1$ or $a\leftarrow b_{i}\ $and $%
b_{0}\leftarrow b_{i}$ for all $i\geq 1$. Two examples when (\ref{abi}) is
satisfied are shown in the following diagram:%
\begin{equation*}
\fbox{$%
\begin{array}{cc}
&  \\
& \nearrow \\
a~\bullet & \longrightarrow \\
& \searrow \\
&
\end{array}%
$\fbox{$%
\begin{array}{c}
\bullet ~b_{1} \\
\uparrow ~~\  \\
\bullet ~b_{0} \\
\downarrow ~~\  \\
\bullet ~b_{1}%
\end{array}%
\cdots $\ \ \ $H$}\ \ \ $G$}\ \ \ \ \fbox{$%
\begin{array}{cc}
&  \\
& \nearrow \\
a~\bullet & \longleftarrow \\
& \nwarrow \\
&
\end{array}%
$\fbox{$%
\begin{array}{c}
\bullet ~b_{1} \\
\uparrow ~~\  \\
\bullet ~b_{0} \\
\uparrow ~~\  \\
\bullet ~b_{1}%
\end{array}%
\cdots $\ \ \ $H$}\ \ \ $G$}
\end{equation*}%
On the contrary, the digraph $G$ on following diagram%
\begin{equation*}
\begin{array}{ccc}
&  & \bullet \ b_{1} \\
& \nearrow & \downarrow \ \ \  \\
a~\bullet & \longleftarrow & \bullet \ b_{0}%
\end{array}%
\end{equation*}%
does not satisfy (\ref{abi}). Moreover, this digraph is not homotopy
equivalent to $H=\left( ^{0}\bullet \rightarrow \bullet ^{1}\right) $ since $%
G$ and $H$ have different homology group $H_{1}$ (cf. Example \ref{e2.8}).

For example, the digraph on the left panel of Fig. \ref{pic15b} is
contractible as one can successively remove the vertices $5,4,3,2$ each time
satisfying (\ref{abi}). \FRAME{ftbpFU}{5.8306in}{1.6847in}{0pt}{\Qcb{The
left digraph is contractible while the right one is not. }}{\Qlb{pic15b}}{%
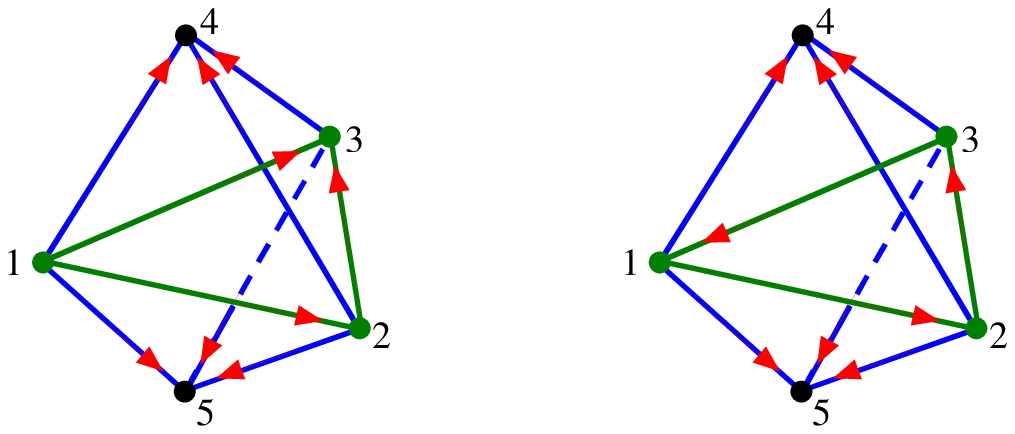}{\special{language "Scientific Word";type
"GRAPHIC";maintain-aspect-ratio TRUE;display "USEDEF";valid_file "F";width
5.8306in;height 1.6847in;depth 0pt;original-width 6.3027in;original-height
1.8023in;cropleft "0";croptop "1";cropright "1";cropbottom "0";filename
'pic15b.eps';file-properties "XNPEU";}}

The digraph on the right panel of Fig. \ref{pic15b} is different from the
left one only by the direction of the edge between $1$ and $3$, but it is
not contractible as its $H_{2}$ group is non-trivial by Example \ref{e2.7}.

Consider one more example: the digraph $G$ on Fig. \ref{pic6a}. \FRAME{ftbhFU%
}{5.0704in}{2.2952in}{0pt}{\Qcb{{}Graph $G$}}{\Qlb{pic6a}}{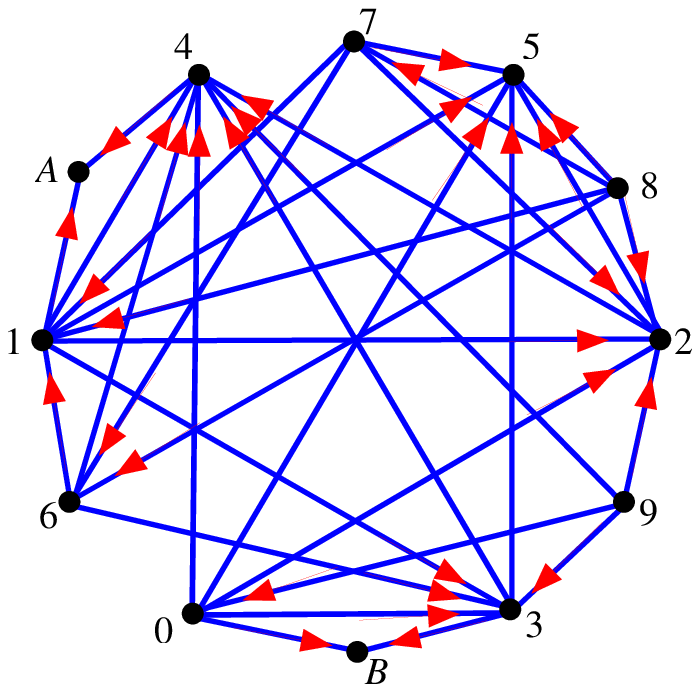}{%
\special{language "Scientific Word";type "GRAPHIC";maintain-aspect-ratio
TRUE;display "USEDEF";valid_file "F";width 5.0704in;height 2.2952in;depth
0pt;original-width 6.3027in;original-height 2.6671in;cropleft "0";croptop
"1";cropright "1";cropbottom "0";filename 'pic6a.eps';file-properties
"XNPEU";}}

Removing successively the vertices $A,B,8,9,6,7$, which each time satisfy (%
\ref{abi}), we obtain a digraph $H$ with $V_{H}=\left\{ 0,1,2,3,4,5\right\} $
that is homotopy equivalent to $G$ and, in particular, has the same
homologies as $G$. The digraph $H$ is shown in two ways on Fig. \ref{pic6b}.
Clearly, the second representation of this graph is reminiscent of an
octahedron.\FRAME{ftbhFU}{3.8086in}{1.4944in}{0pt}{\Qcb{Two representations
of the digraph $H$}}{\Qlb{pic6b}}{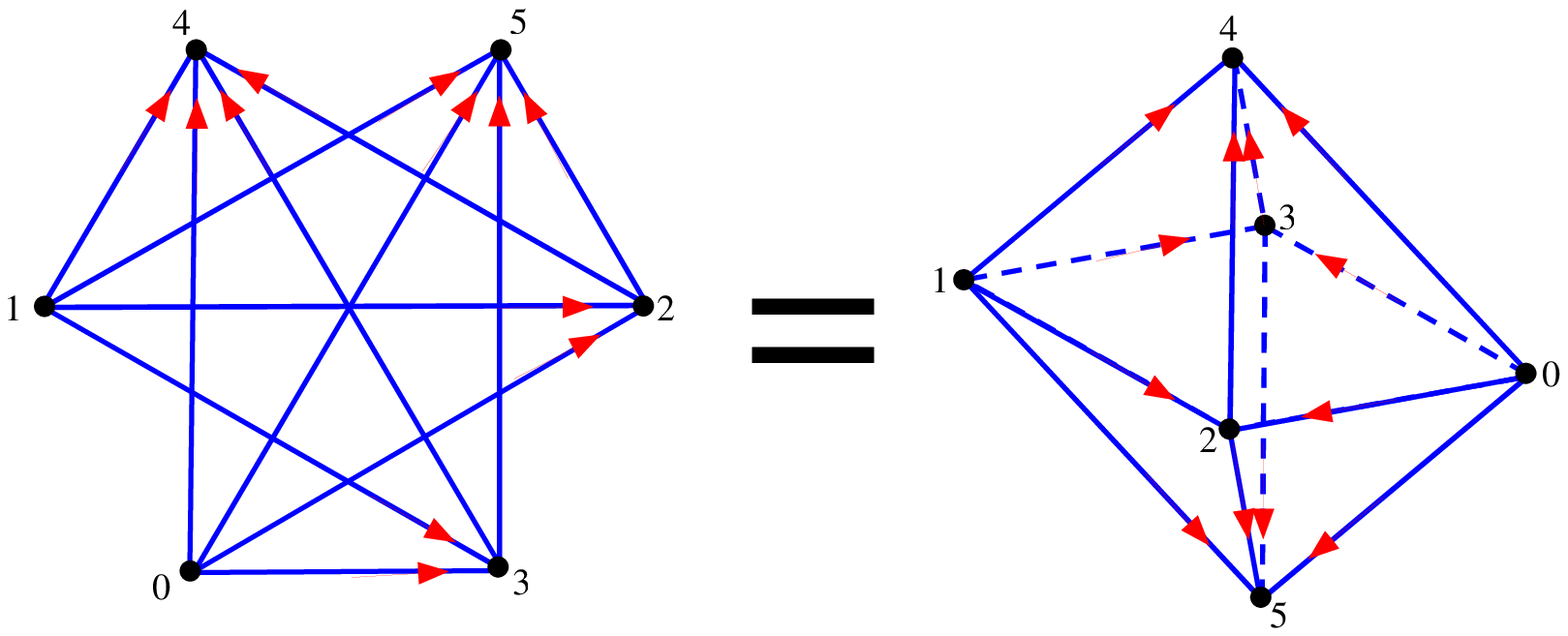}{\special{language "Scientific
Word";type "GRAPHIC";maintain-aspect-ratio TRUE;display "USEDEF";valid_file
"F";width 3.8086in;height 1.4944in;depth 0pt;original-width
6.3027in;original-height 2.5685in;cropleft "0";croptop "1";cropright
"1";cropbottom "0";filename 'pic6b.eps';file-properties "XNPEU";}}

It is possible to show that $H_{p}\left( H,\mathbb{K}\right) =\left\{
0\right\} $ for $p=1$ and $p>2$ while $H_{2}\left( H,\mathbb{K}\right) \cong
\mathbb{K}.$ It follows that the same is true for the homology groups of $G$%
. Furthermore, it is possible to show that $H_{2}\left( G,\mathbb{K}\right) $
is generated by the following $2$-path%
\begin{equation*}
\omega =e_{024}-e_{025}-e_{034}+e_{035}-e_{124}+e_{125}+e_{134}-e_{135},
\end{equation*}%
that determines a $2$-dimensional \textquotedblleft hole\textquotedblright\
in $G$ given by the octahedron $H$. Note that on Fig. \ref{pic6a} this
octahedron is hardy visible.
\end{example}

\subsection{Cylinder of a map}

Let us give some useful examples of homotopy equivalent digraphs.

\begin{definition}
\label{d3.14} \RM Let $G=\left( V_{G},E_{G}\right) $ and $H=\left(
V_{H},E_{H}\right) $ be two digraphs and $f$ be a digraph map from $G$ to $H$%
. The \emph{cylinder} $\func{C}_{f}$ of $f$ is the digraph with the set of
vertices $V_{\func{C}_{f}}=V_{G}\sqcup V_{H}$ and with the set of edges $E_{%
\func{C}_{f}}$ that consists of all the edges from $E_{G}$ and $E_{H}$ as
well as of the edges of the form $x\rightarrow f\left( x\right) $ for all $%
x\in V_{G}.$

The inverse cylinder $\func{C}_{f}^{-}$ is defined in the same way except
that the edge $x\rightarrow f\left( x\right) $ is replaced by $f\left(
x\right) \rightarrow x$.
\end{definition}

For example, for $f=\func{id}_{G}$ we have $\func{C}_{f}=G\boxdot I$ where $%
I=\left( ^{0}\bullet \longrightarrow \bullet ^{1}\right) $ and $\func{C}%
_{f}^{-}=G\boxdot I^{-}$ where $I^{-}=\left( ^{0}\bullet \longleftarrow
\bullet ^{1}\right) .$

\begin{example}
\RM Let $G$ be the digraph with vertices $\left\{ 0,1,2,3,4,5\right\} $ and $%
H$ is be the digraph with vertices $\left\{ a,b,c\right\} $ as on Fig. \ref%
{pic16a}. Consider the digraph map $f:G\rightarrow H$ given by $f\left(
0\right) =f\left( 1\right) =a,$ $f\left( 2\right) =f\left( 3\right) =b$ and $%
f\left( 4\right) =f\left( 5\right) =c$. The cylinder $\func{C}_{f}$ of $f$
is shown on Fig. \ref{pic16a}.\FRAME{ftbpFU}{5.0704in}{1.5696in}{0pt}{\Qcb{%
The cylinder of the map}}{\Qlb{pic16a}}{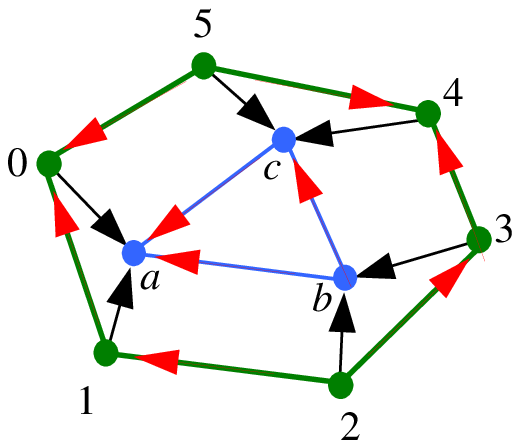}{\special{language
"Scientific Word";type "GRAPHIC";maintain-aspect-ratio TRUE;display
"USEDEF";valid_file "F";width 5.0704in;height 1.5696in;depth
0pt;original-width 6.3027in;original-height 1.0101in;cropleft "0";croptop
"1";cropright "1";cropbottom "0";filename 'pic16a.eps';file-properties
"XNPEU";}}
\end{example}

\begin{proposition}
\label{p3.15} Let $f$ be a digraph map from $G$ to $H$. Then we have the
following homotopy equivalences of the digraphs
\begin{equation*}
\func{C}_{f}\simeq H\simeq \func{C}_{f}^{-}.
\end{equation*}
\end{proposition}

\begin{proof}
The projection $p:\func{C}_{f}\rightarrow H$ defined by%
\begin{equation*}
p\left( x\right) =\left\{
\begin{array}{ll}
x, & x\in V_{H}, \\
f\left( x\right) , & x\in V_{G},%
\end{array}%
\right.
\end{equation*}%
is clearly an $1$-step deformation retraction of $\func{C}_{f}$ onto $H$,
whence it follows by Corollary \ref{c3.8} that $\func{C}_{f}\simeq H$. The
case of the inverse cylinder $\func{C}_{f}^{-}$ is similar.
\end{proof}

\section{Homotopy groups of digraphs}

\label{S4}\setcounter{equation}{0}

In this Section we define homotopy groups of digraphs and describe theirs
basic properties. For that, we introduce the concept of path-map in a
digraph $G$, and then define a fundamental group of $G$. Then the higher
homotopy group can be defined inductively as the fundamental group of the
corresponding iterated loop-digraph.

A based digraph $G^{\ast }$ is a digraph $G$ with a fixed base vertex $\ast
\in V_{G}.$ A based digraph map $f:G^{\ast }\rightarrow H^{\ast }$ is a
digraph map $f:G\rightarrow H$ such that $f\left( \ast \right) =\ast $. A
category of based digraphs will be denoted by $\mathcal{D}^{\ast }$.

A homotopy between two based digraph maps $f,g:G^{\ast }\rightarrow H^{\ast
} $ is defined as in Definition \ref{d3.1} with additional requirement that $%
F|_{\left\{ \ast \right\} \boxdot I_{n}}=\ast .$

\subsection{Construction of $\protect\pi _{0}$}

Let $G^{\ast }$ be a based digraph, and $V_{2}^{\ast }=\{0,1\}$ be the based
digraph consisting of two vertices, no edges and with the base vertex $%
0=\ast $. Let $Hom(V_{2}^{\ast },G^{\ast })$ be the set of based digraph
maps from $V_{2}^{\ast }$ to $G^{\ast }$. Note that the set of such maps is
in one to one correspondence with the set of vertices of the digraph $G$.

\begin{definition}
\label{d4.5} \RM We say that two digraph maps $\phi ,\psi \in
Hom(V_{2}^{\ast },G^{\ast })$ are \emph{equivalent} and write $\phi \simeq
\psi $ if there exists $I_{n}\in \mathcal{I}$ and a digraph map
\begin{equation*}
f\colon I_{n}\rightarrow G,
\end{equation*}%
such that $f(0)=\phi \left( 1\right) $ and $f(n)=\psi \left( 1\right) $. The
relation $\simeq $ is evidently an equivalence relation, and we denote by $%
[\phi ]$ the equivalence class of the element $\phi $, and by $\pi
_{0}(G^{\ast })$ the set of classes of equivalence with the base point $\ast
$ given by a class of equivalence of the trivial map $V_{2}\rightarrow \ast
\in G.$
\end{definition}

The set $\pi _{0}(G^{\ast })$ coincides with the set of connected components
of the digraph $G$. In particular, the digraph $G^{\ast }$ connected if $\pi
_{0}(G^{\ast })=\ast $.

\begin{proposition}
\label{p4.6}Any based digraph map $f\colon G^{\ast }\rightarrow H^{\ast }$
induces a map
\begin{equation*}
\pi _{0}(f)\colon \pi _{0}(G^{\ast })\rightarrow \pi _{0}(H^{\ast })
\end{equation*}%
of based sets. The homotopic maps induce the same map of based sets. We have
a functor from the category $\mathcal{D}^{\ast }$ of digraphs to the
category based sets.
\end{proposition}

\begin{proof}
Let $x=[\phi ]\in \pi _{0}(G^{\ast })$ be presented by a digraph map $\phi
\colon V_{2}^{\ast }\rightarrow G^{\ast }$ we put $y=[\pi
_{0}(f)](x)=[f\circ \phi ]\in \pi _{0}(H^{\ast })$. It is an easy exercise
to check that this map $\pi _{0}(f)$ is well defined and for the homotopic
maps $f\simeq g\colon G^{\ast }\rightarrow H^{\ast }$ we have $\pi
_{0}(f)=\pi _{0}(g)$.
\end{proof}

\subsection{$C$-homotopy and $\protect\pi _{1}$}

For any line digraph $I_{n}\in \mathcal{I}_{n}$, a based digraph $%
I_{n}^{\ast }$ will always have the base point $0$.

\begin{definition}
\RM A\textbf{\ }\emph{path-map }in a digraph $G$ is any digraph map $\phi
:I_{n}\rightarrow G$, where $I_{n}\in \mathcal{I}_{n}$. A \emph{based
path-map} on a based digraph $G^{\ast }$ is a based digraph map $\phi
:I_{n}^{\ast }\rightarrow G^{\ast }$, that is, a digraph map such that $\phi
\left( 0\right) =\ast $. A \emph{loop} on $G^{\ast }$ is a based path-map $%
\phi :I_{n}^{\ast }\rightarrow G^{\ast }$ such that $\phi \left( n\right)
=\ast $.
\end{definition}

Note that the image of a path-map is not necessary an allowed path of the
digraph $G$.

\begin{definition}
\label{d4.2}\RM A digraph map $h:I_{n}\rightarrow I_{m}\ $is\ called \emph{%
shrinking} if $\ h\left( 0\right) =0$, $h(n)=m$, and $h\left( i\right) \leq
h\left( j\right) $ whenever $i\leq j$ (that is, if $h$ as a function from $%
\left\{ 0,...,n\right\} $ to $\left\{ 0,...,m\right\} $ is monotone
increasing).
\end{definition}

Any shrinking $h:I_{n}\rightarrow I_{m}$ is by definition a based digraph
map. Moreover, $h$ is surjective and the preimage of any edge of $I_{m}$
consists of exactly one edge of $I_{n}$. Furthermore, we have necessarily $%
m\leq n$, and if $n=m$ then $h$ is a bijection.

\begin{definition}
\label{d4.3}\RM Consider two based path-maps
\begin{equation*}
\phi \colon I_{n}^{\ast }\rightarrow G^{\ast }\ \ \text{and}\ \ \psi \colon
I_{m}^{\ast }\rightarrow G^{\ast }.
\end{equation*}%
An \emph{one-step direct }$C$\emph{-homotopy }from $\phi $ to $\psi $ is
given by a shrinking map $h:I_{n}\rightarrow I_{m}$ such that the map $F:V_{%
\func{C}_{h}}\rightarrow V_{G}$ given by%
\begin{equation}
F|_{I_{n}}=\phi \ \ \ \text{and\ \ \ }F|_{I_{m}}=\psi ,  \label{Ffipsi}
\end{equation}%
is a digraph map from $\func{C}_{h}$ to $G$. If the same is true with $\func{%
C}_{h}$ replaced everywhere by $\func{C}_{h}^{-}$ then we refer to an
one-step \emph{inverse} $C$-homotopy.
\end{definition}

\begin{remark}
\RM\label{Remfipsi}The requirement that $F$ is a digraph map is equivalent
to the condition%
\begin{equation}
\phi \left( i\right) \overrightarrow{=}\psi \left( h\left( i\right) \right)
\ \ \text{for all }i\in I_{n}.  \label{fipsih}
\end{equation}%
In turn, (\ref{fipsih}) implies that the digraph maps $\phi $ and $\psi
\circ h$ (acting from $I_{n}$ to $G$) satisfy (\ref{f=g}), which yields $%
\phi \simeq \psi \circ h$.

If $n=m$ then $h=\func{id}_{I_{n}}$ and an one-step $C$-homotopy is a
homotopy.
\end{remark}

\begin{example}
\RM An example of one-step direct $C$-homotopy is shown in Fig. \ref{pic1}.%
\FRAME{ftbpFU}{4.5697in}{1.6094in}{0pt}{\Qcb{The loops $\protect\phi %
:I_{3}\rightarrow G$ and and $\protect\psi :I_{5}\rightarrow G$ are $C$%
-homotopic. Note that $\protect\phi \left( 0\right) =\protect\phi \left(
5\right) =\ast =\protect\psi \left( 0\right) =\protect\psi \left( 3\right) .$%
}}{\Qlb{pic1}}{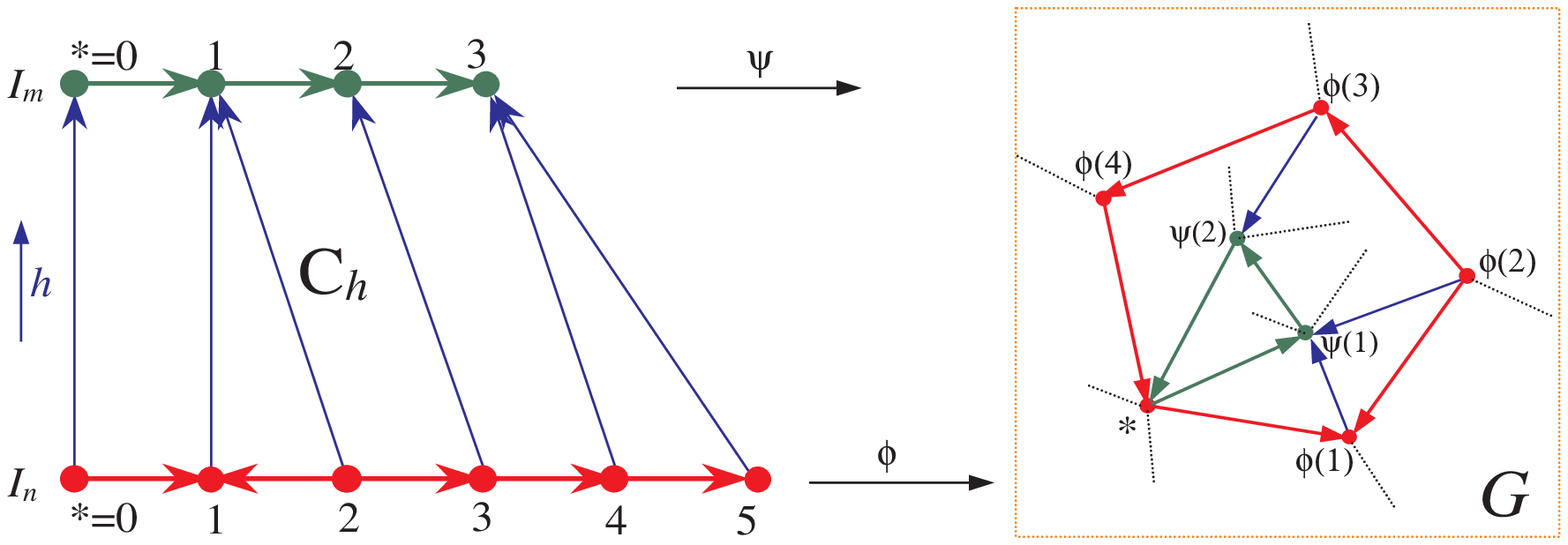}{\special{language "Scientific Word";type
"GRAPHIC";maintain-aspect-ratio TRUE;display "USEDEF";valid_file "F";width
4.5697in;height 1.6094in;depth 0pt;original-width 7.0681in;original-height
2.2485in;cropleft "0";croptop "1";cropright "1";cropbottom "0";filename
'pic1.eps';file-properties "XNPEU";}}

Note that the images of the loops $\phi $ and $\psi $ on Fig. \ref{pic1} are
not homotopic as digraphs because they are cycles of different lengths $5$
and $3$ (see Example \ref{ExSnSm}). Nevertheless, the loops $\phi $ and $%
\psi $ are $C$-homotopic.
\end{example}

\begin{definition}
\label{d4.4} \RM For a based digraph $G^{\ast }$ define a path-digraph $PG$
as follows. The vertices of $PG$ are all the based path-maps in $G^{\ast }$,
and the edges of $PG$ are defined by the following rule: $\phi \rightarrow
\psi $ in $PG$ if $\phi \neq \psi $ \label{rem: is this correct that fi not
=psi}and there is an one-step direct $C$-homotopy from $\phi $ to $\psi $ or
an one-step inverse $C$-homotopy from $\psi $ to $\phi $.

Then define a \emph{based path-digraph } $PG^{\ast }$ by choosing in $PG$
the base vertex $I_{0}^{\ast }\rightarrow G^{\ast }$, which will also be
denoted by $\ast $. \label{d4.7}Define a \emph{based loop-digraph} $LG^{\ast
}$ as a sub-digraph of $PG^{\ast }$ whose set of vertices consists of all
the loops of $G^{\ast }.$
\end{definition}

Any map $f\colon G^{\ast }\rightarrow H^{\ast }$ induces a based map of
path-digraphs
\begin{equation*}
Pf:PG^{\ast }\rightarrow PH^{\ast },\ \ \ \ \left( Pf\right) \left( \psi
\right) =f\circ \psi ,
\end{equation*}%
where $\psi :I_{n}^{\ast }\rightarrow G^{\ast }$ is a based path-map. Hence,
$P$ is a functor from the category $\mathcal{D^{\ast }}$ to itself.
Similarly we have a map of based loop digraphs
\begin{equation}
Lf:LG^{\ast }\rightarrow LH^{\ast },\ \ \ \ \ (Lf)(\psi )=f\circ \psi ,
\label{Lf}
\end{equation}%
where $\psi :I_{n}^{\ast }\rightarrow G^{\ast }$ is a loop. Hence, $L$ is a
functor from the category $\mathcal{D}^{\ast }$ to itself.

\begin{definition}
\RM We call two based path-maps $\phi ,\psi \in PG$ \emph{\ }$C$-\emph{%
homotopic} and write $\phi \overset{C}{\simeq }\psi $ if there exists a
finite sequence $\left\{ \phi _{k}\right\} _{k=0}^{m}$ of based path-maps in
$PG$ such that $\phi _{0}=\phi $, $\phi _{m}=\psi $ and, for any $%
k=0,...,m-1 $, holds $\phi _{k}\rightarrow \phi _{k+1}$ or $\phi
_{k+1}\rightarrow \phi _{k}$.
\end{definition}

Obviously, the relation $\phi \overset{C}{\simeq }\psi $ holds if and only
if $\phi $ and $\psi $ belong to the same connected component of the
undirected graph of $PG$. In particular, the $C$-homotopy is an equivalence
relation.

\begin{definition}
\label{d4.8}\RM Let $\pi _{1}(G^{\ast })$ be a set of equivalence classes
under $C$-homotopy of based loops of a digraph $G^{\ast }$. The $C$-homotopy
class of a based loop $\phi $ will be denoted by $[\phi ]$.
\end{definition}

Note that $\pi _{1}(G^{\ast })=\pi _{0}(LG^{\ast })$ as follows directly
from Definitions \ref{d4.7} and \ref{d4.8}. Denote by $e$ the trivial loop $%
e:I_{0}^{\ast }\rightarrow G^{\ast }.$ We say that a loop $\phi $ is $C$-%
\emph{contractible} if $\phi \overset{C}{\simeq }e.$

\begin{example}
\RM A triangular loop is a loop $\phi :I_{3}^{\ast }\rightarrow G^{\ast }$
such that $I_{3}=\left( 0\rightarrow 1\rightarrow 2\leftarrow 3\right) .$%
\FRAME{ftbpFU}{4.4624in}{1.3898in}{0pt}{\Qcb{A triangular loop $\protect\phi
$ is $C$-contractible.}}{\Qlb{pic2}}{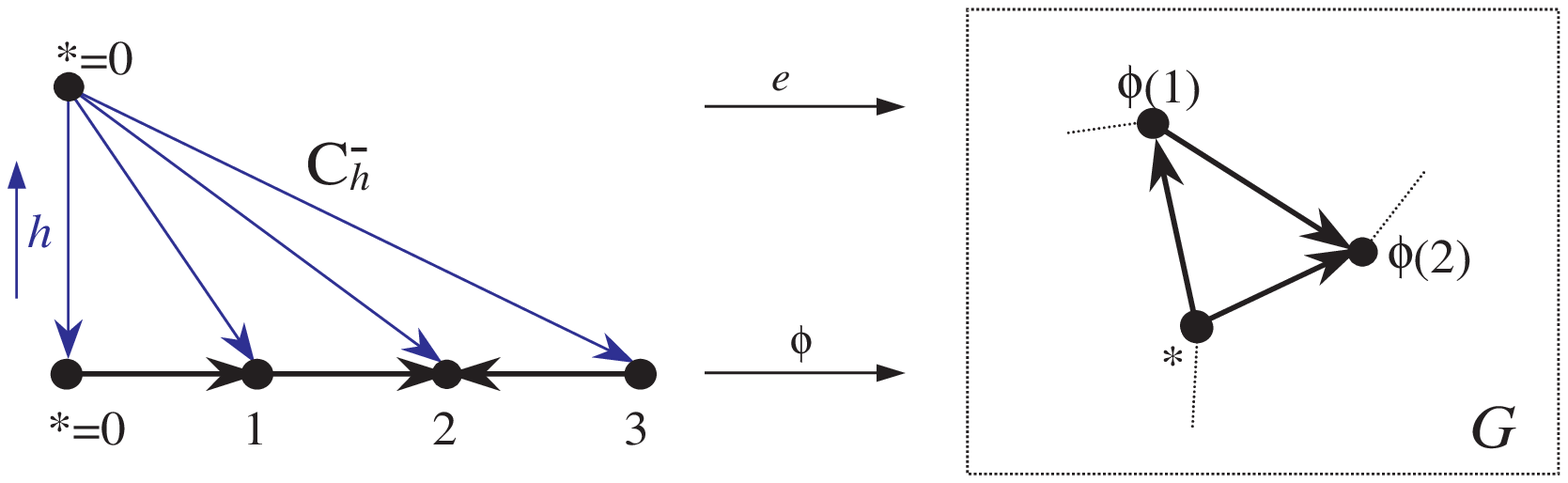}{\special{language "Scientific
Word";type "GRAPHIC";maintain-aspect-ratio TRUE;display "USEDEF";valid_file
"F";width 4.4624in;height 1.3898in;depth 0pt;original-width
6.6366in;original-height 2.0401in;cropleft "0";croptop "1";cropright
"1";cropbottom "0";filename 'pic2.eps';file-properties "XNPEU";}}

The triangular loop is $C$-contractible because the following shrinking map%
\begin{equation*}
h:I_{3}^{\ast }\rightarrow I_{0}^{\ast },\ \ h\left( k\right) =0\ \text{for
all }k=0,...,3,
\end{equation*}%
provides an inverse one-step $C$-homotopy between $\phi $ and $e$ (see Fig. %
\ref{pic2}).

A square loop is a loop $\phi :I_{4}^{\ast }\rightarrow G$ such that $%
I_{4}=\left( 0\rightarrow 1\rightarrow 2\leftarrow 3\leftarrow 4\right) .$%
The square loop can be $C$-contracted to $e$ in two steps as is shown on
Fig. \ref{pic3}.\FRAME{ftbpFU}{4.0335in}{1.7599in}{0pt}{\Qcb{A square loop $%
\protect\phi $ is $C$-contractible. Note that $\protect\phi \left( 0\right) =%
\protect\phi \left( 4\right) =\protect\psi \left( 0\right) =\protect\psi %
\left( 2\right) =\ast .$}}{\Qlb{pic3}}{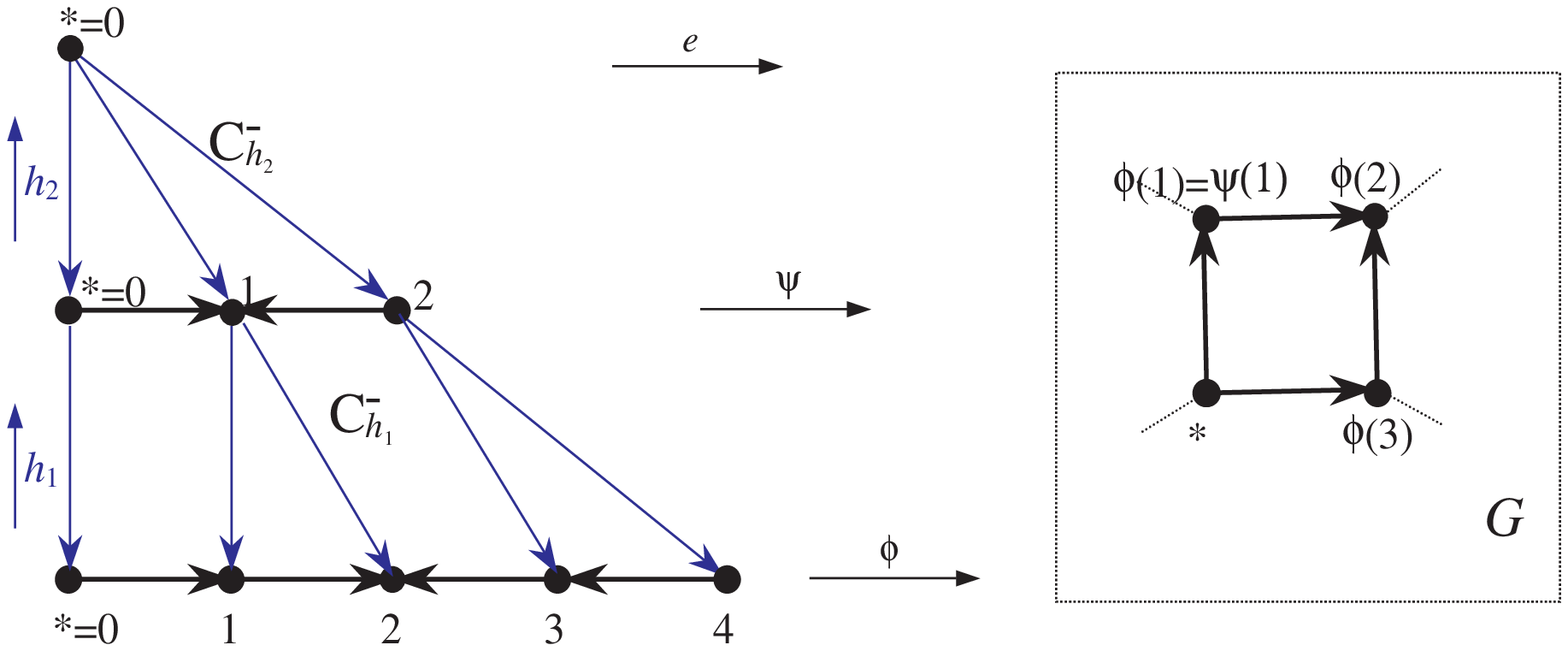}{\special{language
"Scientific Word";type "GRAPHIC";maintain-aspect-ratio TRUE;display
"USEDEF";valid_file "F";width 4.0335in;height 1.7599in;depth
0pt;original-width 7.3215in;original-height 3.1678in;cropleft "0";croptop
"1";cropright "1";cropbottom "0";filename 'pic3.eps';file-properties
"XNPEU";}}

On the other hand, in the case $n\geq 5$, a loop $\phi :I_{n}^{\ast
}\rightarrow G^{\ast }$ does not have to be $C$-contractible, which is the
case, for example, if $\phi $ is the natural map $I_{n}\rightarrow S_{n}$.
\end{example}

\subsection{Local description of $C$-homotopy}

We prove here technical results which has a self-sustained meaning for
practical work with $C$-homotopies.

\begin{lemma}
\label{l5.17}Let $a,b$ be two vertices in a digraph $G$ such that either $%
a=b $ or $a\rightarrow b\rightarrow a$. Then any path-map $\phi
:I_{n}\rightarrow G$, such that $\phi \left( i\right) =a$, $\phi \left(
i+1\right) =b$, and $i\rightarrow i+1$ in $I_{n}$, is $C$-homotopic to a
path-map $\phi ^{\prime }\colon {I}_{n}^{\prime }\rightarrow G$ where $%
I_{n}^{\prime }$ is obtained from $I_{n}$ by changing one edge $i\rightarrow
i+1$ to $i+1\rightarrow i$ and $\phi ^{\prime }\left( j\right) =\phi \left(
j\right) $ for all $j=0,...,n.$
\end{lemma}

\begin{proof}
A $C$-homotopy between $\phi $ and $\phi ^{\prime }$ is constructed in two
one-step inverse $C$-homotopies as is shown on the following diagram:
\begin{equation*}
\begin{matrix}
\phi ^{\prime }:I_{n^{\prime }}\rightarrow G & ... & \underset{a}{i} &
\leftarrow & \underset{b}{i+1} & ... &  \\
&  & \downarrow & \searrow &  & \searrow &  \\
\psi :I_{n+1}\rightarrow G & ... & \underset{a}{i} & \rightarrow & \underset{%
a}{i+1} & \leftarrow & \underset{b}{i+2}... \\
&  & \uparrow &  & \uparrow & \nearrow &  \\
\phi :I_{n}\rightarrow G & ... & \underset{a}{i} & \rightarrow & \underset{b}%
{i+1} & ... &
\end{matrix}%
\
\end{equation*}%
The subscript under each element of the line digraph indicates the value of
the loop on this element.
\end{proof}

Any path-map $\phi \colon I_{n}\rightarrow G$ defines a sequence $\theta
_{\phi }=\left\{ v_{i}\right\} _{i=0}^{n}$ of vertices of $G$ by $v_{i}=\phi
\left( i\right) .$ By definition of a path-map, we have for any $i=0,...,n-1$
one of the following relations:%
\begin{equation*}
v_{i}=v_{i+1},\ \ \ \ v_{i}\rightarrow v_{i+1},\ \ \ \ v_{i+1}\rightarrow
v_{i}.
\end{equation*}%
If $\phi $ is a based path-map, then we have $v_{0}=\ast $, if $\phi $ is a
loop then $v_{0}=\ast =v_{n}$. We consider $\theta _{\phi }$ as a word over
the alphabet $V_{G}$.

\begin{theorem}
\label{p4.12} Two loops $\phi :I_{n}^{\ast }\rightarrow G^{\ast }$ and $\psi
:I_{m}^{\ast }\rightarrow G^{\ast }$ are $C$-homotopic if and only if the
word $\theta _{\psi }$ can be obtained from $\theta _{\phi }$ by a finite
sequence of the following transformations (or inverses to them):

$\left( i\right) $ $...abc...$ $\mapsto $ $...ac...$ where $\left(
a,b,c\right) $ is any permutation of a triple $\left( v,v^{\prime
},v^{\prime \prime }\right) $ of vertices forming a triangle in $G$, that
is, such that $v\rightarrow v^{\prime },v\rightarrow v^{\prime \prime
},v^{\prime }\rightarrow v^{\prime \prime }$ (and the dots \textquotedblleft
$...$\textquotedblright\ denote the unchanged parts of the words).

$\left( ii\right) $ $...abc...\mapsto ...abd...$ where $\left(
a,b,c,d\right) $ is any cyclic permutation (or a cyclic permutation in the
inverse order) of a quadruple $\left( v,v^{\prime },v^{\prime \prime
},v^{\prime \prime \prime }\right) $ of vertices forming a square in $G$,
that is, such that $v\rightarrow v^{\prime },v\rightarrow v^{\prime \prime
\prime },v^{\prime }\rightarrow v^{\prime \prime },v^{\prime \prime \prime
}\rightarrow v^{\prime \prime }$.

$\left( iii\right) $ $...abcd...\mapsto ...ad...\ $where $(a,b,c,d)$ is as
in $\left( ii\right) $.

$\left( iv\right) $ $...aba...\rightarrow ...a...$ if $a\rightarrow b$ or $%
b\rightarrow a.$

$\left( v\right) $ $...aa...\mapsto ...a...$
\end{theorem}

\begin{proof}
Let us first show that if $\theta _{\phi }=\theta _{\psi }$ then $\phi
\overset{C}{\simeq }\psi $. If, for any edge $i\rightarrow i+1$ (or $%
i\leftarrow i+1)$ in $I_{n}$ we have also $i\rightarrow i+1$ (resp. $%
i\leftarrow i+1)$ in $I_{m}$ then $I_{n}=I_{m}$ and $\phi =\psi $ (although $%
n=m,$ the line digraphs $I_{n}$ and $I_{m}$ could a priori be different
elements of $\mathcal{I}_{n}$). Assume that, for some $i$, we have $%
i\rightarrow i+1$ in $I_{n}$ but $i\leftarrow i+1$ in $I_{m}$. Then, by
Lemma \ref{l5.17}, we can change the edge $i\rightarrow i+1$ in $I_{n}$ to $%
i\leftarrow i+1$ while staying in the same $C$-homotopy class of $\phi $.
Arguing by induction, we obtain $\phi \overset{C}{\simeq }\psi .$

We write $\theta _{\phi }\sim \theta _{\psi }$ if $\theta _{\psi }$ can be
obtained from $\theta _{\phi }$ by a finite sequence of transformations $%
\left( i\right) -\left( v\right) $ (or inverses to them). Let us show that $%
\theta _{\phi }\sim \theta _{\psi }$ implies that $\phi \overset{C}{\simeq }%
\psi $. For that we construct for each of the transformations $\left(
i\right) -\left( v\right) $ a $C$-homotopy between $\phi $ and $\psi $. Note
that in this part of the proof $\phi $ and $\psi $ can be arbitrary
path-maps (not necessarily based).

$\left( i\right) $ Assume that $a\rightarrow c$ (the case $c\rightarrow a$
is similar). Then either $b\rightarrow c$ or $a\rightarrow b$ (otherwise we
would have got $a\rightarrow c\rightarrow b\rightarrow a$ which is excluded
by a triangle hypothesis). The $C$-homotopies in the both cases are shown on
the diagram:%
\begin{equation*}
\begin{matrix}
I_{m} & ... & a & \rightarrow & c & ... &  \\
&  & | & \diagdown &  & \diagdown &  \\
I_{n} & ... & a & - & b & \rightarrow & c...%
\end{matrix}%
\ \ \ \ \ \ \ \
\begin{matrix}
I_{m} & ... & a & \rightarrow & c & ... &  \\
&  & | &  & | & \diagdown &  \\
I_{n} & ... & a & \rightarrow & b & - & c...%
\end{matrix}%
\
\end{equation*}%
Each position here corresponds to a vertex in a cylinder $\func{C}_{h}$ or $%
\func{C}_{h}^{-}$ (that is, in $I_{n}$ or $I_{m}$) and shows its image ($a,b$
or $c$) under the map $\phi $ resp. $\psi $. The arrows and undirected
segments shows the edges in the cylinder $\func{C}_{h}$ or $\func{C}_{h}^{-}$
(in particular, horizontal arrows and segments show the edges in $I_{n}$ and
$I_{m}$). The undirected segments, such as $a-b$ and $c-b$, should be given
directions matching those on the digraph $G$.

$\left( ii\right) $ Assume as above $a\rightarrow d$ and $b\rightarrow c$.
Then we have two-step $C$-homotopy as on the diagram:%
\begin{equation*}
\begin{matrix}
&  &  &  &  &  &  &  &  \\
I_{m} & ... & a & \rightarrow & d & - & c & ... &  \\
&  & \uparrow & \nwarrow &  & \nwarrow &  & \nwarrow &  \\
I_{n+1} & ... & a & - & a & \rightarrow & d & - & c... \\
&  & | &  & | &  & | & \diagup &  \\
I_{n} & ... & a & - & b & \rightarrow & c & ... &
\end{matrix}%
\end{equation*}

$\left( iii\right) $ Assume $a\rightarrow d$. Then we have $b\rightarrow c$,
and the $C$-homotopy is shown on the diagram:%
\begin{equation*}
\begin{matrix}
I_{m} & ... &  & ... & a & \rightarrow & d & ... &  \\
&  &  & \diagup & | &  & | & \diagdown &  \\
I_{n} & ... & a & - & b & \rightarrow & c & - & d%
\end{matrix}%
\end{equation*}%
Note that if $a\rightarrow b$ then also $d\rightarrow c$, and if $%
b\rightarrow a$ then also $c\rightarrow d$.

$\left( iv\right) $ Assuming $a\rightarrow b$ we obtain the following $C$%
-homotopy:
\begin{equation*}
\begin{matrix}
I_{m} & ... &  & ... & a & ... &  \\
&  &  & \swarrow & \downarrow & \searrow &  \\
I_{n} & ... & a & \rightarrow & b & \leftarrow & a...%
\end{matrix}%
\end{equation*}

$\left( v\right) $ Here is the required $C$-homotopy:%
\begin{equation*}
\begin{matrix}
I_{m} & ... &  & ... & a & ... \\
&  &  & \nearrow & \uparrow &  \\
I_{n} & ... & a & - & a & ...%
\end{matrix}%
\end{equation*}

Before we go to the second half of the proof, observe that the transformation%
\begin{equation}
...abc...\mapsto ...ac...  \label{abcac}
\end{equation}%
of words is possible not only in the case when $a,b,c$ come from a triangle
as in $\left( i\right) $ but also when \thinspace $a,b,c$ form a \emph{%
degenerate triangle}, that is, when there are identical vertices among $%
a,b,c $ while distinct vertices among $a,b,c$ are connected by an edge.
Indeed, in the case $a=b$ we have by $\left( v\right) $%
\begin{equation*}
abc=aac\sim ac,
\end{equation*}%
in the case $a=c$ we have by $\left( iv\right) $ and $\left( v\right) $%
\begin{equation*}
abc=aba\sim a\sim ac,
\end{equation*}%
and in the case $b=c$ by $\left( v\right) $
\begin{equation*}
abc=acc\sim ac.
\end{equation*}

Now let us prove that $\phi \overset{C}{\simeq }\psi $ implies $\theta
_{\phi }\sim \theta _{\psi }$. It suffices to assume that there exists an
one-step direct $C$-homotopy from $\phi $ to $\psi $ given by a shrinking
map $h:I_{n}^{\ast }\rightarrow I_{m}^{\ast }.$ Set%
\begin{equation*}
\theta _{\phi }=a_{0}a_{1}...a_{n}\ \ \text{and\ \ }\theta _{\psi
}=b_{0}b_{1}...b_{m}
\end{equation*}%
where $a_{i},b_{j}\in V_{G}$ and $a_{0}=b_{0}=a_{n}=b_{m}=\ast $. For any $%
i=0,...,n$ set $j=h\left( i\right) $ and consider two words%
\begin{equation*}
A_{i}=a_{0}a_{1}...a_{i}b_{j}\ \ \ \text{and\ \ \ }B_{i}=b_{0}b_{1}...b_{j}.
\end{equation*}%
We will prove by induction in $i$ that $A_{i}\sim B_{i}$ for all $i=0,...,n$%
. If this is already known, then for $i=n$ we have $j=m$ and%
\begin{equation*}
a_{0}a_{1}...a_{n}b_{m}\sim b_{0}b_{1}...b_{m}.
\end{equation*}%
Since $a_{n}b_{m}=\ast \ast \sim \ast =a_{n}$, it follows that $\theta
_{\phi }\sim \theta _{\psi }$.

Now let us prove that $A_{i}\sim B_{i}$ for all $i=0,...,n$. For $i=0$ we
have $A_{0}=a_{0}b_{0}=\ast \ast \sim \ast =b_{0}=B_{0}.$ Assuming that $%
A_{i}\sim B_{i}$, let us prove that $A_{i+1}\sim B_{i+1}.$ Let us consider a
structure of the cylinder $\func{C}_{h}$ over the edge between $i$ and $i+1$
in $I_{n}.$ Set $h\left( i\right) =j$, $a=a_{i},a^{\prime
}=a_{i+1},b=b_{j},b^{\prime }=b_{j+1}$. There are only the following two
cases:
\begin{equation}
\begin{matrix}
& b &  \\
& \nearrow \ \ \nwarrow &  \\
a & - & a^{\prime }%
\end{matrix}%
\ \ \ \ \text{and\ }\ \
\begin{matrix}
b & - & b^{\prime } \\
\uparrow &  & \uparrow \\
a & - & a^{\prime }%
\end{matrix}%
\ .\ \ \ \   \label{4.271}
\end{equation}%
Note that each arrow on $\func{C}_{h}$ transforms either to an arrow between
the vertices of $G$ or to the identity of the vertices.

Consider first the case of the left diagram in (\ref{4.271}). In this case $%
b^{\prime }=b$ and we obtain by (\ref{abcac}) and by the induction
hypothesis that
\begin{equation*}
A_{i+1}=a_{0}a_{1}...a_{i-1}aa^{\prime }b\sim
a_{0}a_{1}...a_{i-1}ab=A_{i}\sim B_{i}=B_{i+1}.
\end{equation*}%
Consider now the case of the right diagram in (\ref{4.271}) and prove that
in this case
\begin{equation}
aa^{\prime }b^{\prime }\sim abb^{\prime }.  \label{aabb}
\end{equation}%
If (\ref{aabb}) is already known, then we obtain%
\begin{equation*}
A_{i+1}=a_{0}a_{1}...a_{i-1}aa^{\prime }b^{\prime }\sim
a_{0}a_{1}...a_{i-1}abb^{\prime }=A_{i}b^{\prime }\sim B_{i}b^{\prime
}=B_{i+1},
\end{equation*}%
which concludes the induction step in this case.

In order to prove (\ref{aabb}) observe first that if all the vertices $%
a,a^{\prime },b,b^{\prime }$ are distinct, then they form a square and (\ref%
{aabb}) follows by transformation $\left( ii\right) $. In the case $%
a^{\prime }=b$ (\ref{aabb}) is an equality, and in the case $a=b^{\prime }$
the relation (\ref{aabb}) follows by transformation $\left( iv\right) $:%
\begin{equation*}
aa^{\prime }b^{\prime }\sim a=b^{\prime }\sim abb^{\prime }.
\end{equation*}%
In the case $a=b$ the triple $a,a^{\prime },b^{\prime }$ is a triangle or a
degenerate triangle, and we obtained from (\ref{abcac}) and $\left( v\right)
$%
\begin{equation*}
aa^{\prime }b^{\prime }\sim ab^{\prime }\sim aab^{\prime }=abb^{\prime },
\end{equation*}%
and the case $a^{\prime }=b^{\prime }$ is similar. Finally, if $a=a^{\prime
} $ then similarly by $\left( v\right) $ and (\ref{abcac}) we obtain%
\begin{equation*}
aa^{\prime }b^{\prime }=aab^{\prime }\sim ab^{\prime }\sim abb^{\prime },
\end{equation*}%
and the case $b=b^{\prime }$ is similar.
\end{proof}

\begin{remark}
\RM Note that the transformation $\left( iii\right) $ was not used in the
second half of the proof, so $\left( iii\right) $ is logically not necessary
in the statement of Theorem \ref{p4.12}. Note also that $\left( iii\right) $
can be obtained as composition of $\left( ii\right) $ and $\left( iv\right) $
as follows:%
\begin{equation*}
abcd\sim adcd\sim ad.
\end{equation*}%
However, in applications it is still convenient to be able to use $\left(
iii\right) $.
\end{remark}

\begin{example}
\RM A triangular loop on Fig. \ref{pic2} is contractible because if $a,b,c$
are vertices of a triangle then%
\begin{equation*}
abca\sim aca\sim a.
\end{equation*}%
A square loop on Fig. \ref{pic3} is contractible because if $a,b,c,d$ are
vertices of a square then%
\begin{equation*}
abcda\sim ada\sim a.
\end{equation*}

Consider the loops $\phi $ and $\psi $ on Fig. \ref{pic1}, that are known to
be $C$-homotopic. It is shown on Fig. \ref{pic5} how to transform $\theta
_{\phi }$ to $\theta _{\psi }$ using transformations of Theorem \ref{p4.12}.%
\FRAME{ftbpFU}{4.1087in}{2.4267in}{0pt}{\Qcb{Transforming a $5$-cycle $%
\protect\theta _{\protect\phi }$ to a $3$-cycle $\protect\theta _{\protect%
\psi }$ using successively $\left( i\right) ^{-}$ (the inverse of $\left(
i\right) $), $\left( i\right) ,\left( ii\right) $ and $\left( iii\right) .$ }%
}{\Qlb{pic5}}{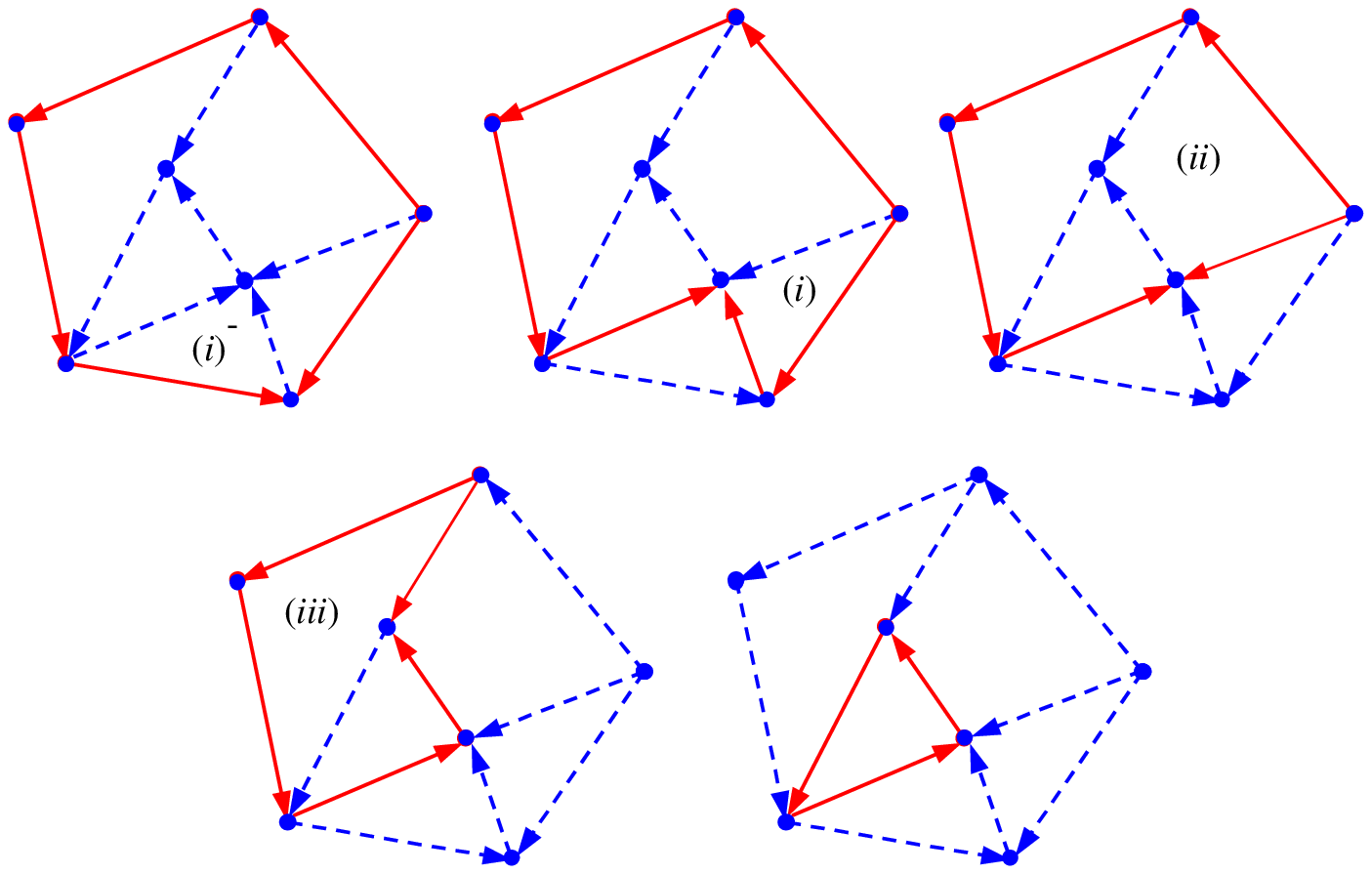}{\special{language "Scientific Word";type
"GRAPHIC";maintain-aspect-ratio TRUE;display "USEDEF";valid_file "F";width
4.1087in;height 2.4267in;depth 0pt;original-width 6.3027in;original-height
3.7118in;cropleft "0";croptop "1";cropright "1";cropbottom "0";filename
'pic5.eps';file-properties "XNPEU";}}
\end{example}

\subsection{Group structure in $\protect\pi _{1}$}

\label{Secpi1}For any $I_{n}\in \mathcal{I}_{n}$ define a line digraph $\hat{%
I}_{n}\in \mathcal{I}_{n}$ as follows:
\begin{equation*}
i\rightarrow j\text{ in }\hat{I}_{n}\ \Leftrightarrow \ (n-i)\rightarrow
(n-j)\text{ in }I_{n}.
\end{equation*}%
For any two line digraphs $I_{n}$ and $I_{m}$, define the line digraph $%
I_{n+m}=I_{n}\vee I_{m}\in \mathcal{I}_{n+m}$ that is obtained from $I_{n}$
and $I_{m}$ by identification of the vertices $n\in I_{n}$ and $0\in I_{m}$.

\begin{definition}
\RM\label{d4.1}$\left( i\right) $ For a path-map $\phi :I_{n}\rightarrow G$
define the \emph{inverse path-map }$\hat{\phi}:\hat{I}_{n}\rightarrow G$ by $%
\hat{\phi}(i)=\phi (n-i).$

$\left( ii\right) $ For two path-maps $\phi :I_{n}\rightarrow G$ and $\psi
:I_{m}\rightarrow G$ with $\phi (n)=\psi (0)$ define the \emph{concatenation
path-map} $\phi \vee \psi :I_{n+m}\rightarrow G$ by%
\begin{equation*}
\phi \vee \psi (i)=%
\begin{cases}
\phi (i), & 0\leq i\leq n \\
\psi (i-n), & n\leq i\leq n+m.%
\end{cases}%
\end{equation*}
\end{definition}

The operation $\phi \mapsto \hat{\phi}$ is evidently an involution on the
set of path-maps. Clearly, if $\phi $ is a loop in $G^{\ast }$ then $\hat{%
\phi}$ is also a loop, and the concatenation of two loops is also a loop.
Let us define a product in $\pi _{1}\left( G^{\ast }\right) $ as follows.

\begin{definition}
\RM For any two loops
\begin{equation*}
\phi \colon I_{n}^{\ast }\rightarrow G^{\ast }\ \ \ \text{and}\ \ \psi
\colon I_{m}^{\ast }\rightarrow G^{\ast }
\end{equation*}%
define the product of $[\phi ]$ and $[\psi ]$ by
\begin{equation}
\lbrack \phi ]\cdot \lbrack \psi ]=[\phi \vee \psi ],  \label{4.19}
\end{equation}%
where $\phi \vee \psi :I_{n+m}^{\ast }\rightarrow G^{\ast }$ is the
concatenation of $\phi $ and $\psi $.
\end{definition}

\begin{lemma}
\label{l4.9}The product in $\pi _{1}\left( G^{\ast }\right) $ is well
defined.
\end{lemma}

\begin{proof}
Let $\phi ,\phi ^{\prime },\psi ,\psi ^{\prime }$ be loops of $G^{\ast }$
and let
\begin{equation}
\phi \overset{C}{\simeq }\phi ^{\prime },\ \ \psi \overset{C}{\simeq }\psi
^{\prime }.  \label{4.20}
\end{equation}%
We must prove that
\begin{equation}
\phi \vee \psi \overset{C}{\simeq }\phi ^{\prime }\vee \psi ^{\prime }.
\label{4.21}
\end{equation}%
It suffices to consider only the case when the both $C$-homotopies in (\ref%
{4.20}) are one-step $C$-homotopies. Then we have
\begin{equation*}
\phi \vee \psi \overset{C}{\simeq }\phi ^{\prime }\vee \psi
\end{equation*}%
because one-step $C$-homotopy between $\phi $ and $\phi ^{\prime }$ easily
extends to that between $\phi \vee \psi $ and $\phi ^{\prime }\vee \psi $.
In the same way we obtain%
\begin{equation*}
\phi ^{\prime }\vee \psi \overset{C}{\simeq }\phi ^{\prime }\vee \psi
^{\prime },
\end{equation*}%
whence (\ref{4.21}) follows.
\end{proof}

\begin{lemma}
\label{l4.10}For any loop $\phi \colon I_{n}^{\ast }\rightarrow G^{\ast }$
we have $\phi \vee \hat{\phi}\overset{C}{\simeq }e$ where $\hat{\phi}$ is
the inverse loop for the loop $\phi $ and
\begin{equation}
e:I_{0}^{\ast }\rightarrow G^{\ast }  \label{e}
\end{equation}
is the trivial loop.
\end{lemma}

\begin{proof}
Let $\theta _{\phi }=v_{0}...v_{n}$. Then $\theta _{\hat{\phi}%
}=v_{n}...v_{0} $ and%
\begin{equation*}
\theta _{\phi \vee \hat{\phi}}=v_{0}...v_{n-1}v_{n}v_{n-1}...v_{0}.
\end{equation*}%
Using successively the transformations $aba\mapsto a$ and $aa\mapsto a$ of
Theorem \ref{p4.12}, we obtain that $\theta _{\phi \vee \hat{\phi}}\sim \ast
$ whence $\phi \vee \hat{\phi}\overset{C}{\simeq }e$ follows.
\end{proof}

\begin{theorem}
\label{t4.11}Let $G,H$ be digraphs.

$\left( i\right) $ The set $\pi _{1}(G^{\ast })$ with the product \emph{(\ref%
{4.19})} and neutral element $[e]$ from \emph{(\ref{e})} is a group. It will
be referred to as the fundamental group of a digraph $G^{\ast }$.

$\left( ii\right) $ A based digraph map $f\colon G^{\ast }\rightarrow
H^{\ast }$ induces a group homomorphism
\begin{equation*}
\pi _{1}(f):\pi _{1}(G^{\ast })\rightarrow \pi _{1}(H^{\ast }),\ \ \left(
\pi _{1}(f)\right) [\phi ]=[f\circ \phi ],
\end{equation*}%
which depends only on homotopy class of $f$. Hence, we obtain a functor from
the category of digraphs $\mathcal{D}^{\ast }$ to the category of groups.

$\left( iii\right) $ Let $\gamma :I_{k}^{\ast }\rightarrow G^{\ast }\ $be a
based path-map with $\gamma (k)=v.$ Then $\gamma $ induces an isomorphism of
fundamental groups
\begin{equation*}
\gamma _{\sharp }\colon \pi _{1}(G^{\ast })\rightarrow \pi _{1}(G^{v}),
\end{equation*}%
which depends only on $C$-homotopy class of the path-map $\gamma $.
\end{theorem}

\begin{proof}
$\left( i\right) $ This follows from Lemmas \ref{l4.9} and \ref{l4.10},
since the product in $\pi _{1}(G^{\ast })$ satisfies the associative law,
the class $[e]\in \pi _{1}(G^{\ast })$ satisfies the definition of a neutral
element, and $[\hat{\phi}]$ is the inverse of $\left[ \phi \right] $ for any
$\left[ \phi \right] \in \pi _{1}\left( G^{\ast }\right) $.

$\left( ii\right) $ Let $\phi $ and $\psi $ be $C$-homotopic loops in $%
G^{\ast }$. It follows from Definition \ref{d4.3} and (\ref{fipsih}) that $%
f\circ \phi \overset{C}{\simeq }f\circ \psi $ and, hence, the map $\pi
_{1}(f)$ is well defined.

The map $\pi _{1}(f)$ is a homomorphism because $\pi _{1}([e])=[e]$ and, for
any two loops $\phi ,\phi ^{\prime }$ in $G^{\ast }$,%
\begin{equation*}
f\circ (\phi \vee \phi ^{\prime })=\left( f\circ \phi \right) \vee \left(
f\circ \phi ^{\prime }\right) .
\end{equation*}%
If $f$ and $g$ two homotopic based maps from $G^{\ast }$ to $H^{\ast }$ then
$f\circ \phi \simeq g\circ \phi $ and hence $f\circ \phi \overset{C}{\simeq }%
g\circ \phi $, which finishes the proof.

$\left( iii\right) $ For any loop $\phi $ in $G^{\ast }$, define a based
loop $\gamma _{\sharp }(\phi )$ in $G^{v}$ by
\begin{equation*}
\gamma _{\sharp }(\phi )=\hat{\gamma}\vee \phi \vee \gamma
:I_{k+n+k}\rightarrow G,
\end{equation*}%
where $\hat{\gamma}$ is the inverse path-map of $\gamma $ as in Definition %
\ref{d4.1}. Similarly to the proof of $\left( ii\right) $ and using Lemma %
\ref{l4.10}, one shows that $\gamma _{\sharp }\colon \pi _{1}(G,\ast
)\rightarrow \pi _{1}(G,v)$ is a group homomorphism. Since $\hat{\gamma}%
_{\sharp }$ is obviously the inverse map of $\gamma _{\sharp }$, it follows
that $\gamma _{\sharp }$ is an isomorphism.

If $\gamma _{1}$ and $\gamma _{2}$ are two $C$-homotopic path-maps
connecting vertices $\ast $ and $v$ then $\hat{\gamma}_{1}\vee \phi \vee
\gamma _{1}$ and $\hat{\gamma}_{2}\vee \phi \vee \gamma _{2}$ are $C$%
-homotopic (cf. the proof of Lemma \ref{l4.9}). Hence, $\gamma _{\sharp }$
depends only on $C$-homotopy class of the map $\gamma $.
\end{proof}

\begin{lemma}
\label{Lemgfg}Let $f:G^{\ast }\rightarrow H^{a}$ and $g:G^{\ast }\rightarrow
H^{b}$ be two based digraphs maps. If $f\simeq g:G\rightarrow H$ then there
exists a based path-map $\gamma :I_{k}^{\ast }\rightarrow H^{a}$ with $%
\gamma \left( k\right) =b$ such that, for any loop $\phi :I_{n}^{\ast
}\rightarrow G^{\ast }$, we have%
\begin{equation}
\gamma _{\sharp }\left( f\circ \phi \right) \overset{C}{\simeq }g\circ \phi .
\label{gfg}
\end{equation}%
Consequently, the following diagram is commutative:%
\begin{equation*}
\begin{array}{ccc}
\pi _{1}\left( G^{\ast }\right) & \overset{\pi _{1}\left( f\right) }{%
\longrightarrow } & \pi _{1}\left( H^{a}\right) \\
\downarrow ^{\func{id}} &  & \ \downarrow ^{\gamma _{_{\sharp }}} \\
\pi _{1}\left( G^{\ast }\right) & \overset{\pi _{1}\left( g\right) }{%
\longrightarrow } & \pi _{1}\left( H^{b}\right)%
\end{array}%
\end{equation*}
\end{lemma}

\begin{proof}
Note that $f\circ \phi $ is a loop in $H^{a}$ and $g\circ \phi $ is a loop
in $H^{b}$. It suffices to prove the statement in the case when $f$ and $g$
are related by an one-step homotopy, that is, $f\left( x\right)
\overrightarrow{=}g\left( x\right) $ for all $x\in V_{G}.$ In particular, we
have $a\overrightarrow{=}b$.

Consider the path-map $\gamma :I\rightarrow H$ given by $\gamma \left(
0\right) =a$ and $\gamma \left( 1\right) =b.$ Then the loop $\gamma _{\sharp
}\left( f\circ \phi \right) :\hat{I}\vee I_{n}\vee I\rightarrow H^{b}$ is
defined by%
\begin{equation*}
\gamma _{\sharp }\left( f\circ \phi \right) =\hat{\gamma}\vee \left( f\circ
\phi \right) \vee \gamma .
\end{equation*}%
Define shrinking $h:$ $\hat{I}\vee I_{n}\vee I\rightarrow I_{n}$ as follows:
$h$ on $I_{n}$ is identical, and the endpoints of $\hat{I}\vee I_{n}\vee I$
are mapped by $h$ to the corresponding endpoints of $I_{n}$:%
\begin{equation*}
\begin{array}{ccccccccc}
I_{n} &  &  & 0 & ... & ... & n &  &  \\
\ \uparrow ^{h} &  & \nearrow & \uparrow & ... & ... & \uparrow & \nwarrow &
\\
\hat{I}\vee I_{n}\vee I & -1 & \leftarrow & 0 & ... & ... & n & \rightarrow
& n+1%
\end{array}%
\end{equation*}%
where for convenience we enumerate the vertices of $\hat{I}\vee I_{n}\vee I$
as $\left\{ -1,0,...,n+1\right\} .$

Then we have, for $0\leq i\leq n,$
\begin{equation*}
\gamma _{\sharp }\left( f\circ \phi \right) \left( i\right) =f\left( \varphi
\left( i\right) \right) \overrightarrow{=}g\left( \phi \left( i\right)
\right) =\left( g\circ \varphi \right) \left( h\left( i\right) \right) ,
\end{equation*}%
for $i=-1$%
\begin{equation*}
\gamma _{\sharp }\left( f\circ \phi \right) \left( -1\right) =b=g\left(
\varphi \left( 0\right) \right) =\left( g\circ \varphi \right) \left(
h\left( -1\right) \right) ,
\end{equation*}%
and for $i=n+1$%
\begin{equation*}
\gamma _{\sharp }\left( f\circ \phi \right) \left( n+1\right) =b=g\left(
\varphi \left( n\right) \right) =\left( g\circ \varphi \right) \left(
h\left( n+1\right) \right) .
\end{equation*}%
Hence, for all $i$,%
\begin{equation*}
\gamma _{\sharp }\left( f\circ \phi \right) \left( i\right) \overrightarrow{=%
}\left( g\circ \varphi \right) \left( h\left( i\right) \right) ,
\end{equation*}%
which implies (\ref{gfg}) by (\ref{fipsih}).
\end{proof}

\begin{theorem}
\label{TGHpi}Let $G,H$ be two connected digraphs. If $G\simeq H$ then the
fundamental groups $\pi _{1}\left( G^{\ast }\right) $ and $\pi _{1}\left(
H^{\ast }\right) $ are isomorphic (for any choice of the based vertices).
\end{theorem}

\begin{proof}
Let $f:G\rightarrow H$ and $g:H\rightarrow G$ be homotopy inverses maps (cf. %
\ref{fg}). Applying Lemma \ref{Lemgfg} to $f\circ g\simeq \func{id}_{G}$ and
to $g\circ f\simeq \func{id}_{H}$, we obtain the result  by a standard
argument (cf. \cite[Ch.1, Thm 8]{Spanier}).
\end{proof}

\subsection{Relation between $H_{1}$ and $\protect\pi _{1}$}

One of our main results is the following theorem.

\begin{theorem}
\label{t4.13}For any based connected digraph $G^{\ast }$ we have an
isomorphism
\begin{equation*}
\pi _{1}(G^{\ast })\left/ [\pi _{1}(G^{\ast }),\pi _{1}(G^{\ast })]\right.
\cong H_{1}(G,\mathbb{Z})
\end{equation*}%
where $[\pi _{1}(G^{\ast }),\pi _{1}(G^{\ast })]$ is a commutator subgroup.
\end{theorem}

\begin{proof}
The proof is similar to that in the classical algebraic topology \cite[p.166]%
{Hatcher}. For any based loop $\phi \colon I_{n}^{\ast }\rightarrow G^{\ast
} $ of a digraph $G^{\ast }$, define a $1$-path $\chi (\phi )$ on $G$ as
follows: $\chi (\phi )=0$ for $n=0,1,2$, and for $n\geq 3$
\begin{equation}
\chi \left( \phi \right) =\sum_{\left\{ i:i\rightarrow i+1\right\} }e_{\phi
(i)\phi (i+1)}-\sum_{\left\{ i:i+1\rightarrow i\right\} }e_{\phi (i+1)\phi
(i)},  \label{4.27}
\end{equation}%
where the summation index $i$ runs from $0$ to $n-1$. It is easy to see that
the $1$-path $\chi \left( \phi \right) $ is allowed and closed and, hence,
determines a homology class $\left[ \chi \left( \phi \right) \right] \in
H_{1}\left( G,\mathbb{Z}\right) $. Let us first prove that, for any two
based loops $\phi \colon I_{n}^{\ast }\rightarrow G^{\ast }$ and $\psi
\colon I_{m}^{\ast }\rightarrow G^{\ast }$,
\begin{equation}
\phi \overset{C}{\simeq }\psi \ \ \Rightarrow \ \left[ \chi (\phi )\right] =%
\left[ \chi (\psi )\right] .  \label{4.28}
\end{equation}%
Note that any based loop with $n\leq 2$ is $C$-homotopic to trivial. For $%
n\geq 3$, it is sufficiently to check (\ref{4.28}) assuming that $\phi
\overset{C}{\simeq }\psi $ is given by an one-step direct $C$-homotopy with
a shrinking map $h:I_{n}^{\ast }\rightarrow I_{m}^{\ast }$. Set
\begin{equation*}
\phi ^{\prime }:=\psi \circ h:I_{n}^{\ast }\rightarrow G^{\ast }
\end{equation*}%
and observe that by (\ref{4.27}) $\chi \left( \phi ^{\prime }\right) =\chi
\left( \psi \right) .$ It remains to show that $\left[ \chi \left( \phi
\right) \right] =\left[ \chi \left( \phi ^{\prime }\right) \right] .$

By Remark \ref{Remfipsi} the digraph maps $\phi $ and $\phi ^{\prime }$,
acting from $I_{n}$ to $G$, are homotopic. Denote by $S_{n}$ the digraph
that is obtained from $I_{n}$ by identification of the vertices $0$ and $n$
(cf. Example \ref{e2.8}). Then $\varphi $ and $\phi ^{\prime }$ can be
regarded as digraph maps from $S_{n}$ to $G$, and they are again homotopic
as such.

Consider the standard homology class $\left[ \varpi \right] \in H_{1}\left(
S_{n}\right) $ given by (\ref{om}). Comparing (\ref{om}) and (\ref{4.27}),
we see that%
\begin{equation*}
\phi _{\ast }\left( \varpi \right) =\chi \left( \varphi \right) \ \ \text{%
and\ \ }\phi _{\ast }^{\prime }\left( \varpi \right) =\chi \left( \phi
^{\prime }\right) .
\end{equation*}%
On the other hand, by Theorem \ref{t3.4} we have $\left[ \phi _{\ast }\left(
\varpi \right) \right] =\left[ \phi _{\ast }^{\prime }\left( \varpi \right) %
\right] $, which finishes the proof of (\ref{4.28}).

Hence, $\chi $ determines a map
\begin{equation*}
\chi _{\ast }\colon \pi _{1}(G^{\ast })\rightarrow H_{1}(G,\mathbb{Z}),\ \ \
\chi _{\ast }[\phi ]=[\chi (\phi )].
\end{equation*}%
The map $\chi _{\ast }$ is a group homomorphism because, for based loops $%
\phi ,\psi $ and the neutral element $[e]\in \pi _{1}\left( G^{\ast }\right)
$, we have $\chi _{\ast }([e])=0$ and
\begin{eqnarray*}
\chi _{\ast }([\phi ]\cdot \lbrack \psi ]) &=&\chi _{\ast }([\phi \vee \psi
])=[\chi (\phi \vee \psi )] \\
&=&[\chi (\phi )+\chi (\psi )]=[\chi (\phi )]+[\chi (\psi )]=\chi _{\ast
}([\phi ])+\chi _{\ast }([\psi ]).
\end{eqnarray*}%
Since the group $H_{1}(G,\mathbb{Z})$ is abelian, it follows that
\begin{equation*}
\lbrack \pi _{1}(G^{\ast }),\pi _{1}(G^{\ast })]\subset \func{Ker}\chi
_{\ast }.
\end{equation*}

Now let us prove that $\chi _{\ast }$ is an epimorphism. Define a \emph{%
standard loop} on $G$ as a finite sequence $v=\left\{ v_{k}\right\}
_{k=0}^{n}$ of vertices of $G$ such that $v_{0}=v_{n}$ and, for any $k$ $%
=0,...,n-1$, either $v_{k}\rightarrow v_{k+1}$ or $v_{k+1}\rightarrow v_{k}.$
For a standard loop $v$ define an $1$-path%
\begin{equation}
\varpi _{v}=\sum_{\left\{ k:v_{k}\rightarrow v_{k+1}\right\}
}e_{v_{k}v_{k+1}}-\sum_{\left\{ k:v_{k+1}\rightarrow v_{k}\right\}
}e_{v_{k}v_{k+1}}  \label{omv}
\end{equation}%
and observe that $\varpi _{v}$ is allowed and closed. The $1$-paths of the
form (\ref{omv}) will be referred to as standard paths. Consider an
arbitrary closed $1$-path
\begin{equation*}
w=\sum_{k}n_{k}e_{i_{k}j_{k}}\in \Omega _{1}(G,\mathbb{Z}).
\end{equation*}%
Since $\partial w=0$ and $\partial e_{ij}=e_{j}-e_{i}$, the path $w$ can be
represented as a finite sum of standard paths. Hence, in order to prove that
$\chi _{\ast }$ is an epimorphism, it suffices to show that any standard $1$%
-path $\varpi _{v}$ is in the image of $\chi $. Note that the standard loop $%
v$ determines naturally a based loop $\phi :I_{n}^{\ast }\rightarrow
G^{v_{0}}$ by $\phi \left( i\right) =v_{i}$. Since the digraph $G$ is
connected, there exists a based path $f:I_{s}^{\ast }\rightarrow G^{\ast }$
with $f(s)=v_{0}$. Thus we obtain a based loop
\begin{equation*}
f\vee \phi \vee \hat{f}:I_{2s+n}^{\ast }\rightarrow G^{\ast }.
\end{equation*}%
It follows directly from our construction, that $\chi (f\vee \phi \vee \hat{f%
})=\varpi _{v}$, and hence $\chi _{\ast }$ is an epimorphism.

We are left to prove that
\begin{equation*}
\func{Ker}\chi _{\ast }\subset \lbrack \pi _{1}(G^{\ast }),\pi _{1}(G^{\ast
})].
\end{equation*}%
For that we need to prove that, for any loop $\phi :I_{n}^{\ast }\rightarrow
G^{\ast }$, if $\chi _{\ast }([\phi ])=0\in H_{1}(G,\mathbb{Z})$, then $%
\left[ \phi \right] $ lies in the commutator $[\pi _{1}(G^{\ast }),\pi
_{1}(G^{\ast })]$. In the case $n\leq 2$ any loop $\phi $ is $C$-homotopic
to the trivial loop. Assuming in the sequel $n\geq 3$, we use the word $%
\theta _{\phi }=v_{0}v_{1}...v_{n}$ where $v_{i}=\phi \left( i\right) $.

Consider first the case, when $\chi (\phi )=0\in \Omega _{1}(G)$. Since the
digraph $G$ is connected, for any vertex $v_{i}$ there exists a based
path-map $\psi _{i}\colon I_{p_{i}}^{\ast }\rightarrow G^{\ast }\ \ \text{%
with}\ \ \psi _{i}(p_{i})=v_{i}$. If $v_{i}=v_{j}$ for some $i,j$ then we
make sure to choose $\psi _{i}$ and $\psi _{j}$ identical. For $i=0$ and $%
i=n $ choose $\psi _{i}$ to be trivial path-map $e:I_{0}^{\ast }\rightarrow
G^{\ast }$. For any $i=0,...,n-1$ define path-map $\phi _{i}\colon I^{\pm
}\rightarrow G$ by the conditions $\phi _{i}(0)=v_{i},\phi _{i}(1)=v_{i+1}$
and consider the following loop%
\begin{equation}
\gamma =\psi _{0}\vee \phi _{0}\vee \hat{\psi}_{1}\vee \psi _{1}\vee \phi
_{1}\vee \hat{\psi}_{2}\vee \psi _{2}\vee \phi _{2}\vee \dots \vee \hat{\psi}%
_{n-1}\vee \psi _{n-1}\vee \phi _{n-1}\vee \psi _{n}  \label{4.29}
\end{equation}%
(see Fig. \ref{pic12}).\FRAME{ftbpFU}{3.0822in}{1.938in}{0pt}{\Qcb{Loop $%
\protect\psi _{i}\vee \protect\phi _{i}\vee \hat{\protect\psi}_{i+1}$}}{\Qlb{%
pic12}}{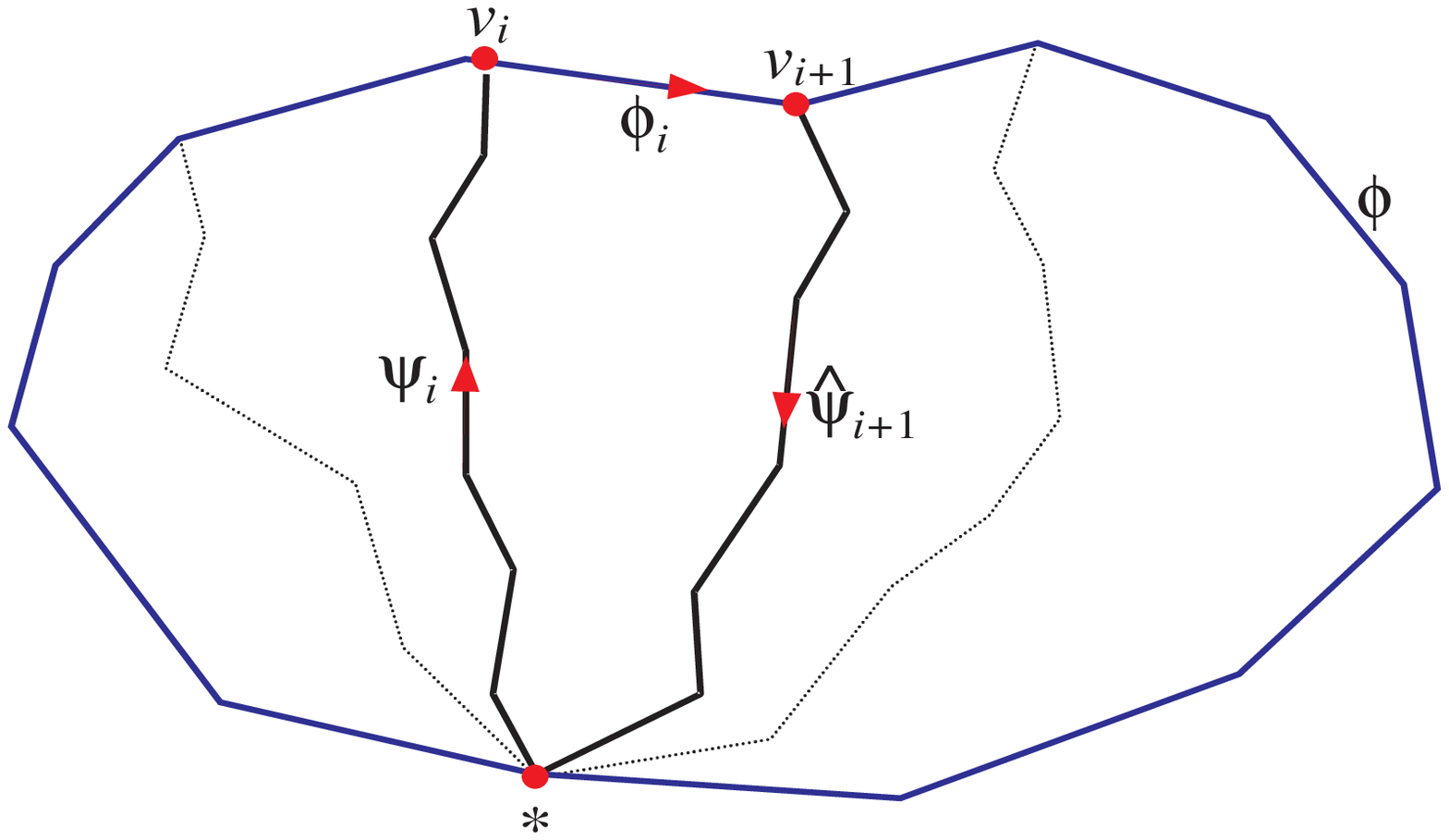}{\special{language "Scientific Word";type
"GRAPHIC";maintain-aspect-ratio TRUE;display "USEDEF";valid_file "F";width
3.0822in;height 1.938in;depth 0pt;original-width 6.3607in;original-height
3.9894in;cropleft "0";croptop "1";cropright "1";cropbottom "0";filename
'pic12.eps';file-properties "XNPEU";}}

Using transformation $\left( iv\right) $ of Theorem \ref{p4.12} (similarly
to the proof of Lemma \ref{l4.10}), we obtain that
\begin{equation*}
\gamma \overset{C}{\simeq }\phi _{0}\vee \phi _{1}\vee ...\vee \phi
_{n-1}=\phi .
\end{equation*}
On the other hand, it follows from (\ref{4.29}) that%
\begin{equation*}
\left[ \gamma \right] =\prod_{i=0}^{n-1}\left[ \psi _{i}\vee \phi _{i}\vee
\hat{\psi}_{i+1}\right]
\end{equation*}%
Consider for some $i=0,...,n-1$, such that $i\rightarrow i+1$, the vertices $%
a=v_{i}$ and $b=v_{i+1}$. If $a=b$ then the loop $\psi _{i}\vee \phi
_{i}\vee \hat{\psi}_{i+1}$ is $C$-homotopic to $e$. Assume $a\neq b$, so
that $a\rightarrow b$. Then the term $e_{ab}$ is present in the right hand
side of the identity (\ref{4.27}) defining $\chi \left( \phi \right) $. Due
to $\chi \left( \phi \right) =0$, the term $e_{ab}$ should cancel out with $%
-e_{ab}$ in the right hand side of (\ref{4.27}). Therefore, there exists $%
j=0,...,n-1$ such that $j+1\rightarrow j$, $v_{j+1}=a$ and $v_{j}=b$. It
follows that
\begin{equation*}
\psi _{j}\vee \phi _{j}\vee \hat{\psi}_{j+1}=\psi _{i+1}\vee \hat{\phi}%
_{i}\vee \hat{\psi}_{i},
\end{equation*}%
and that the loops%
\begin{equation}
\left[ \psi _{i}\vee \phi _{i}\vee \hat{\psi}_{i+1}\right] \text{ and\ }%
\left[ \psi _{j}\vee \phi _{j}\vee \hat{\psi}_{j+1}\right]  \label{4.30}
\end{equation}%
are mutually inverse. Therefore, $\left[ \gamma \right] $ is a product of
pairs of mutually inverse loops, which implies that $\left[ \gamma \right] =%
\left[ \phi \right] $ lies in the commutator of $\pi _{1}$.

Now consider the general case, when $\chi \left( \phi \right) \in \Omega
_{1}\left( G\right) $ is exact, that is, $\chi (\phi )=\partial \omega $ for
some$\ \omega \in \Omega _{2}(G).$ Recall that by Proposition \ref{pijk} any
$2$-path $\omega \in \Omega _{2}$ can be represented in the form%
\begin{equation*}
\omega =\sum_{j=1}^{N}\kappa _{j}\sigma _{j}
\end{equation*}%
where $N\in \mathbb{N}$, $\kappa _{l}=\pm 1$ and $\sigma _{l}$ is one of the
following $2$-paths: a double edge, a triangle, a square. Further proof goes
by induction in $N$. In the case $N=0$ we have $\omega =0$ which was already
considered above.

In the case $N\geq 1$ choose an arbitrary index $i=0,...,n-1$ such that the
vertices $a=\phi \left( i\right) $ and $b=\phi \left( i+1\right) $ are
distinct. Assume for certainty that $i\rightarrow i+1$ and, hence, $%
a\rightarrow b$ (the case $i+1\rightarrow i$ can be handled similarly). Then
$e_{ab}$ enters $\chi \left( \phi \right) $ with the coefficient $1$. Since%
\begin{equation*}
\chi \left( \phi \right) =\partial \omega =\sum_{j=1}^{N}\kappa _{j}\partial
\sigma _{j},
\end{equation*}%
there exists $\sigma _{l}$ such that $\partial \sigma _{l}$ contains a term $%
\kappa _{l}e_{ab}$. Fix this $l$ and define a new loop $\phi ^{\prime }$ as
follows.

If $\sigma _{l}$ is a double edge $a,b,a$, then consider a loop $\phi
^{\prime }$ that is obtained from $\phi :I_{n}^{\ast }\rightarrow G^{\ast }$
by changing one edge $i\rightarrow i+1$ in $I_{n}$ to $i\rightarrow i+1$.
Then by Lemma \ref{l5.17} we have $\phi ^{\prime }\overset{C}{\simeq }\phi $.

Let $\sigma _{l}$ be a triangle with the vertices $a,b,c$. Noticing that%
\begin{equation*}
\theta _{\phi }=...ab...
\end{equation*}%
consider a loop $\phi ^{\prime }$ such that%
\begin{equation*}
\theta _{\phi ^{\prime }}=...acb...
\end{equation*}%
(see Fig. \ref{pic13}).\FRAME{ftbpFU}{2.4915in}{1.7798in}{0pt}{\Qcb{Loops $%
\protect\phi $ and $\protect\phi ^{\prime }$ in the case when $\protect%
\sigma _{l}$ is a triangle.}}{\Qlb{pic13}}{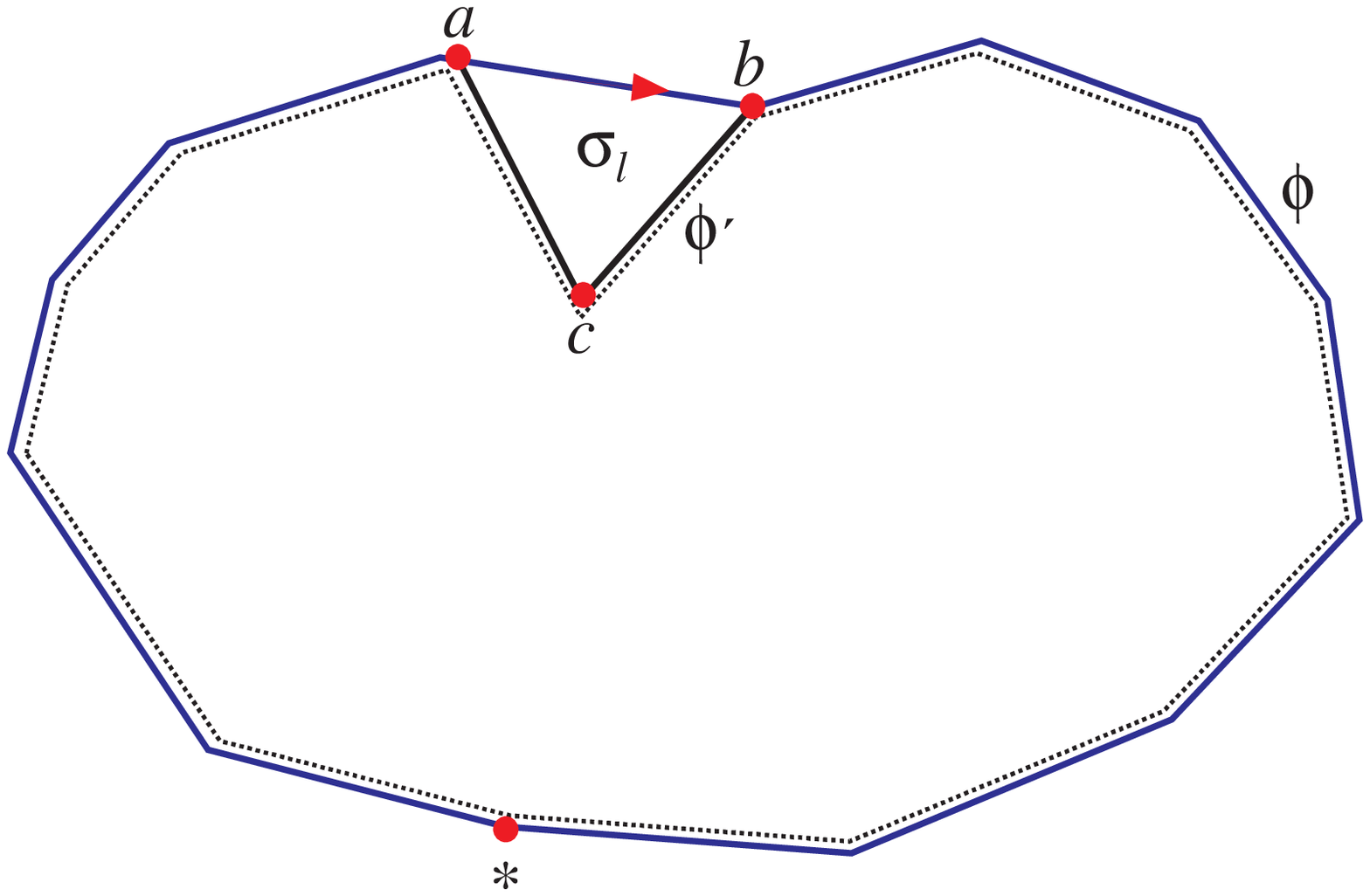}{\special{language
"Scientific Word";type "GRAPHIC";maintain-aspect-ratio TRUE;display
"USEDEF";valid_file "F";width 2.4915in;height 1.7798in;depth
0pt;original-width 6.3183in;original-height 4.5031in;cropleft "0";croptop
"1";cropright "1";cropbottom "0";filename 'pic13.eps';file-properties
"XNPEU";}}

If $\sigma _{l}$ is a square with the vertices $a,b,c,d$, then we define a
loop $\phi ^{\prime }$ so that%
\begin{equation*}
\theta _{\phi ^{\prime }}=...adcb.
\end{equation*}%
By Theorem \ref{p4.12}, we have in the both cases $\phi ^{\prime }\overset{C}%
{\simeq }\phi $ and, hence, $\left[ \phi ^{\prime }\right] =\left[ \phi %
\right] $.

By construction, $\chi \left( \phi ^{\prime }\right) $ contains no longer
the term $e_{ab}$. On the other hand, we will prove below that, for some $%
\kappa =\pm 1$,
\begin{equation}
\chi \left( \phi ^{\prime }\right) =\chi \left( \phi \right) -\kappa
\partial \sigma _{l}.  \label{xiab}
\end{equation}%
Comparing the coefficients in front of $e_{ab}$ in the both parts of (\ref%
{xiab}), we obtain the identity $0=1-\kappa \kappa _{l}$ whence $\kappa
=\kappa _{l}$. It follows from (\ref{xiab}) with $\kappa =\kappa _{l}$ that%
\begin{equation*}
\chi \left( \phi ^{\prime }\right) =\chi \left( \phi \right) -\partial
\left( \kappa _{l}\sigma _{l}\right) =\partial \omega -\partial \left(
\kappa _{l}\sigma _{l}\right) =\partial \omega ^{\prime },
\end{equation*}%
where%
\begin{equation*}
\omega ^{\prime }=\sum_{j\neq l}c_{j}\sigma _{j}.
\end{equation*}%
By the inductive hypothesis we conclude that $\left[ \phi ^{\prime }\right] $
lies in the commutator $[\pi _{1}(G^{\ast }),\pi _{1}(G^{\ast })]$, whence
the same for $\left[ \phi \right] $ follows.

We are left to prove the identity (\ref{xiab}). If $\sigma _{l}$ is a double
edge $a,b,a$ then%
\begin{equation*}
\chi \left( \phi ^{\prime }\right) -\chi \left( \phi \right)
=-e_{ba}-e_{ab}=-\partial e_{aba}=-\partial \sigma _{l}.
\end{equation*}%
If $\sigma _{l}$ is a triangle
\begin{equation*}
\begin{matrix}
& c &  \\
& \diagup \ \ \diagdown &  \\
a & \longrightarrow & b%
\end{matrix}%
\end{equation*}%
then we obtain a cycle digraph $S_{3}$ with the vertices $a,b,c$, and if $%
\sigma _{l}$ is a square
\begin{equation*}
\begin{array}{ccc}
d & \longrightarrow & c \\
\ | &  & |\  \\
a & \longrightarrow & b%
\end{array}%
\end{equation*}%
then we obtain a cycle digraph $S_{4}$ with the vertices $a,b,c,d.$ Let $%
\varpi $ be the standard $1$-path on $S_{3}$ in the first case and that on $%
S_{4}$ in the second case (see (\ref{om})). Then it is easy to see that%
\begin{equation*}
\chi \left( \phi \right) -\chi \left( \phi ^{\prime }\right) =\varpi ,
\end{equation*}%
and (\ref{xiab}) follows from the observation that $\partial \sigma _{l}=\pm
\varpi $ (cf. Example \ref{e2.8}).
\end{proof}

\subsection{Higher homotopy groups}

\label{SecHigher}Recall that, for any based digraph $G^{\ast }$, a based
loop-digraph $LG^{\ast }$ was defined in Definition \ref{d4.4}, and, for a
digraph map $f:G^{\ast }\rightarrow H^{\ast },$ we defined a digraph map $%
Lf:LG^{\ast }\rightarrow LH^{\ast }$ by (\ref{Lf}).

\begin{definition}
\label{d4.14}\RM For any digraph $G^{\ast }$ let $L^{n}G=L^{n}G^{\ast
},n=0,1,2,3,\dots $ be based digraphs defined inductively as
\begin{equation*}
L^{0}G^{\ast }=G^{\ast },\ \ L^{1}G^{\ast }=LG^{\ast },\ \ \text{ and, for $%
n\geq 2$, }\ \ \ L^{n}G^{\ast }\overset{def}{=}L\left( L^{n-1}G^{\ast
}\right)
\end{equation*}%
where the base point in $LG^{\ast }$ is the based map $I_{0}^{\ast
}\rightarrow G^{\ast }$ which we also denote by $\ast $.

For $n\geq 2$, define \emph{homotopy group }$\pi _{n}(G^{\ast })$ of the
digraph $G^{\ast }$ inductively by
\begin{equation*}
\pi _{n}(G^{\ast })=\pi _{n-1}(LG^{\ast }).
\end{equation*}
\end{definition}

\begin{theorem}
\label{t4.15}Let $G^{\ast },H^{\ast }$ be two based digraphs. If $f$ and $g~$%
are homotopic digraph maps $G^{\ast }\rightarrow H^{\ast }$ then $Lf$ and $%
Lg $ are homotopic digraph maps $LG^{\ast }\rightarrow LH^{\ast }$. If $%
G^{\ast }\simeq H^{\ast }$ then also $LG^{\ast }\simeq LH^{\ast }$.
\end{theorem}

\begin{proof}
In the first statement, it suffices to consider the case of one-step
homotopy between $f$ and $g$, which by (\ref{f=g}) amounts to either $%
f\left( x\right) \overrightarrow{=}g\left( x\right) \ $for\ all\ $x\in
V_{G}\ $or\ \ $g\left( x\right) \overrightarrow{=}f\left( x\right) \ $for\
all\ $x\in V_{G}$. Assume without loss of generality that
\begin{equation*}
f\left( x\right) \overrightarrow{=}g\left( x\right) \ \text{for\ all}\ x\in
V_{G}.
\end{equation*}%
Then, for any loop $\psi \in LG^{\ast }$, $\psi :I_{n}^{\ast }\rightarrow
G^{\ast }$, we have also%
\begin{equation*}
f\left( \psi \left( i\right) \right) \overrightarrow{=}g\left( \psi \left(
i\right) \right) \ \text{for all }i=0,...,n,
\end{equation*}%
which implies that $f\circ \psi $ and $g\circ \psi $ are one-step homotopic
and, hence, one-step $C$-homotopic. Therefore, the loops $f\circ \psi $ and $%
g\circ \psi $ as elements of $LH^{\ast }$ are either identical or connected
by an edge in $LH^{\ast }$, that is
\begin{equation*}
\left( Lf\right) \left( \psi \right) \overrightarrow{=}\left( Lg\right)
\left( \psi \right) \ \ \text{for all }\psi \in V_{LG}.
\end{equation*}%
Hence, $Lf\simeq Lg$, which finishes the proof of the first statement.

Since $L$ is a functor we obtain the proof of the rest part of the Theorem.
\end{proof}

\begin{corollary}
\label{c4.16} For $n\geq 0$, the functor $\pi _{n}$ is well defined on the
homotopy category of based digraphs.
\end{corollary}

\begin{remark}
\RM The definition of higher homotopy groups $\pi _{n}\left( G^{\ast
}\right) $ depends crucially on how we define edges in the loop-digraph $%
LG^{\ast }.$ Our present definition uses for that one-step $C$-homotopy.
There may be other definitions of edges in $LG^{\ast }$, for example, one
could use for that the transformations of Theorem \ref{p4.12}. By switching
to the latter (or any other reasonable) definition of $LG^{\ast }$, the set
of connected components of $LG^{\ast }$ remains unchanged, so that $\pi
_{1}\left( G^{\ast }\right) =\pi _{0}\left( LG^{\ast }\right) $ is
unchanged, but $\pi _{1}\left( LG^{\ast }\right) $ and, hence, $\pi
_{2}\left( G^{\ast }\right) $ may become different. At present it is not
quite clear what is the most natural choice of edges in $LG^{\ast }$. We
plan to return to this question in the future research.
\end{remark}

\section{Application to graph coloring}

\label{SecSperner}\setcounter{equation}{0}An an illustration of the theory
of digraph homotopy, we give here a new proof of the classical lemma of
Sperner, using the notion the fundamental group and $C$-homotopy.

Consider a triangle $ABC$ on the plane $\mathbb{R}^{2}$ and its
triangulation $T$. The set of vertices of $T$ is colored with three colors $%
1,2,3$ in such a way that

\begin{itemize}
\item the vertices $A,B,C$ are colored with $1,2,3$ respectively;

\item each vertex on any side of $ABC$ is colored with one of the two colors
of the endpoints of the side (see Fig. \ref{pic7}).
\end{itemize}

\FRAME{ftbpFU}{4.2073in}{1.8498in}{0pt}{\Qcb{A Sperner coloring}}{\Qlb{pic7}%
}{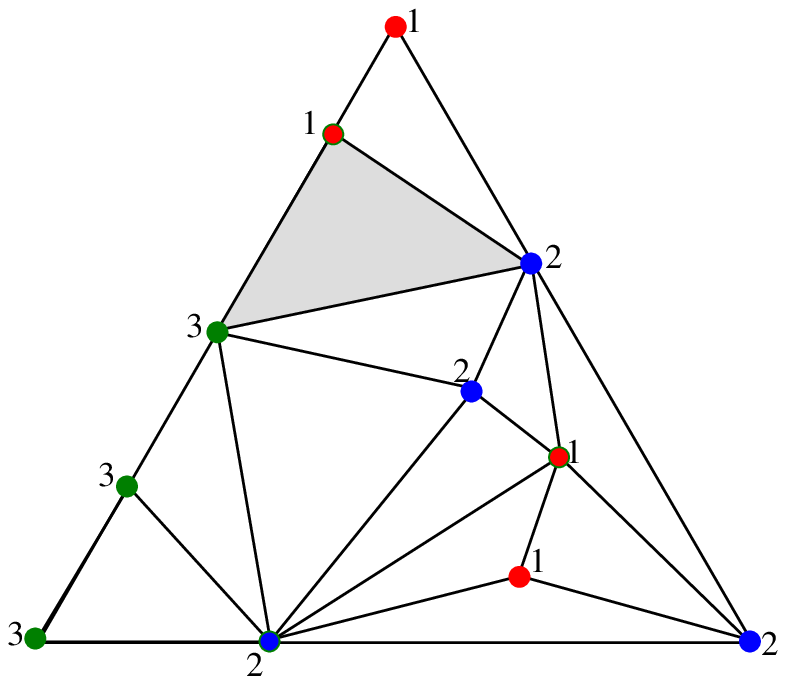}{\special{language "Scientific Word";type
"GRAPHIC";maintain-aspect-ratio TRUE;display "USEDEF";valid_file "F";width
4.2073in;height 1.8498in;depth 0pt;original-width 6.3027in;original-height
2.7492in;cropleft "0";croptop "1";cropright "1";cropbottom "0";filename
'pic7.eps';file-properties "XNPEU";}}

The classical lemma of Sperner says that then there exists in $T$ a $3$%
-color triangle, that is, a triangle, whose vertices are colored with the
three different colors.

To prove this, let us first modify the triangulation $T$ so that there are
no vertices on the sides $AB,AC,BC$ except for $A,B,C.$ Indeed, if $X$ is a
vertex on $AB$ then we move $X$ a bit inside the triangle $ABC.$ This gives
rise to a new triangle in the triangulation $T$ that is formed by $X$ and
its former neighbors, say $Y$ and $Z$, on the edge $AB$ (while keeping all
other triangles). However, since all $X,Y,Z$ are colored with two colors, no
$3$-color triangle emerges after that move. By induction, we remove all the
vertices from the sides of $ABC.$

The triangulation $T$ can be regarded as a graph. Let us make it into a
digraph $G$ by choosing the direction on the edges as follows. If the
vertices $a,b$ are connected by an edge in $T$ then choose direction between
$a,b$ using the colors of $a,b$ and the following rule:%
\begin{equation}
\begin{array}{ccc}
1\rightarrow 2, & 2\rightarrow 3, & 3\rightarrow 1 \\
1\leftrightarrows 1, & 2\leftrightarrows 2, & 3\leftrightarrows 3%
\end{array}
\label{123}
\end{equation}%
Assume now that there is no $3$-color triangle in $T.$ Then each triangle
from $T$ looks in $G$ like
\begin{equation*}
\begin{matrix}
& \bullet  &  \\
& \nearrow \ \ \nwarrow  &  \\
\bullet  & \leftrightarrows  & \bullet
\end{matrix}%
\ \ \text{or\ \ \ }%
\begin{matrix}
& \bullet  &  \\
& \swarrow \ \ \searrow  &  \\
\bullet  & \leftrightarrows  & \bullet
\end{matrix}%
\ \ \text{or\ \ }%
\begin{matrix}
& \bullet  &  \\
& \nearrow \swarrow \ \ \searrow \nwarrow  &  \\
\bullet  & \leftrightarrows  & \bullet
\end{matrix}%
,
\end{equation*}%
in particular, each of them contains a triangle in the sense of Theorem \ref%
{p4.12}. Using the transformations $\left( ii\right) $ and $\left( iv\right)
$ of Theorem \ref{p4.12} and the partition of $G$ into the triangles, we
contract any loop on $G$ to an empty word (cf. Fig. \ref{pic13}), whence $%
\pi _{1}\left( G^{\ast }\right) =\left\{ 0\right\} $.

Consider now a colored cycle $S_{3}$
\begin{equation}
\begin{matrix}
& 1 &  \\
& \nearrow \ \ \searrow  &  \\
3 & \longleftarrow  & 2%
\end{matrix}
\label{T123}
\end{equation}%
and the following two maps: $f:G\rightarrow S_{3}$ that preserves the colors
of the vertices and $g:S_{3}\rightarrow G$ that maps the vertices $1,2,3$ of
$S_{3}$ onto $A,B,C,$ respectively. Both $f,g$ are digraph maps, which for
the case of $f$ follows from the choice (\ref{123}) of directions of the
edges of $G$. Since $f\circ g=\func{id}_{S_{3}}$, we obtain that $\pi
_{1}\left( f\circ g\right) =\pi _{1}\left( f\right) \circ \pi _{1}\left(
g\right) $ is an isomorphism of $\pi _{1}\left( S_{3}\right) \simeq \mathbb{Z%
}$ onto itself, which is not possible by $\pi _{1}\left( G^{\ast }\right)
=\left\{ 0\right\} .$

\section{Homology and homotopy of (undirected) graphs}

\label{S5}\setcounter{equation}{0}

A homotopy theory of undirected graphs was constructed in \cite{Babson} and
\cite{Barcelo} (see also \cite{Dochtermann}). Here we show that this theory
can be obtained from our homotopy theory of digraphs as restriction to a
full subcategory. The same restriction enables us to define a homotopy
invariant homology theory of undirected graphs such that the classical
relation between fundamental group and the first homology group given by
Theorem \ref{t4.13} is preserved. In particular, the so obtained homology
theory for graphs answers a question raised in \cite[p.32]{Babson}.

To distinguish digraphs (see Definition \ref{d2.1}) and (undirected) graphs
(see Definition \ref{d3.1} below) we use the following notations. To denote
a digraph and its sets of vertices and edges, we use as in the previous
sections the standard font as $G=(V_{G},E_{G}).$ To denote a graph and its
sets of vertices and edges, we will use a bold font, for example, $\mathbf{G}%
=(\mathbf{V}_{\mathbf{G}},\mathbf{E}_{\mathbf{G}})$. The bold font will also
be used to denoted the maps between graphs.

\begin{definition}
\label{d5.1}\RM$\left( i\right) $ A\emph{\ graph } $\mathbf{G}=(\mathbf{V}_{%
\mathbf{G}},\mathbf{E}_{\mathbf{G}})$ is a couple of a set $\mathbf{V}_{%
\mathbf{G}}$ of \emph{vertices} and a subset $\mathbf{E}_{\mathbf{G}}\subset
\{\mathbf{V}_{\mathbf{G}}\times \mathbf{V}_{\mathbf{G}}\setminus {\func{diag}%
}\}$ of non-ordered pairs of vertices that are called \emph{edges}. Any edge
$\left( v,w\right) \in \mathbf{E}_{\mathbf{G}}$ will be also denoted by $%
v\sim w$.

$\left( ii\right) $ A\emph{\ morphism }from a graph\emph{\ }$\mathbf{G}%
=\left( \mathbf{V}_{\mathbf{G}},\mathbf{E}_{\mathbf{G}}\right) $\emph{\ }to
a graph\emph{\ }$\mathbf{H}=\left( \mathbf{V}_{\mathbf{H}},\mathbf{E}_{%
\mathbf{H}}\right) $ is a map
\begin{equation*}
\mathbf{f}\colon \mathbf{V}_{\mathbf{G}}\rightarrow \mathbf{V}_{\mathbf{H}}
\end{equation*}%
such that for any edge $v\sim w$ on $\mathbf{G}$ we have either $\mathbf{f}%
\left( v\right) =\mathbf{f}\left( w\right) $ or $\mathbf{f}\left( v\right)
\sim \mathbf{f}\left( w\right) $. We will refer to morphisms of graphs as
\emph{\ graph maps}.
\end{definition}

To each graph $\mathbf{G}=(\mathbf{V}_{\mathbf{G}},\mathbf{E}_{\mathbf{G}})$
we associate a digraph\emph{\ }$G=(V_{G},E_{G})$ where $V_{G}=\mathbf{V}_{%
\mathbf{G}}$ and $E_{G}$ is defined by the condition $v\rightarrow
w\Leftrightarrow v\sim w.$ Clearly, the digraph $G$ satisfies the condition $%
w\rightarrow v\Leftrightarrow v\rightarrow w$. Any digraph with this
property will be called a \emph{double} digraph.

The set of all graphs with graph maps forms a category (which was also
introduced by \cite{Babson} and \cite{Barcelo}), that will be denoted by $%
\mathcal{G}$.

The assignment $\mathbf{G}\mapsto G$ and a similar assignment $\mathbf{f}%
\mapsto f$ of maps, that is well defined, provide a functor $\mathcal{O}$
from $\mathcal{G}$ to $\mathcal{D}$. It is clear that the image $\mathcal{O}$
is a full subcategory $\mathcal{O}(\mathcal{G})$ of $\mathcal{D}$ that
consists of double digraphs, such that the inverse functor $\mathcal{O}%
^{-1}\colon \mathcal{O}(\mathcal{G})\rightarrow \mathcal{G}$ is well defined.

\begin{definition}
\label{d2.3g}\RM For two graphs $\mathbf{G}=(\mathbf{V}_{\mathbf{G}},\mathbf{%
E}_{\mathbf{G}})$ and $\mathbf{H}=\mathbf{(V}_{\mathbf{H}}\mathbf{,E}_{%
\mathbf{H}}\mathbf{)}$ define the \emph{Cartesian product} $\mathbf{G}%
\boxdot \mathbf{H}$ as a digraph with the set of vertices $\mathbf{V}_{%
\mathbf{G}}\times \mathbf{V}_{\mathbf{H}}$ and with the set of edges as
follows: for $x,x^{\prime }\in \mathbf{V}_{\mathbf{G}}\ $and$\ \ y,y^{\prime
}\in \mathbf{V}_{\mathbf{H}}$, we have $(x,y)\sim (x^{\prime },y^{\prime })\
$in $\mathbf{G}\boxdot \mathbf{H}$ if and only if
\begin{equation*}
\text{either}\ \ x^{\prime }=x\text{ and}\ y\sim y^{\prime }\text{,}\ \
\text{or}\ \ x\sim x^{\prime }\ \text{and}\ \ y=y^{\prime }.
\end{equation*}
\end{definition}

The comparison of Definitions \ref{d2.3} and \ref{d2.3g} yields the
following statement.

\begin{lemma}
\label{l5.3} The functors $\mathcal{O}$ and $\mathcal{O}^{-1}$ preserve the
product $\boxdot $, that is
\begin{equation*}
\mathcal{O}(\mathbf{G}\boxdot \mathbf{H})=G\boxdot H,\ \ \mathcal{O}%
^{-1}(G\boxdot H)=\mathbf{G}\boxdot \mathbf{H}.
\end{equation*}
\end{lemma}

By definition, a \emph{\ line graph} is a graph $\mathbf{J}_{n}=(\mathbf{V},%
\mathbf{E})$ with $\mathbf{V}=\{0,1,\dots ,n\}$ and $\mathbf{E}=\{k\sim
k+1|0\leq k\leq n-1\}$. Let $\mathbf{J}=\{0\sim 1\}$ be the line graph with
two vertices. Let $J_{n}=\mathcal{O}(\mathbf{J}_{n})$ and $J=\mathcal{O}(%
\mathbf{J})$.

\begin{definition}
\label{d5.2} \RM\cite{Barcelo} Let $\mathbf{G},\mathbf{H}$ be two graphs.

$\left( i\right) $ Two graph maps $\mathbf{f},\mathbf{g}\colon \mathbf{G}%
\rightarrow \mathbf{H}\ $ are called \emph{homotopic} if there exists a line
graph $\mathbf{J}_{n}\ (n\geq 0)$ and a graph map $\mathbf{F}\colon \mathbf{G%
}\boxdot \mathbf{J}_{n}\rightarrow \mathbf{H}$ such that
\begin{equation*}
\mathbf{F}|_{\mathbf{G}\boxdot \{0\}}=\mathbf{f}_{0}\text{ and }\ \mathbf{F}%
|_{\mathbf{G}\boxdot \{n\}}=\mathbf{f}_{1}
\end{equation*}%
In this case we shall write $\mathbf{f}\simeq \mathbf{g}$.

$\left( ii\right) $ The graphs $\mathbf{G}$ and $\mathbf{H}$ are\emph{\ }%
called \emph{homotopy equivalent} if there exist graph maps $\mathbf{f}%
\colon \mathbf{G}\rightarrow \mathbf{H}$ and$\ \mathbf{g}\colon \mathbf{H}%
\rightarrow \mathbf{G}$ such that
\begin{equation}
\mathbf{f}\circ \mathbf{g}\simeq \func{id}_{\mathbf{H}},\ \ \ \mathbf{g}%
\circ \mathbf{f}\simeq \func{id}_{\mathbf{G}}.  \label{fgb}
\end{equation}%
In this case we shall write $\mathbf{H}\simeq \mathbf{G}$. The maps $\mathbf{%
f}$ and $\mathbf{g}$ are as in (\ref{fgb}) called \emph{homotopy inverses}
of each other.
\end{definition}

The relation \textquotedblright $\simeq $" is an equivalence relation on the
set of graph maps and on the set of graphs (see \cite{Barcelo}).

\begin{proposition}
\label{t5.4}Let $\mathbf{f},\mathbf{g}\colon \mathbf{G}\rightarrow \mathbf{H}
$ be graph maps. The maps $\mathbf{f}$ and $\mathbf{g}$ are homotopic if and
only if the digraph maps $f=\mathcal{O}(\mathbf{f})$ and $g=\mathcal{O}(%
\mathbf{g})$ are homotopic.
\end{proposition}

\begin{proof}
Let $\mathbf{F}\colon \mathbf{G}\boxdot \mathbf{J}_{n}\rightarrow \mathbf{H}$
be a homotopy between $\mathbf{f}$ and $\mathbf{g}$ as in Definition \ref%
{d5.2}. The natural digraph inclusion $I_{n}\rightarrow J_{n}\ $(where $%
I_{n}\in \mathcal{I}$ is arbitrary) induces the digraph inclusion $\Theta
\colon G\boxdot I_{n}\rightarrow G\boxdot J_{n}$. Applying functor $\mathcal{%
O}$ and Lemma \ref{l5.3} we obtain a digraph map $F\colon =G\boxdot
J_{n}\rightarrow H$ such that the composition $F\circ \Theta \colon G\boxdot
I_{n}\rightarrow H$ provides a digraph homotopy. Now let $F\colon G\boxdot
I_{n}\rightarrow H$ be a digraph homotopy as in Definition \ref{d3.1}
between two double digraphs. Define a digraph map $F^{\prime }\colon
G\boxdot J_{n}\rightarrow H$ on the set of vertices by $F^{\prime
}(x,i)=F(x,i)$. Since $H$ is a double digraph, this definition is correct.
Applying functor $\mathcal{O}^{-1}$ and Lemma \ref{l5.3} we obtain a graph
homotopy $\mathbf{F}^{\prime }\colon \mathbf{G}\boxdot \mathbf{J}%
_{n}\rightarrow \mathbf{H}$.
\end{proof}

Denote by $\mathcal{D}^{\prime }$ the homotopy category of digraphs. The
objects of this category are digraphs, and the maps are classes of homotopic
digraphs maps. Similarly, denote by $\mathcal{G}^{\prime }$ the homotopy
category of graphs and by $\mathcal{O}(\mathcal{G}^{\prime })$ the homotopy
category of double digraphs.

Proposition \ref{t5.4} implies the following.

\begin{corollary}
\label{c5.5} The functors $\mathcal{O}$ and $\mathcal{O}^{-1}$ induce an
equivalence between homotopy category of graphs and homotopy category of
double digraphs
\end{corollary}

\begin{definition}
\label{d5.6} \RM Let $\mathbb{K}$ be a commutative ring with unity. Define
\emph{homology groups} of a graph $\mathbf{G}$ with coefficients in $\mathbb{%
K}$ as follows: $H_{n}(\mathbf{G},\mathbb{K})\colon =H_{n}(G,\mathbb{K})$
where $G=\mathcal{O}\left( \mathbf{G}\right) .$
\end{definition}

The following statement follows from Theorem \ref{t3.4} and Proposition \ref%
{t5.4}.

\begin{proposition}
\label{c5.7} The homology groups of a graph $\mathbf{G}$ with coefficients $%
\mathbb{K}$ are homotopy invariant.
\end{proposition}

A subgraph $\mathbf{H}$ of a graph $\mathbf{G}$ is a graph whose set of
vertices is a subset of that of $\mathbf{G}$ and the edges of $\mathbf{H}$
are all those edges of $\mathbf{G}$ whose adjacent vertices belong to $%
\mathbf{H}$.

\begin{definition}
\label{d5.8} \RM Let $\mathbf{G}$ be a graph and $\mathbf{H}$ be its
subgraph.

$\left( i\right) $ A \emph{retraction} of $\mathbf{G}$ onto $\mathbf{H}$ is
a graph map $\mathbf{r}\colon \mathbf{G}\rightarrow \mathbf{H}$ such that $%
\mathbf{r}|_{\mathbf{H}}=\func{id}_{\mathbf{H}}.$

$\left( ii\right) $ A retraction $\mathbf{r}\colon \mathbf{G}\rightarrow
\mathbf{H}$ is called a \emph{deformation retraction} if $\mathbf{i}\circ
\mathbf{r}\simeq \func{id}_{\mathbf{G}},$ where $\mathbf{i}\colon \mathbf{H}%
\rightarrow \mathbf{G}$ is the natural inclusion map.
\end{definition}

Note that the condition $\mathbf{i}\circ \mathbf{r}\simeq \func{id}_{\mathbf{%
G}}$ is equivalent to the existence of a graph morphism $\mathbf{F}:\mathbf{G%
}\boxdot \mathbf{J}_{n}\rightarrow \mathbf{G}$ such that
\begin{equation}
\mathbf{F}|_{\mathbf{G}\boxdot \{0\}}=\func{id}_{\mathbf{G}},\ \ \mathbf{F}%
|_{\mathbf{G}\boxdot \{n\}}=\mathbf{i}\circ \mathbf{r}.  \label{5.32}
\end{equation}%
Similarly Proposition \ref{p3.6}, a deformation retraction provides homotopy
equivalence $\mathbf{G}\simeq \mathbf{H}$ with homotopy inverse maps $%
\mathbf{i},\mathbf{r}$ (compare with \cite[p.119]{Barcelo}).

\begin{example}
\label{e5.9} \RM$\left( i\right) $ Let us define a\emph{\ cycle} graph $%
\mathbf{S}_{n}$ $(n\geq 3)$ as the graph that is obtained from $\mathbf{J}%
_{n}$ by identifying of the vertices $n$ and $0$. Then
\begin{equation*}
H_{p}(\mathbf{S}_{n},\mathbb{K})=%
\begin{cases}
\mathbb{K}, & \forall n\ \text{and}\ p=0, \\
\mathbb{K}, & n\geq 5\ \text{and}\ p=n, \\
0, & \text{in other cases}.%
\end{cases}%
\end{equation*}

$\left( ii\right) $ Let $\mathbf{G}$ be \emph{a star-like graph}, that there
is a vertex $a\in \mathbf{V}_{\mathbf{G}}$ such that $a\sim v$ for any $v\in
\mathbf{V}_{\mathbf{G}}$. Then the map $\mathbf{r}:\mathbf{G}\rightarrow
\left\{ a\right\} $ is a deformation retraction which implies $\mathbf{G}%
\simeq \left\{ a\right\} $ (cf. Example \ref{e3.11}). Consequently, $H_{0}(%
\mathbf{G},\mathbb{K})=\mathbb{K}$ and $H_{p}(\mathbf{G},\mathbb{K})=0$ for
all $p>0$.

$\left( iii\right) $ If a graph $\mathbf{G}$ is a tree, then $\mathbf{G}$ is
contractible (cf. Example \ref{e8.4}). In particular, $H_{0}(\mathbf{G},%
\mathbb{K})=\mathbb{K}$ and $H_{p}(\mathbf{G},\mathbb{K})=0$ for all $p>0$.
\end{example}

\begin{definition}
\label{d5.10} \RM Let $\mathbf{f}\colon \mathbf{G}\rightarrow \mathbf{H}$ be
a graph map. The \emph{cylinder} $\func{C}_{\mathbf{f}}$ of $\mathbf{f}$ is
a graph with the set of vertices $\mathbf{V}_{\func{C}_{\mathbf{f}}}=\mathbf{%
V}_{\mathbf{G}}\sqcup \mathbf{V}_{\mathbf{H}}$ and with the set of edges $%
\mathbf{E}_{\func{C}_{\mathbf{f}}}$ that consists of all the edges from $%
\mathbf{E}_{\mathbf{G}}$ and $\mathbf{E}_{\mathbf{H}}$ as well as of the
edges of the form $x\sim f\left( x\right) $ for all $x\in \mathbf{V}_{%
\mathbf{G}}$.
\end{definition}

Analogously to Proposition \ref{p3.15}, we obtain the following.

\begin{proposition}
\label{t5.11} We have a homotopy equivalence $\func{C}_{\mathbf{f}}\simeq
\mathbf{H}$.
\end{proposition}

Below we consider based graphs $\mathbf{G}^{\ast }$, where $\ast $ is a
based vertex of $\mathbf{G.}$ The based vertex of $\mathbf{J}_{n}$ will be
usually $0$.\label{here}

\begin{definition}
\label{d5.12} \RM Let $\mathbf{G}$ be a graph. A\emph{\ path-map }in a graph
$\mathbf{G}$ is any digraph map $\Phi :\mathbf{J}_{n}\rightarrow \mathbf{G}.$
A\emph{\ based path }on based graph $\mathbf{G}^{\ast }$ is a based map $%
\Phi :\mathbf{J_{n}^{\ast }}\rightarrow \mathbf{G}^{\ast }$. A loop in $%
\mathbf{G}$ is a based path-map $\Phi :\mathbf{J_{n}^{\ast }}\rightarrow
\mathbf{G}^{\ast }$ such that $\Phi (n)=\ast $.
\end{definition}

The \emph{\ inverse} path-map\emph{\ } and the \emph{concatenation }of
path-maps are defined similarly to Definition \ref{d4.1}.

\begin{definition}
\label{d5.13} \RM$\left( i\right) $ A graph map $\mathbf{h}\colon \mathbf{J}%
_{n}\rightarrow \mathbf{J}_{m}$ is called shrinking if $\mathbf{h}\left(
0\right) =0,$ $\mathbf{h}(n)=m$, and $\mathbf{h}\left( i\right) \leq \mathbf{%
h}\left( j\right) $ whenever $i\leq j$.

\emph{An extension }of a based path-map $\Phi :\mathbf{J}_{m}^{\ast
}\rightarrow \mathbf{G}^{\ast }$ is any path-map $\Phi ^{E}=\Phi \circ
\mathbf{h}\ \ $where $\mathbf{h}\colon \mathbf{J}_{n}^{\ast }\rightarrow
\mathbf{J}_{m}^{\ast }$ is shrinking. An extension $\Phi ^{E}$ is called
\emph{a stabilization} of $\Phi $ if the shrinking map $\mathbf{h}$\textbf{\
}satisfies the condition $\mathbf{h}|_{\mathbf{J}_{m}}=\func{id}$. A
stabilization of $\Phi $ will be denoted by $\Phi ^{S}$.

$\left( ii\right) $ Two loops $\Phi ,\Psi $ in a based graph $\mathbf{G}%
^{\ast }$ are called $S$-\emph{homotopic} if there exist stabilizations $%
\Phi ^{S},\Psi ^{S}$ which are homotopic. In this case we shall write $\Phi
\overset{S}{\simeq }\Psi $. This is an equivalence relation and equivalence
class of a loop $\Phi $ will be denoted by $[\Phi ]$ (cf. \cite{Babson} and
\cite{Barcelo}).
\end{definition}

Define a set $\pi _{1}(\mathbf{G}^{\ast })$ as the set of $S$-equivalence
classes of loops in $\mathbf{G}^{\ast }$, and the product in $\pi _{1}(%
\mathbf{G}^{\ast })$ by $[\Phi ]\cdot \lbrack \Psi ]\colon =[\Phi \vee \Psi
] $. Let $\mathbf{e}\colon \mathbf{J}_{0}^{\ast }\rightarrow \mathbf{G}%
^{\ast } $ be the trivial loop.

\begin{proposition}
\label{p5.14} {\RM\cite{Babson}, \cite[Proposition 5.6]{Barcelo} } The set $%
\pi _{1}(\mathbf{G}^{\ast })$ with the product defined above and with the
neutral element $[\mathbf{e}]$ is a group, that will be referred to as a
fundamental group of the graph $\mathbf{G}^{\ast }$ and denoted by $\pi
_{1}\left( \mathbf{G}^{\ast }\right) $.
\end{proposition}

\begin{definition}
\label{d5.15} \RM Consider two based path-maps
\begin{equation*}
\Phi \colon \mathbf{J}_{n}^{\ast }\rightarrow \mathbf{G}^{\ast }\text{ \ and
\ }\ \Psi \colon \mathbf{J}_{m}^{\ast }\rightarrow \mathbf{G}^{\ast }.
\end{equation*}
An\emph{\ one-step }$C$\emph{-homotopy} from $\Phi $ to $\Psi $ is given by
a shrinking map $\mathbf{h}:\mathbf{J}_{n}\rightarrow \mathbf{J}_{m}$ such
that the map $\mathbf{F}:\mathbf{V}_{\func{C}_{\mathbf{h}}}\rightarrow
\mathbf{V}_{\mathbf{G}}$ given by%
\begin{equation*}
\mathbf{F}|_{\mathbf{J}_{n}}=\Phi \ \ \ \text{and\ \ \ }\mathbf{F}|_{\mathbf{%
J}_{m}}=\Psi ,
\end{equation*}
is a graph map from $\func{C}_{\mathbf{h}}$ to $\mathbf{G}.$

The path-maps $\Phi $ and $\Psi $ are said to be $C$\emph{-homotopic} if
there exists a sequence of one-step $C$-homotopies that connect $\Phi $ and $%
\Psi $. We shall write in this case $\Phi \overset{C}{\simeq }\Psi $.
\end{definition}

The following statement follows immediately from definitions of the functor $%
\mathcal{O}$ and cylinder of the graph and digraph maps.

\begin{lemma}
\label{l5.16} Let $\mathbf{h}\colon \mathbf{G}\rightarrow \mathbf{H}$ be a
graph map. There exists a natural digraph inclusion $\func{C}_{h}\rightarrow
\mathcal{O}(\mathbf{C}_{\mathbf{h}}),$ where $\func{C}_{h}$ is a cylinder of
the digraph map $h\colon G\rightarrow H$.
\end{lemma}

\begin{theorem}
\label{t5.18} Let $G^{\ast }$ be a based double digraph. We have a natural
isomorphism of fundamental groups $\pi _{1}(\mathbf{G}^{\ast })\cong \pi
_{1}(G^{\ast })$ where $\mathbf{G}^{\ast }=\mathcal{O}^{-1}({G}^{\ast })$.
\end{theorem}

\begin{proof}
Let $\Phi \colon \mathbf{J}_{n}^{\ast }\rightarrow \mathbf{G}^{\ast }$ be a
based loop. Denote by $I_{n}^{s}$ the special line digraph with the vertices
$0,1,...,n$ and edges $i\rightarrow i+1$ for all $i=0,...,n-1$. There is a
natural inclusion $\tau \colon I_{n}^{s}\rightarrow J_{n}$. The composition $%
\mathcal{O}(\Phi )\circ \tau \colon {I_{n}^{s}}^{\ast }\rightarrow G^{\ast }$
defines a based loop $\phi $ in $G^{\ast }$. At first we would like to
prove, that the correspondence $\Phi \longrightarrow \mathcal{O}(\Phi )\circ
\tau =\phi $ provides a well defined map of sets
\begin{equation*}
\mathcal{O}_{\ast }\colon \pi _{1}(\mathbf{G}^{\ast })\rightarrow \pi
_{1}(G^{\ast }),\ \ \mathcal{O}_{\ast }([\Phi ])\mapsto \lbrack \mathcal{O}%
(\Phi )\circ \tau ]=[\phi ].
\end{equation*}%
Let $\Phi \colon \mathbf{J}_{k}^{\ast }\rightarrow \mathbf{G}^{\ast },\Psi
\colon \mathbf{J}_{m}^{\ast }\rightarrow \mathbf{G}^{\ast }$ be loops and $%
\Phi \overset{S}{\simeq }\Psi $. The homotopic stabilizations $\Phi ^{S}$
and $\Psi ^{S}$ provide one-step $C$-homotopies $\Phi ^{S}\overset{C}{\simeq
}\Phi $, $\Psi ^{S}\overset{C}{\simeq }\Psi $. A homotopy between $\Phi ^{S}$
and $\Psi ^{S}$ provides $C$-homotopy $\Phi ^{S}\overset{C}{\simeq }\Psi
^{S} $. Since $C$-homotopy is an equivalence relation, we obtain $\Phi
\overset{C}{\simeq }\Psi $. Now by Lemma \ref{l5.16} we obtain that $\phi
\overset{C}{\simeq }\psi $. That is the map $\mathcal{O}_{\ast }$ is well
defined, and it is easy to see that this is a homomorphism of groups. This
is an epimorphism as follows from Proposition \ref{l5.17}.

Digraph maps
\begin{equation*}
\phi \colon {I_{n}^{s}}^{\ast }\rightarrow G^{\ast },\ \ \psi \colon {%
I_{m}^{s}}^{\ast }\rightarrow G^{\ast }
\end{equation*}%
define graphs maps
\begin{equation*}
\Phi \colon \mathbf{J}_{n}\rightarrow \mathbf{G}^{\ast },\ \ \Phi \colon
\mathbf{J}_{m}\rightarrow \mathbf{G}^{\ast },
\end{equation*}
such that $\mathcal{O}(\Phi )\circ \tau =\phi $ and $\mathcal{O}(\Psi )\circ
\tau =\psi $. A one-step $C$-homotopy $\phi \overset{C}{\simeq }\psi $
implies a one-step $C$-homotopy $\Phi \overset{C}{\simeq }\Psi $. That
implies that $\Phi ^{E}$ is homotopic to $\Psi $ or vice versa. To finish
the proof of the Theorem, it suffices to prove that $\Phi \overset{S}{\simeq
}\Psi $. But this follows directly from definition of fundamental group of
graph in \cite{Babson} and \cite{Barcelo}.
\end{proof}

\begin{theorem}
\label{t5.19} For any based graph $\mathbf{G}^*$ we have an isomorphism
\begin{equation*}
\pi_1(\mathbf{G}^*)/[\pi_1(\mathbf{G}^*), \pi_1(\mathbf{G}^*)]\cong H_1(%
\mathbf{G}, \mathbb{Z})
\end{equation*}
where $[\pi_1(\mathbf{G}^*), \pi_1(\mathbf{G}^*)]$ is a commutator subgroup.
\end{theorem}

\begin{proof}
Follows from Theorems \ref{t5.18} and \ref{t4.13}.
\end{proof}

\begin{definition}
\label{d5.20} \RM\cite[p.41]{Babson} Let $\mathbf{G}^{\ast }$ be a based
graph.

$\left( i\right) $ \emph{A based path graph} $\mathbf{PG^{\ast }}$ is a
graph with the set of vertices $\mathbf{V}_{\mathbf{PG^{\ast }}}=\{\Phi
\colon \mathbf{J}_{n}\rightarrow \mathbf{G}^{\ast }\}$, a base vertex $\ast
\colon \mathbf{J}_{0}\rightarrow \mathbf{G}^{\ast }$, and there is an edge $%
\Phi \sim \Psi $ if and only $\Phi ^{S}\simeq \Psi $ or $\Phi \simeq \Psi
^{S}$.

$\left( ii\right) $ \emph{A based loop graph} $\mathbf{LG^{\ast }}$ is a
based sub-graph of $\mathbf{PG^{\ast }}$ with the set of vertices $\mathbf{V}%
_{\mathbf{LG^{\ast }}}=\{\Phi \colon \mathbf{J}_{n}\rightarrow \mathbf{G}%
^{\ast }|\Phi (n)=\ast \}$ and with the restricted from $\mathbf{PG^{\ast }}$
set of vertices.

$\left( iii\right) $ Define \emph{\ higher homotopy groups} $\pi _{n}(%
\mathbf{G}^{\ast })\colon =\pi _{n-1}(\mathbf{LG}^{\ast })$ for $n\geq 2$.
\end{definition}

\begin{proposition}
\label{p5.21} Let $G^*=\mathcal{O}(\mathbf{G}^*)$ be a based double digraph.
Then $LG^*$ be a double digraph and we have a natural inclusion $\mathbf{l}%
\colon \mathbf{LG^*}\overset{\subset}{\to }\mathcal{O}^{-1} (LG^*) $ that is
an identity map on the set of vertices. This map induces a homomorphism of
homotopy groups $\pi_n(\mathbf{LG^*})\to \pi_n(LG^*) $ for $n\geq 1$ and an
isomorphism for $n=0$.
\end{proposition}

\begin{proof}
The proof that $LG^{\ast }$ is a double digraph is similar to the proof of
Proposition \ref{t5.4} and the graph map $\mathbf{l}$ is well defined by
Lemma \ref{l5.16}. Then the result follows.
\end{proof}

\bibliographystyle{amsplain}
\bibliography{biblio}

\end{document}